\documentclass[11pt,reqno]{amsart} 
\usepackage[portrait,margin=3cm]{geometry} 
\usepackage{mathrsfs,xfrac} 
\usepackage[colorlinks=true,linkcolor=blue,citecolor=blue,urlcolor=blue,unicode]{hyperref} 
\usepackage[mathlines]{lineno} %% <- with [mathlines] to number lines in equations
\usepackage{amsmath,amssymb,amsthm,amsfonts,amsbsy,latexsym,dsfont,color,booktabs} 
\usepackage[numeric,initials,nobysame]{amsrefs} 
\usepackage[foot]{amsaddr}
\usepackage{upref,halloweenmath}

\usepackage[utf8]{inputenc}
\usepackage[english]{babel}

\usepackage[textsize=tiny]{todonotes}
\usepackage{regexpatch}
\makeatletter
\xpatchcmd{\@todo}{\setkeys{todonotes}{#1}}{\setkeys{todonotes}{inline,#1}}{}{}
\makeatother

\usepackage{etoolbox}          %% <- for \cspreto, \csappto and \patchcmd
\newcommand*\linenomathpatch[1]{%
  \cspreto{#1}{\linenomath}%
  \cspreto{#1*}{\linenomath}%
  \csappto{end#1}{\endlinenomath}%
  \csappto{end#1*}{\endlinenomath}%
}
\newcommand*\linenomathpatchAMS[1]{%
  \cspreto{#1}{\linenomathAMS}%
  \cspreto{#1*}{\linenomathAMS}%
  \csappto{end#1}{\endlinenomath}%
  \csappto{end#1*}{\endlinenomath}%
}

\expandafter\ifx\linenomath\linenomathWithnumbers
  \let\linenomathAMS\linenomathWithnumbers
  \patchcmd\linenomathAMS{\advance\postdisplaypenalty\linenopenalty}{}{}{}
\else
  \let\linenomathAMS\linenomathNonumbers
\fi

\linenomathpatch{equation}
\linenomathpatchAMS{gather}
\linenomathpatchAMS{multline}
\linenomathpatchAMS{align}
\linenomathpatchAMS{alignat}
\linenomathpatchAMS{flalign}

\makeatletter
\patchcmd{\mmeasure@}{\measuring@true}{
  \measuring@true
  \ifnum-\linenopenaltypar>\interdisplaylinepenalty
    \advance\interdisplaylinepenalty-\linenopenalty
  \fi
  }{}{}
\makeatother

\usepackage{enumerate}
\newenvironment{enumeratei}{\begin{enumerate}[\upshape i.]}{\end{enumerate}}
\newenvironment{enumeratea}{\begin{enumerate}[\upshape a)]}{\end{enumerate}}

\newtheorem{thm}{Theorem}[section]
\newtheorem{lem}[thm]{Lemma}
\newtheorem{cor}[thm]{Corollary}

\newtheorem{defn}[thm]{Definition}
\newtheorem{rem}[thm]{Remark}

\renewcommand{\leq}{\leqslant} 
\renewcommand{\geq}{\geqslant} 
\renewcommand{\le}{\leqslant} 
\renewcommand{\ge}{\geqslant}

\newcommand{\ind}{\mathds{1}}
\newcommand{\eps}{\varepsilon}

\newcommand{\norm}[1]{\left\Vert#1\right\Vert}

\newcommand{\abs}[1]{\left\vert#1\right\vert}

\newcommand{\ie}{\emph{i.e.,}}

\newcommand{\equald}{\stackrel{\mathrm{d}}{=}}

\def\qed{ \hfill $\blacksquare$}  
\let\ga=\alpha \let\gb=\beta \let\gc=\gamma \let\gd=\delta 
    \let\gk=\kappa \let\gl=\lambda        \let\go=\omega     
  
\let\gC=\Gamma \let\gD=\Delta  \let\gL=\Lambda

\newcommand{\cA}{\mathcal{A}}\newcommand{\cB}{\mathcal{B}}\newcommand{\cC}{\mathcal{C}}
\newcommand{\cD}{\mathcal{D}}\newcommand{\cF}{\mathcal{F}}
\newcommand{\cH}{\mathcal{H}}

\newcommand{\cU}{\mathcal{U}}

\newcommand{\vone}{\mathbf{1}}

\newcommand{\vd}{\mathbf{d}}

\newcommand{\vy}{\mathbf{y}} 
\providecommand{\1}{\mathds{1}}
\newcommand{\mv}[1]{\boldsymbol{#1}}\newcommand{\mvzero}{\boldsymbol{0}}

\newcommand{\mvk}{\boldsymbol{k}}

\newcommand{\mvx}{\boldsymbol{x}}\newcommand{\mvy}{\boldsymbol{y}}

\newcommand{\fJ}{\mathfrak{J}}

\newcommand{\dC}{\mathds{C}}

\newcommand{\dN}{\mathds{N}}
\newcommand{\dP}{\mathds{P}}
\newcommand{\dR}{\mathds{R}}
\newcommand{\dT}{\mathds{T}}

\newcommand{\dZ}{\mathds{Z}} 
\DeclareMathOperator{\E}{\mathds{E}}
\DeclareMathOperator{\pr}{\mathds{P}}

\DeclareMathOperator{\argmin}{argmin}

\newcommand{\dlap}{\ensuremath{\Delta_h}}
\newcommand{\ddt}{\ensuremath{\frac{d}{dt}}}
\renewcommand{\i}{{\mathrm{i}\mkern2mu}}
\newcommand{\peq}{\preccurlyeq}
\begin{document}
\title[Phase transition in DNLS]{Phase transition for discrete nonlinear Schr\"odinger equation in three and higher dimensions}
\author[Dey]{Partha S.~Dey}
\author[Kirkpatrick]{Kay Kirkpatrick}
\author[Krishnan]{Kesav Krishnan}
\address{Department of Mathematics, University of Illinois Urbana-Champaign, 1409 W Green Street, Urbana, Illinois 61801}
\email{$\{$psdey, kkirkpat, kesavsk2$\}$@illinois.edu}
\date{\today}
%\subjclass[2020]{Primary:  Secondary: }
%\thanks{}

\begin{abstract}
    We analyze the thermodynamics of the focusing discrete nonlinear Schr\"odinger equation in dimensions $d\ge 3$ with general nonlinearity $p>1$  and under a model with two parameters, representing inverse temperature and strength of the nonlinearity, respectively. We prove the existence of limiting free energy and analyze the phase diagram for general $d,p$. We also prove the existence of a continuous phase transition curve that divides the parametric plane into two regions involving the appearance or non-appearance of solitons. Appropriate upper and lower bounds for the curve are constructed that match the result in~\cite{CK12} in a one-sided asymptotic limit. We also look at the typical behavior of a function chosen from the Gibbs measure for certain parts of the phase diagram. 
\end{abstract}

\maketitle
\setcounter{tocdepth}1\tableofcontents

%%%%%%%%%%%%%%%%%%%%%%%%%%%%%%%%%%%%%%%%%%%%%%%%%%%%%%%%%%%%%%%%%%%%%%
\section{Introduction}\label{sec:intro}
%%%%%%%%%%%%%%%%%%%%%%%%%%%%%%%%%%%%%%%%%%%%%%%%%%%%%%%%%%%%%%%%%%%%%%

Nonlinear Schr\"odinger (NLS) equations have a fundamental physical importance. They arise in descriptions of a multitude of classical and quantum phenomena, examples include  nonlinear optics~\cite{CLS03}, Bose-Einstein condensation~\cite{BK04} and even the complex dynamics of  DNA~\cite{P04}. A close cousin of the NLS, the Gross Pitaevski equation was recently used to describe a theory of dark matter~\cite{KMW}. The focusing NLS is of significant mathematical interest due to the competition between the dispersive character of the linear part of the equation and the nonlinearity. A striking consequence is the formation of solitons, localized solutions preserved in time. Additionaly, the NLS has algebraic structure; it admits a Hamiltonian description and several conserved quantities. In fact, in dimension one and for non-linearity $p=3$, the NLS is completely integrable, \ie\ it can be described in terms of a Lax pair~\cite{LP68}. All this being said, the behavior of the focusing NLS is particularly challenging to understand in higher dimensions and it is this situation that we aim to address. We first discuss the  continuum focusing nonlinear Schr\"odinger Equation (NLS). 

Let $\psi(t,x)$ be a complex-valued function of time $t$ and spatial variable $x\in\dR^{d}$. We say that $\psi: [0,\infty) \times \dR^{d}\to \dC$ satisfies the continuum focusing NLS with power non-linearity $p>1$ if
\begin{align}\label{def:FNLS}
    \i \partial_{t}\psi =-\gD \psi -|\psi|^{p-1}\psi \text{ for all } t,x.
\end{align}
The continuum Hamiltonian functional is given by
\begin{align}\label{def:HAM0}
    H_c(\psi):=\int_{\dR^{d}}|\nabla \psi|^2 dx-\frac2{p+1}\int_{\dR^{d}}|\psi|^{p+1}dx.
\end{align}
Formally~\eqref{def:FNLS} may be rewritten via the variation of the Hamiltonian as
\begin{align*}
    \ddt \psi= \i \frac{\gd}{\gd \psi^{*}}H_c(\psi).
\end{align*}

Given the Hamiltonian structure, it is reasonable to address questions regarding well-posedness and asymptotic behavior via the construction of invariant measures for the flow. This approach has a rich history, and we will survey the results about the existence of solutions and invariant measures in Section~\ref{sec:lit}. There is a significant obstacle to applying the standard method of construction to the continuum equation in three dimensions and higher, which we will also address in Section~\ref{sec:lit}. Essentially, the natural candidate is not normalizable due to spatial regularity issues (see~\cite{LRS88}). One way around this obstacle is to consider a spatial discretization and study the discrete NLS instead (see~\cites{CK12, C14}). For the spatial dimension, we fix an integer $d\ge 3$.
Let $\mathbb{T}_{n}^{d}$ be the $d$-dimensional discrete torus with vertex set indexed by
\begin{align*}
    V=V_n=[n]^{d}:=\{0,2,\ldots,n-1\}^{d}
\end{align*}
of size $N=n^{d}$ and edge set $E=E_n$. We will denote the $d-$dimensional integer lattice as $\dZ^{d}$. We take $h$ to be the spacing between two neighboring vertices in either case. The discrete nearest neighbor Laplacian acting on $\ell^{2}(G)$ with spacing $h$ and $G=\dT^{d}_{n}$ or $\dZ^{d}$ (the case under consideration will always be specified as required) is defined as\begin{align}\label{def:finitelement}
    \dlap \psi_{{\mvx}}:= \frac1{h^2}\sum_{\mvy\sim {\mvx}} (\psi_{{\mvx}}-\psi_{\mvy}),
\end{align}
where $\mvy\sim {\mvx}$ denotes the sum over all nearest neighbors $\mvy$ of ${\mvx}$ and $\psi=(\psi_{{\mvx}})_{{\mvx}\in V} \in \ell^{2}(G)$. We may regard the parameter $1/h$ to be the coupling strength of the lattice. The focusing Discrete Nonlinear Schr\"odinger (DNLS) equation on $G$ with nonlinearity $p$ is defined as coupled system of ODEs with $(\psi_{{\mvx}}(t))_{{\mvx}\in V}, t>0$ satisfying\begin{align}\label{eq:DNLS}
    \i\ddt \psi_{{\mvx}}(t) = -\dlap\psi_{{\mvx}}(t) - |\psi_{{\mvx}}(t)|^{p-1}\psi_{{\mvx}}(t),\qquad {\mvx}\in V, t>0.
\end{align}

Equation~\eqref{eq:DNLS} admits a global solution for $\ell^2$ initial data for both cases of $G$. Like the continuum equation, it may be cast into a Hamiltonian form. The discrete Hamiltonian associated with the focusing DNLS is
\begin{align}\label{def:HAM1}
    \begin{split}
        \cH_{h}(\psi)
        & = h^{d-2}\sum_{({\mvx},{\mvx'})\in E} \abs{\psi_{{\mvx}}-\psi_{{\mvx'}}}^2 - \frac{2}{p+1}\cdot h^d\sum_{x\in G} |\psi_{{\mvx}}|^{p+1}\\
        & = {h^{d-2}}\norm{\nabla\psi}_2^2 - \frac{2}{p+1}\cdot h^d\norm{\psi}_{p+1}^{p+1}.
    \end{split}
\end{align}
Up to scaling by a constant,~\eqref{def:HAM1} is the discrete analog of~\eqref{def:HAM0}. When defined on the discrete torus, the phase space of the equation is isomorphic to $\dC^N$, which crucially is finite-dimensional and has a natural volume form. The volume form is preserved under the flow of the equation via Liouville's theorem. As a consequence of Noether's theorem, the invariance of the Hamiltonian with respect to multiplication of $(\psi_{{\mvx}})_{{\mvx}\in V}$ by a constant phase implies that the $\ell^{2}$ norm is a conserved by the dynamics. We refer to the $\ell^{2}$ norm as the mass.  We immediately obtain an invariant Gibbs probability measure for the dynamics, defined on $\dC^V$ with probability density of $(\psi_{{\mvx}})_{{\mvx}\in V}$ proportional to
\begin{align}\label{def:orig}
    e^{-\gb \cH_{h}(\psi)}\cdot \ind_{\{h^{d}\norm{\psi}^2_2\le B\}}\,\vd\psi
\end{align}
with $\gb,B>0$. For simplicity, we express the volume element as
\begin{align*}
    \vd\psi := \prod_{{\mvx}\in V} d\Re{\psi_{{\mvx}}}\,d\Im{\psi}_{{\mvx}}.
\end{align*}
The relationship between $B$, $\beta$, $h$ and $N$ governs the scaling behavior of this measure. In this article, we will work with the following version of the problem.
\begin{defn}[The model]
    Let $d\ge 3, p>1$ be fixed.
    Given a positive real number $\nu>0$, we consider the Hamiltonian
    \begin{align}\label{def:Hn}
        H_{\nu,N}(\psi):=\norm{\nabla\psi}_2^2 - \left( \frac{\nu}{N}\right)^{(p-1)/2}\cdot\frac2{p+1} \norm{\psi}_{p+1}^{p+1}.
    \end{align}
    for $\psi:\dT^{d}_{n}\to \dC$. With this Hamiltonian, we obtain a Gibbs measure of the form
    \begin{align}\label{equiv}
        \mu_N^{\theta,\nu}(\psi) = \frac1{Z_N(\theta,\nu)} e^{-\theta \cH_{\nu,N}(\psi)}\cdot \ind_{\{\norm{\psi}_2^2\le N\}} \vd\psi
    \end{align}
    where $\theta>0$ is the inverse temperature and
    \begin{align}\label{def:parfn}
        Z_N(\theta,\nu) := \int_{\dC^V} e^{-\theta \cH_{\nu,N}(\psi)}\cdot \ind_{\{\norm{\psi}_2^2\le N\}} \vd\psi
    \end{align}
    is the partition function.
\end{defn}

In terms of the original measure~\eqref{def:orig}, this corresponds to choosing our parameters to obey
\begin{align}\label{eq:param}
    \begin{split}
        \theta             & = \gb B\cdot \frac1{Nh^2} \text{ and }
        \nu = B\cdot h^{-d+4/(p-1)};                            \\
        \text{or, equivalently } \beta & = \theta/\nu \cdot N h^{-(d-2)+4/(p-1)}= \theta/\nu \cdot N h^{-(d-2)(p-p_e)/(p-1)},\\
        \text{ and }
        B &= \nu\cdot h^{d-4/(p-1)} = \nu\cdot h^{d(p-p_m)/(p-1)},
    \end{split}
\end{align}
where $p_{m}:=1+4/d< p_{e}:=1+4/(d-2)$ are the mass critical and energy critical threshold, respectively. When $h$ is of constant order, we get $B=\Theta(1)$ whereas $\gb=\Theta(N)$. However, when $h\to 0$, we have $B\to 0$ or $\infty$ according as $p>p_{m}$ or $p<p_{m}$; and $\gb \ll N$ or $\gg N$ according as $p>p_{e}$ or $p<p_{e}$.
We will elaborate further on the scaling in Section~\ref{sec:heuristic}. It suffices to say here that the dependence in $N$ is chosen such that asymptotically the linear and nonlinear parts of the energy contribute on the same scale.

\section{Main Results}
In order to state our results, we need to introduce two functions, $I$ and $W$ defined on $(0,\infty)$. They correspond to nonlinear and linear contributions to the free energy respectively, as will be explained.  

\begin{defn}[Soliton with given mass]
Recall the discrete Hamiltonian introduced in~\eqref{def:HAM1}, with $h=1$ and defined on $\ell^{2}(\dZ^{d})$. We define the minimum energy at mass $a$ as
    \begin{align}\label{def:Ifn}
        I(a):=\inf_{\norm{\psi}_2^2=a} \cH(\psi).
    \end{align}
\end{defn}

This is the energy of a soliton of mass $a$, when it exists. We will elaborate more on this in Section~\ref{sec:soliton}. A standard result in the theory of DLNS states that the minimizer is attained whenever $I(a)<0$. Moreover, there exists $R_p=R_{p,d}\ge 0$ such that $I(a)=0$ iff $a\le R_{p}$ (see~\cite{WEI99}). It can also be shown that $I(a)\geq 0$ implies that $I(a)=0$. 

As for the function $W$, let $\gD$ denote the graph Laplacian on the discrete torus. Let $\phi^{\mv0}$ denote the constant $\mv1$ vector, which spans $\ker(\gD)$, and $\Delta^{\perp}$ be the restriction of $\Delta$ on $\dC^{N}/\mathrm{span}\{\phi^{\mv0}\}$. We define
\begin{align}\label{def:Kfn}
    K(y)=\lim_{N\to \infty} \frac{1}{N}\log \det (y-\gD^{\perp}),\ y\in [0,\infty).
\end{align}
In Section~\ref{sec:GFF}, we will show the convergence of the limit and give a more useful expression for $K(y)$. 
\begin{defn}[Free field energy with given mass]
    Let $K$ be as in~\eqref{def:Kfn}. We define $W:(0,\infty)\to(0,\infty)$ as  
\begin{align}\label{def:Wfn}
    W(b):=\inf_{y\,:\,K'(y)\le b} (K(y)-yb).
\end{align}
\end{defn}
In Lemma~\ref{lem:propW} we will show that $W$ is a decreasing convex function. Moreover, it is the limiting free energy of the Gaussian Free Field conditioned to have mass $b$ and will be explicitly demonstrated in Section~\ref{sec:GFF}. There is a $d-$dependent constant $C_{d}$ (see~\eqref{def:Cd}) such that $W(b)=K(0)$ for all $b\geq C_{d}$.

Throughout the rest of the paper, we will fix the spatial dimension $d\ge 3$ and unless otherwise specified, the non-linearity will be fixed $p>1$.

%%%%%%%%%%%%%%%%%%%%%%%%%%%%%%%%%%%%%%%%%%%%%%%%%%%%%%%%%%%%%%%%%%%%%%
\subsection{Free energy limit}\label{ssec:free}
%%%%%%%%%%%%%%%%%%%%%%%%%%%%%%%%%%%%%%%%%%%%%%%%%%%%%%%%%%%%%%%%%%%%%%

Let $W(b)$ denote the limiting mean free energy of the Gaussian Free Field conditioned to have mass $b$, as given in equation~\eqref{def:Wfn} and $I(a)$ denote the minimum energy for the Hamiltonian~\eqref{def:HAM1} on $\dZ^d$ at mass $a$, as given in equation~\eqref{def:Ifn}. Our first main result is the following.

\begin{thm}[Convergence of free energy]\label{thm:free}
    Let $Z_N(\theta,\nu)$ be as in~\eqref{def:parfn}. We have,
    \begin{align*}
        \left|\frac1N\log Z_N(\theta,\nu)- F(\theta,\nu)\right| \leq O(N^{-2(d-2)/{3d}}),
    \end{align*}
    where
    \begin{align}\label{def:F}
        F(\theta,\nu):=\log\frac{\pi}{\theta}- \min_{0<a<1} \left( W(\theta(1-a)) + \frac\theta{\nu} I(a\nu)\right).
    \end{align}
\end{thm}

Essentially, Theorem~\ref{thm:free} states that asymptotically the mass of a typical function may be divided into two parts:
\begin{enumerate}
    \item Structured part having mass $ \approx a n$, localized to a region of size $\Theta(1)$, function values of $\psi_{x}$ are of order $\sqrt{N}$ and contributes $\exp(- N\theta \nu^{-1} I( \nu a)+o(N))$ to the free energy. 

    \item Random part having mass  $\approx bN$ with $b\le 1-a$, maximum value of $|\psi_{x}|^2$ is $o(N)$, Gaussian fluctuation dominates the typical behavior and contribution to free energy is  given by the integral
       \begin{align*}
           \int_{\sum_{x}|\psi_{x}|^2\ \approx\ bN} \exp\left(-\theta \norm{\nabla \psi}_2^2\right) \vd\psi 
           &= (1/\theta)^{N}\cdot \int_{\sum_{x}|\psi_{x}|^2\ \approx\ b\theta N} \exp\left(- \norm{\nabla \psi}_2^2\right) \vd\psi\\
           &= (\pi/\theta)^N\cdot \exp(-NW(b\theta)+o(N)).
       \end{align*}
\end{enumerate}

Optimizing over $a,b, a+b\le 1$ should give us the scaling behavior for $Z_N(\theta,\nu)$. We prove that this is indeed the case.

%%%%%%%%%%%%%%%%%%%%%%%%%%%%%%%%%%%%%%%%%%%%%%%%%%%%%%%%%%%%%%%%%%%%%%
\subsection{Phase transition curve}\label{ssec:phase}
%%%%%%%%%%%%%%%%%%%%%%%%%%%%%%%%%%%%%%%%%%%%%%%%%%%%%%%%%%%%%%%%%%%%%%
With the behavior of the free energy established, the question moves towards the behavior of the phases, which we characterize by the mass fraction $a$ allocated to the soliton portion. We will prove in Lemma~\ref{lem:propW} that $W$ is continuous and differentiable. In particular, if the minimizer is attained at $a=a_\star\in (0,1)$, then by differentiability of $W,I$, we get the relation
$W'(\theta(1-a_\star))+I'(a_\star\nu)=0.$ However, we have no explicit formula for $I'$ and only an implicit formula for $W'$, making it difficult to utilize this relation. We will still be able to characterize the phase diagram in Section~\ref{ssec:phase}. 
We define
\begin{align}\label{def:minset}
    \mathscr{M}(\theta,\nu):= \argmin_{0\le a\le 1} \left( W(\theta(1-a)) + \frac\theta{\nu} I(a\nu)\right)
\end{align}
as the set of minimizers for the variation formula. This is a compact set by continuity of $W$ and $I$.  We further define 
\begin{align}
    a_{\star}(\theta,\nu):=\min \mathscr{M}(\theta, \nu)
\end{align}
as the smallest $a$ attaining the global minima for~\eqref{def:F}. We define
\begin{align}
    \mathscr{S}:=\{(\theta,\nu)\in (0,\infty)^2: a_\star(\theta,\nu)>0\}.\label{def:sregion}
\end{align}
as the open region in the $(\theta,\nu)$ phase-space having non-zero solitonic contribution and 
\[
\mathscr{D}:=\text{int}((0,\infty)^{2}\setminus \mathscr{S}),
\]
as the open region in the $(\theta,\nu)$ phase-space having zero solitonic contribution. Note that
\begin{align*}
\mathscr{S} &= \{(\theta,\nu)\in (0,\infty)^2: \min_{0\le a\le 1} \left( (W(\theta(1-a))-W(\theta))/\theta + I(a\nu)/\nu\right)<0\}\\
&= \bigcup_{\eps>0,\ 0< a < 1}\left\{ (\theta,\nu)\in (0,\infty)^2: (W(\theta(1-a))-W(\theta))/\theta + I(a\nu)/\nu <-\eps\right\}
\end{align*}
is an open set by continuity of the map $(\theta,\nu)\mapsto (W(\theta(1-a))-W(\theta))/\theta + I(a\nu)/\nu$ for fixed $a$.
The following result characterizes the phase transition. Define the function
\begin{align}\label{def:zeta}
\xi_{p}(t):=\inf_{0<a<1}\frac{-\ln(1-a)}{a^{(p+1)/2}+t} \text{ for } t\ge 0.
\end{align}
See Figure~\ref{fig:xi0} for a plot of $\xi_{p}(0)$.
\begin{figure}[htbp] 
   \centering
   \includegraphics[width=4in]{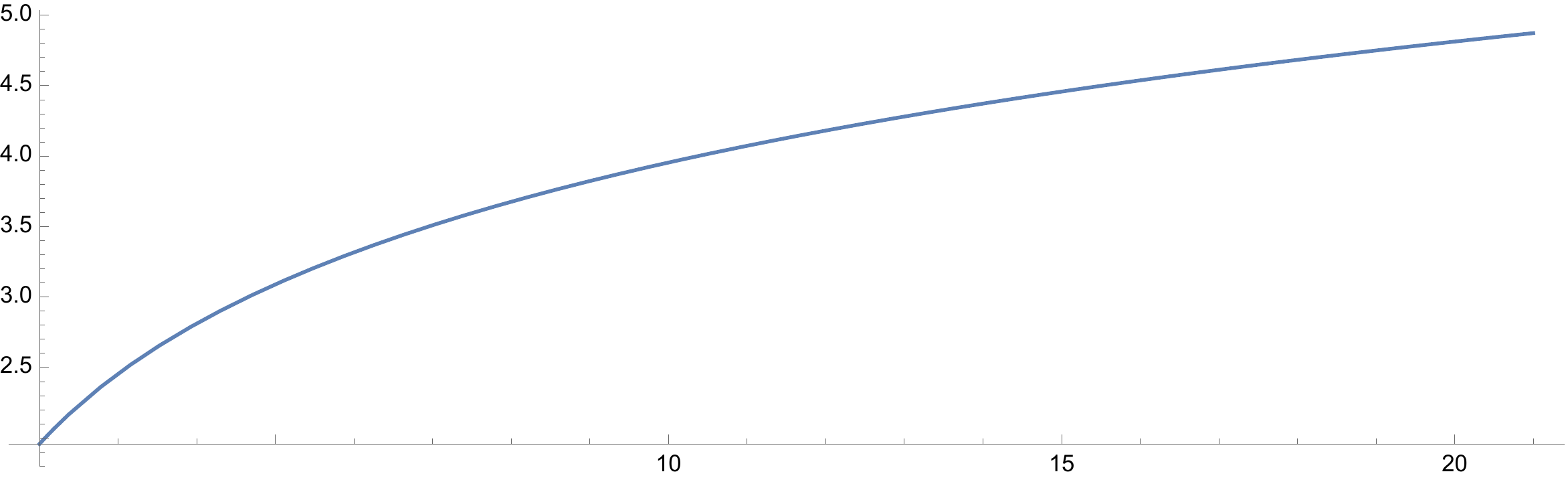} 
   \caption{Plot of $p$ vs.~$\xi_{p}(0)$.}
   \label{fig:xi0}
   % f[x_,p_]:=-(Log[1-x])*x^(-(p+1)/2)
   % g[p_]:=x/.Last[FindMinimum[f[x,p],{x,1-2/(p+1),1}]]
   % Plot[f[g[p],p],{p,2,21},ImageSize->Medium,Background->None]
\end{figure}

\begin{thm}[Existence of Phase transition curve]\label{thm:ptcurve}
We have the following results.
\begin{enumeratea}
    \item For $\nu\le R_{p}$, we have $a_{\star}(\theta,\nu)=0$. 
    \item For $\nu>R_{p}$, there exists a strictly decreasing continuous function $\theta_{c}:(R_{p},\infty)\to (0,\infty)$ such that 
    \[
    a_{\star}(\theta,\nu)\quad
    \begin{cases} 
    =0 &\text{ for }\theta\le \theta_{c}(\nu),\\
    >0 &\text{ for }\theta>\theta_{c}(\nu).
    \end{cases}
    \]
\item The function $\theta_{c}$ is bounded by
\[
\theta_{c}(\nu) \le \min\left\{ \frac{p+1}{2}\nu^{-(p-1)/2}\cdot \xi_{p}(0), \frac{C_{d}\nu}{\nu-R_{p}} \right\}
\]
and satisfies,
\begin{align*}
\lim_{\nu\uparrow\infty}\theta_{c}(\nu)\cdot \frac{2}{p+1}\nu^{(p-1)/2} &= \xi_{p}(0)\text{ and }
\lim_{\nu\downarrow R_{p}} \theta_{c}(\nu)\cdot (\nu-R_{p})= C_{d}R_{p}.
\end{align*}

\item Moreover, we have that for all $\nu>R_{p}$, 
\[
\liminf_{\theta \downarrow \theta_{c}} a_{\star}(\theta,\nu)>0. 
\]
    \end{enumeratea}
\end{thm}

\begin{rem}
We expect $\mathscr{M}$ to be singleton when $a_{\star}>0$. However, to prove that, we need detailed behavior of the $I$ function, especially close to $R_{p}$. We currently lack this. 
\end{rem}

See Figure~\ref{fig:phd} for a pictorial description of the phase diagram.
\begin{figure}[htbp] 
\centering
      \includegraphics[width=3in,page=1]{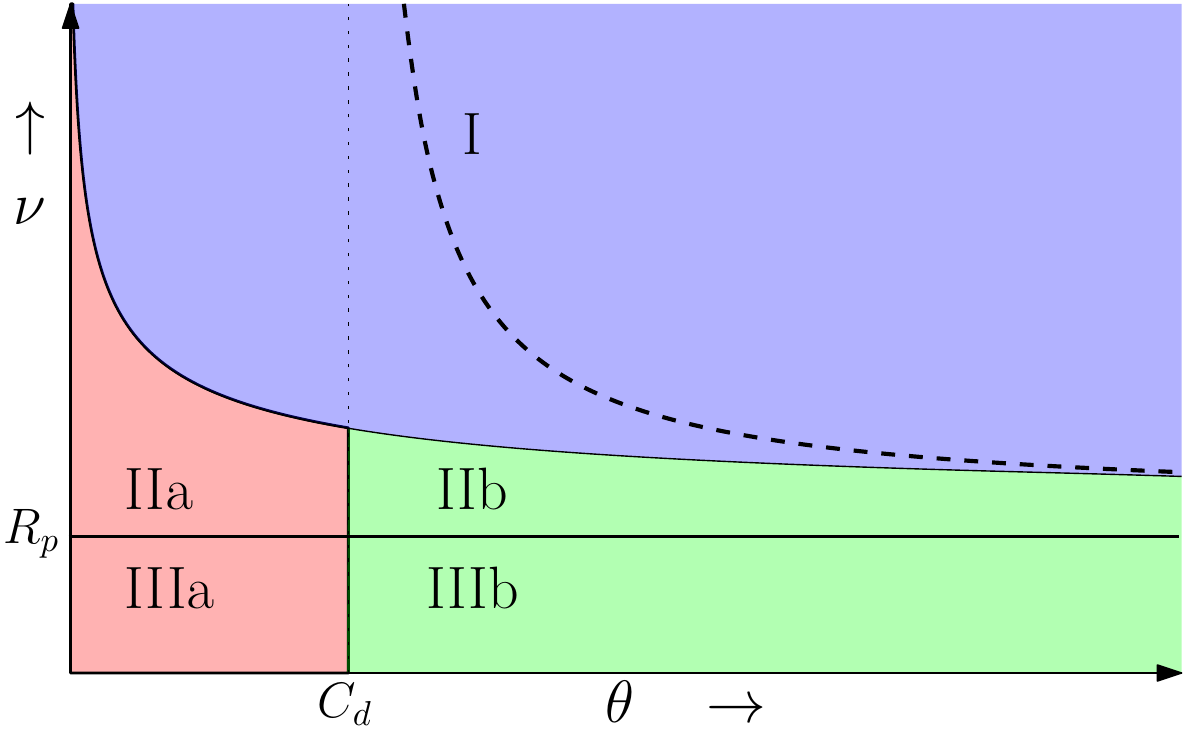}\includegraphics[width=3in,page=2]{phasediag.pdf}\\[2ex]
      \includegraphics[width=1.1in, height=1in, trim={0 0 0in 3in}, clip, page=5]{region.pdf}\hfill
      \includegraphics[width=1.1in, height=1in, trim={0 0 0in 3in}, clip, page=3]{region.pdf}\hfill 
      \includegraphics[width=1.1in, height=1in, trim={0 0 0in 3in}, clip, page=4]{region.pdf}\hfill
      \includegraphics[width=1.1in, height=1in, trim={0 0 0in 3in}, clip, page=1]{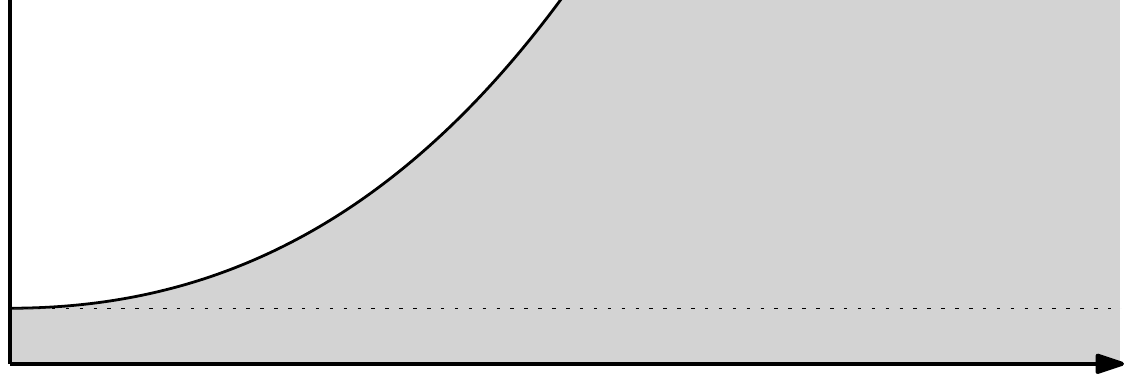}\hfill
      \includegraphics[width=1.1in, height=1in, trim={0 0 0in 3in}, clip, page=2]{region.pdf}\\
   \quad I\qquad\qquad\qquad\quad IIa\qquad\qquad\qquad\quad IIb\qquad\qquad\qquad\quad IIIa\qquad \qquad\qquad IIIb\quad
    \caption{Top: Phase Diagrams for $\{R>0, I'(R+)=0\}$ and $\{ R=0$ or $R>0, I'(R+)<0\}$, respectively. Bottom: Representative plots of $a\mapsto W(\theta(1-a))/\theta+I(a\nu)/\nu$ function for different regions of the phase diagram.}
   \label{fig:phd}
\end{figure}

\begin{rem}\label{rem:inf}
    With the scaling $\theta\nu^{-(p-1)/2}=C$ and $\nu \to \infty$, we recover the regime considered in~\cite{CK12}, the critical value being $C=\xi_{p}(0)$. This is due to the fact that as $\nu$ becomes large, the nonlinear part of the Hamiltonian dominates and the lattice sites decouple.
\end{rem}

\subsection{Typical Dispersive Function}
We now take $(\theta,\nu) \in \mathscr{D}$. In practice, for fixed $\theta$ we will have to take $\nu$ appropriately small (but \textbf{not} vanishing). We introduce the prototypical object to which we will compare a typical function in the dispersive phase. 
\begin{defn}
Let $\{\zeta_{\mvk}\}_{\mvk \in [n]^{d}}$ be i.i.d.~ standard complex Gaussian random variables. Let $\{\phi^{\mvk}\}_{\mvk \in [n]^{d}}$ and $\{\gl_{\mvk}\}_{\mvk \in [n]^{d}}$  respectively be the eigenfunctions~\eqref{eq:efns} and eigenvalues~\eqref{eq:LaplaceValues} of $-\Delta$, as defined on $\ell^{2}(\dT^{d}_{n})$. We define the massive Gaussian free field (MGFF) with mass parameter $y>0$ as 
\begin{align}\label{def:MGFF}
\Psi^{y}:=\sum_{\mvk \in [n]^{d}} \frac{\zeta_{\mvk}}{\sqrt{y+\gl_{\mvk}}}\phi^{\mvk}.
\end{align}
\end{defn}
The properties of the massive free field are well understood. We will recap some of them in Section~\ref{sec:GFF}. 

\begin{thm}\label{thm:typical}
Let $\theta < C_{d}$ be fixed, $p<5+4\sqrt{2}$ and $\nu$ be sufficiently small. Let $y_{N}$ be such that $\E \norm{\Psi^{y_{N}}}_{2}^{2} = N\theta$. Let $\cA_{N}\subset \dC^{N}$ be such that $\dP(\Psi^{y_{N}} \in \cA_{N})\leq N^{-\ga}$ where $\ga >1/2$. Then we have an $\epsilon>0$ such that 
\[
\mu_{N}(\cA_{N})\leq N^{-\epsilon}.
\]

\end{thm}
In the dispersive phase, any event sufficiently rare for the massive GFF with appropriate parameter $\theta$ is also rare for the measure $\mu_{N}$. 
\begin{rem}
We anticipate that the upper bound on $p$ should be removable; it is a consequence of using a particular auxiliary function to aid our proofs. 
\end{rem}
\begin{cor}
For $\Psi$ sampled from $\mu_{N}$, we have with probability approaching $1$
\[
\norm{\Psi}_{\infty} \leq \sqrt{3C_{d} \log N}.
\]
\end{cor}

The penalty factor of $N^{1/2}$ does not arise from the exponential tilting due to the nonlinearity, but rather the severe restriction placed on the allowed value of the mass, as will be discussed in Section~\ref{sec:typical}.

%%%%%%%%%%%%%%%%%%%%%%%%%%%%%%%%%%%%%%%%%%%%%%%%%%%%%%%%%%%%%%%%%%%%%%
\section{Background, Literature Review and Heuristics}\label{sec:lit}
%%%%%%%%%%%%%%%%%%%%%%%%%%%%%%%%%%%%%%%%%%%%%%%%%%%%%%%%%%%%%%%%%%%%%%
The purpose of this section is threefold; provide a very brief survey of some of the important characteristics of the continuum NLS, discuss the obstacles that arise when studying them and finally describe how the discrete NLS captures some of these phenomena without the obstacles. Thus the analysis of the discrete case sheds some light on the continuum. 
\subsection{Criticality} In the continuum focusing NLS, the value of the nonlinearity parameter $p$ can lead to drastically different long-term behavior and also impose different requirements on the regularity and size of the initial data to obtain a well-posed solution. Fundamentally, the issue is that $|\psi|^{p-1}\psi$ need not be in $L^{2}$, and thus can lead to the development of point singularities. The Gagliardo-Nirenberg-Sobolev (GNS) inequality allows us to use the $H^{1}$ norm to bound the $L^{p}$ norm, there is a constant depending on $p$ and $d$ such that
\begin{align}
\norm{\psi}_{p}\leq C_{p,d} \norm{\nabla\psi}_{2}^{(p-1)d/2}\cdot\norm{\psi}_{2}^{2+(2-d)(p-1)/2}.
\end{align}

The regime $p< 1+4/d$ is referred to as mass subcritical, $H^1$ initial data is adequate for global well-posedness. The regime $p=1+4/d$ is referred to as mass critical. In this case, $H^{1}$ data leads to a globally well-posed solution so long as $L^{2}$ norm is smaller than a threshold depending on $p$ and $d$. When $p>1+4/d$, referred to as the mass supercritical regime, we require an upper threshold for both $L^{2}$ and $H^{1}$ norms. 
\subsection{Soliton Solutions} The competition between the dispersion the focusing effect of the nonlinearity yield spatially stationary solutions called solitons. Solton solutions may be realized in one of two ways; either using the separation of variables or through a variational characterization by minimizing the Hamiltonian subject to a mass constraint. If $\varphi({\mvx})$ denotes a soliton solution, it satisfies the following nonlinear elliptic problem for $\go<0$
\[
\go \varphi + \gD \varphi + |\varphi|^{p-1}\varphi=0.
\]

Solitons are strongly localized; it can be proved that they decay exponentially about a center. Further, the variational characterization implies that they are radially symmetric about a center and smooth for energy subcritical nonlinearity. Under their construction, optimizing the difference of $L^{p+1}$ and $H^{1}$ norms, they happen to be the functions for which the GNS inequality is realized, and the constant is sharp~\cite{WEI82}.

\subsection{Use of Invariant Measures} The invariant measure approach to studying the continuum NLS is well-known and celebrated. The hope is that an invariant measure sheds light on the `typical' behavior of the NLS. For instance, this can include questions of well-posedness, as well as questions of types of solutions. Consider the famous \emph{soliton resolution conjecture}, which states that for generic initial data, we see a potion of mass coalescing into a soliton and a portion dispersing away. Invariant measures are a natural means for talking about what constitutes generic initial data.  

In terms of construction, the idea is to use the intuition provided by finite-dimensional Hamiltonian systems to obtain candidate invariant measures on appropriate function spaces. Of course, there is  no version of a Lebesgue measure on a function space to which Liouville's theorem can be applied. However, as is well known in probability, we may rigorously make sense of Gaussian measures which have `density' proportional to $\exp( -\norm{\psi}_{H^{1}})$ on appropriate Wiener spaces. 

For instance, in dimension one with Dirichlet boundary conditions, this is the Brownian bridge on the torus $\dT^{d}$ with harmonic function fixed; this corresponds to the Gaussian Free Field. The regularity of these as distributions is well understood, and the hope is that by exponentially tilting these Gaussian measures by $\norm{\psi}_{p+1}^{p+1}$ and employing a mass cut-off, we may obtain an invariant measure for the dynamics. This requires verifying that the tilt is integrable with respect to the reference Gaussian measure and that the NLS flow is defined on the support of the measure. This technique was used to construct a candidate invariant measure for the 1-d periodic focusing NLS by Lebowitz, Rose, and Speer~\cite{LRS88}. They worked in the subcritical mass regime, where GNS inequality can be applied to control the nonlinearity. Later, McKean and Vaninsky~\cites{MV94,MV97a, MV97b} proved that this measure is indeed invariant for the flow. Bourgain~\cite{B94} used a version of this invariant measure to prove global well-posedness for the periodic equation. Brydges and Slade~\cite{BS96} followed a similar approach for the two-dimensional periodic equation, with a slightly different ultraviolet cutoff. They established a normalizable measure for mass below a critical threshold. However, their measure is not invariant for the flow of the NLS. Indeed, this approach breaks down in $d\geq 3$, where this approach fails to yield even a normalizable measure, and the associated Gaussian field is too rough. It is at this juncture that discretization comes into play. 

\subsection{The Discrete NLS}
We introduced the DNLS in~\eqref{eq:DNLS}, and observed that it retains the Hamiltonian structure, akin to its continuum counterpart. The discrete setting is harder to work with in many ways as several of the symmetries of the continuum NLS are lost, such as Galilean invariance and rotational invariance. On the other hand, we do not have the same regularity issues; the DNLS is globally well-posed for $\ell^{2}$ initial data. Like the continuum equation, the focusing DNLS admits soliton solutions. Like the continuum, they can be realized either through the separation of variables or as minimizers of the Hamiltonian. We discuss them in detail in Section~\ref{sec:soliton}. 

In~\cite{WEI99}, Weinstein studied the discrete focusing NLS on $\dZ^{d}$ and showed that soliton solutions of arbitrary mass could be realized for mass-subcritical nonlinearity. On the other hand, in the mass-supercritical case, there is a constant depending only on the lattice coupling strength and nonlinearity parameter, denoted by $R_{p}$, such that soliton solutions can only be realized when the mass is more than $R_{p}$; a phenomenon strikingly familiar to the blow-up in the continuum. The precise statement is provided later; see Lemma~\ref{lem:Rp}. This analogy is strengthened by the observation that there is a correspondence between solutions of a large mass and solutions on the lattice with low coupling strength; soliton solutions are increasingly concentrated onto a single lattice site as the mass increases. 

In~\cite{CK12}, Chatterjee and Kirkpatrick examined the behavior of the discrete focusing cubic (with $p=3$) NLS defined on the torus of dimension $d\geq 3$, via the analysis of a Gibbs measure of the form~\eqref{def:orig}. Their regime of scaling is chosen to correspond to taking a limit to the continuum; the blow-up phenomenon is realized as a phase transition. They show that a single parameter $\theta=\theta(\beta, B)$ governs the phase behavior. When $\theta \geq \theta_{c}$, a single site acquires a positive fraction of the mass. This is explained by the scaling regime considered; the $H^1$ norm part of the Hamiltonian becomes irrelevant, and the measure may be regarded as an exponential tilt of the uniform measure on the ball via the $\ell^{4}$ norm. The immediate conclusion is that the favored states are those where all the mass is localized to a single site. 

Discrete invariant measures were used to rigorously establish a version of the soliton resolution conjecture by Chatterjee in~\cite{C14}. Chatterjee worked with a microcanonical ensemble, \ie\ the uniform measure defined on an $\eps-$thickening of a $2N-2$ dimensional surface defined by taking constant values of mass and energy and showed that a function uniformly drawn from this measure, modulo translation and phase rotation, converges in a suitable sense to a continuum soliton of the same mass. 

%%%%%%%%%%%%%%%%%%%%%%%%%%%%%%%%%%%%%%%%%%%%%%%%%%%%%%%%%%%%%%%%%%%%%%
\subsection{Scaling commentary}\label{sec:heuristic}
%%%%%%%%%%%%%%%%%%%%%%%%%%%%%%%%%%%%%%%%%%%%%%%%%%%%%%%%%%%%%%%%%%%%%%

There is a natural scale invariance associated with the continuum NLS. If $\psi(t,x)$ is a solution of~\eqref{def:FNLS}, then for any $\gl>0$, $\gl^{\frac2{p-1}} \psi(\gl^2 t,\gl x)$ is also a solution. Since the lattice cannot be scaled, the discrete equation does not admit any symmetry with respect to scaling. However, we do have the following equivalence.
\begin{lem}\label{lem:scaling}
    Let $\psi_{x}(t)$ denote a solution of the discrete NLS on either the lattice $\dZ^{d}$ or the discrete torus $\dT^{d}$ with lattice spacing $h$. Then $\gl^{\frac2{p-1} } \psi_{x}(\gl^2 t)$ is a solution to the discrete NLS corresponding to the same graph with lattice spacing $h/\gl$.
\end{lem}
\begin{proof}
    Proof follows immediately by noting that for $h'=h/\gl, \psi'_{x}(t):=\gl^{\frac2{p-1} } \psi_{x}(\gl^2 t)$ we have $H_{h}(\psi)=\gl^{d-2-\frac{4}{p-1}}H_{h{'}}(\psi{'})$, and using the discrete analogue in~\eqref{eq:DNLS}.
\end{proof}

This yields a family of equivalent ODEs on the lattice, where solutions of one can be scaled into solutions of the other. For the discrete equation, this is the reason why solutions with low values of lattice coupling can be placed in correspondence with solutions of significant mass, as seen in~\cite{WEI99}. In particular, in the model~\eqref{equiv} if we replace $\psi$ by $\gl \psi$ we have equivalence between~\eqref{def:orig} and~\eqref{equiv} as long as the parameters satisfy
\begin{align*}
    \gb h^{d-2}=\theta \gl^2,\quad \gb h^{d}=\theta\gl^{2}\cdot (\nu\gl^{2}/N)^{(p-1)/2}\text{ and } Bh^{-d}=N\gl^{-2}.
\end{align*}
Solving we get the relations~\eqref{eq:param} with $\gl= \sqrt{Nh^d/B}$. 

To use physics terminology, this regime of scaling corresponds to taking the infrared limit, without removing the ultraviolet cutoff, essentially allowing $\dT^{d}_{n}$ to grow to $\dZ^{d}$. As a consequence of this scaling, we may consider the behavior of concentrated and dispersed parts of typical functions separately. Take a function $\psi$ in $\ell^{2}(\dT^{d}_{n})$ with mass bounded above by $N$. We may break it into a region where the values are of order $\sqrt{N}$ and a region where they are of strictly lower order. Let the region of concentration be $U$. The energy of the restriction $\psi_{U}$ is then expressible in terms of the $\dZ^{d}$ Hamiltonian as
\[
\cH_{N}(\psi_{U})\approx \frac{N}{\nu} \cH \left(\sqrt{\frac{\nu}{N}}\psi_{U}\right)
\]
where $\sqrt{\nu/N} \cdot \psi_{U}$ scales to yield a valid function in $\ell^{2}(\dZ^{d})$. 

As for the dispersive part, whenever the order of typical values of $\psi_{x}$ is lower than $N$, the nonlinearity does not contribute, and we see Gaussian Free Field behavior for $\psi_{U^{c}}$. Moreover, the contribution of this portion to the free energy is non-trivial. The analysis of this regime is what makes this article novel; usually the scaling is chosen such that a function sampled from the measure converges to a soliton, our work strongly suggests (yet falls short of explicitly characterizing) the behavior of the \textit{fluctuations} about this soliton, in the vein of studying fluctuations for various random surface models. Indeed, this is the case for our analysis of the typical function in the dispersive phase. The fact that there is no soliton essentially corresponds to the fact that no centering is required; we only see the fluctuations.

%%%%%%%%%%%%%%%%%%%%%%%%%%%%%%%%%%%%%%%%%%%%%%%%%%%%%%%%%%%%%%%%%%%%%%
\subsection{Notations}\label{ssec:not}
%%%%%%%%%%%%%%%%%%%%%%%%%%%%%%%%%%%%%%%%%%%%%%%%%%%%%%%%%%%%%%%%%%%%%%
The following are fixed for the entire article
\begin{enumeratei}
    \item We will denote the integer lattice by $\dZ^{d}$ and the discrete torus of side length $n$ by $\dT^{d}_{n}$.
    \item Vertices in $\dT^{d}_{n}$ or $\dZ^{d}$ will be denoted by ${\mvx}$.
    \item The dual variable to ${\mvx}$ in the sense of the Fourier transform defined on $\dT_{n}^{d}$ will be denoted by $\mvk$. 
    \item $N$ will always denote $n^{d}$.
    \item $\dC^{N}$, as is standard will denote the complex vector space of dimension $N$, with the standard inner product.
    \item $\theta$ and $\nu$ will be positive real numbers denoting the inverse temperature and the coupling constant, respectively.
    \item $\mathscr{S}$ and $\mathscr{D}$ are subsets of $[0,\infty)^{2}$ denoting the solitonic and dispersive regions of the parametric plane $(\theta,\nu)$. 
    \item $\mathscr{M}(\theta,\nu)$ is the collection of optimizers of the variational formula defining the free energy. 
\end{enumeratei}

As best as possible, we work with the following conventions. We also highlight frequently recurring examples. 
\begin{enumeratei}
    \item Subsets of $\dZ^{d}$ or $\dT^{d}$ will be denoted by capitalized Roman letters such as $U$ or $V$. 
    \item Subsets of the function spaces $\ell^{2}(\dT^{d})$ will be denoted by calligraphic letters such as $\cA$ or $\cB$. 
    \item Functions in $\ell^{2}(\dZ^{d})$ or $\ell^{2}(\dT^{d})$ will be denoted by small Greek letters such as $\psi$ or $\phi$, with the following recurring examples. 
    \begin{enumerate}
        \item Restrictions of a function $\psi$ to a subset $U$ will be denoted as $\psi_{U}$.
        \item Discrete solitons of mass $a$ will be denoted as $\varphi^{a}$.
        \item Eigenfunctions of the $\dT^{d}_{n}$ Laplacian will be denoted as $\phi^{\mvk}$ with $\mvk \in [n]^{d}$.
        \item Functions in the orthogonal complement of the subspace spanned by $\{\phi^{\mv0}\}$ will be denoted by $\psi^{\perp}$. 
    \end{enumerate}
    \item Random fields will be denoted by capital Greek letters such as $\Psi$ and $\Phi$, with the following recurring examples,
    \begin{enumerate}
        \item $\Psi^{U,y}$ will denote the massive Dirichlet Gaussian Free Field taking values in $\ell^{2}(U)$.
        \item $\Psi^{\mv0,y}$ will denote the massive zero average Gaussian Free Field taking values in $\ell^{2}(\dT^{d}_{n})$.
        \item The superscript $y$ will be dropped when we have $y=0$, when appropriate. 
    \end{enumerate}
    \item The Laplacians under consideration will be variations of $\gD$. 
    \begin{enumerate}
        \item Restriction to the complement of the kernel will be denoted as $\gD^{\perp}$.
        \item Dirichlet Laplacians on $U$ will be denoted as $\gD^{0}_{U}$. 
    \end{enumerate}
    \item Constants will usually be denoted by variations on the letter $C$. The most frequently recurring standard constant is the following. 
    \begin{enumerate}
        \item $C_{d}$ will denote the limiting mass per site of $\Psi^{\mv0}$.

    \end{enumerate}
\end{enumeratei}

Important exceptions to the conventions listed above are unavoidable for various reasons, such as consistency with prior literature. We list them below.

\begin{enumeratei}
    \item The Hamiltonians will always be denoted with variations of $\cH$. There are three cases
    \begin{enumerate}
        \item $\cH_{c}$ is the continuum Hamiltonian and will not be refrerred to beyond the background material. 
        \item $\cH$ is the discrete Hamiltonian for the DNLS defined on $\dZ^{d}$, with $h=1$. See~\eqref{def:HAM1}. 
        \item $\cH_{N}$ is our scale-dependent model Hamiltonian with which we define the Gibbs measure of interest. 
    \end{enumerate}
    \item The following capital Roman Letters have specific meanings, and are \textbf{not} subsets of lattice sites. 
    \begin{enumerate}
        \item The function $I(a)$ for $a\geq 0$ will always denote the minimum $\cH$ for fixed mass $a$, $\cH$ as defined in~\eqref{def:HAM1}. See~\eqref{def:Ifn}.
        \item The function $K(y)$ will always denote the limiting scaled log determinant of $y-\gD^{\perp}$. See~\eqref{def:Kfn}.
        \item The function $W(b)$ for $b>0$ will always denote the Legendre transform of $K$. See~\eqref{def:Wfn}.
        \item The function $L$ will always denote the inverse of $K'$. See~\eqref{def:L}
        \item $R_{p}$ will always denote the mass threshold for soliton formation. See Lemma~\ref{lem:Rp}
    \end{enumerate}
    \item Parameters $\theta>0$ and $\nu>0$  are the inverse temperature and coupling constant for the nonlinearity in~\eqref{def:parfn}, and are \textbf{not} functions in $\ell^{2}$. 
\end{enumeratei}

Along the way, we will define certain auxiliary functions and random variables to make calculations more convenient to express. These will be defined in a context-appropriate fashion. 

%%%%%%%%%%%%%%%%%%%%%%%%%%%%%%%%%%%%%%%%%%%%%%%%%%%%%%%%%%%%%%%%%%%%%%
\subsection{Organization of the Paper}\label{ssec:org}
%%%%%%%%%%%%%%%%%%%%%%%%%%%%%%%%%%%%%%%%%%%%%%%%%%%%%%%%%%%%%%%%%%%%%%
The article is organized as follows. In Section~\ref{sec:GFF}, we provide a description of the Discrete Gaussian Free Field and explain its importance for our analysis. We then  prove  the convergence of its limiting free energy, and of that conditioned to have a specified mass. We also establish some useful $\ell^{\infty}$ bounds. In Section~\ref{sec:soliton}, we discuss soliton solutions of the DNLS. In particular, we construct exponentially decaying minimizers for~\eqref{def:HAM1}. We also establish properties of  the function $I$. In Section~\ref{sec:feconv}, we combine insights from Sections~\ref{sec:GFF} and~\ref{sec:soliton} to prove the convergence of the limiting free energy, that is Theorem~\ref{thm:free}. In Section~\ref{sec:phase}, we analyze the phases. In particular, we demonstrate that there are two regimes of optimal mass allocation, one where we have a non-trivial soliton, and one where we do not. That is, we prove Theorem~\ref{thm:ptcurve}. In Section~\ref{sec:typical}, we provide commentary on the behavior of a typical function in the dispersive phase. We provide an explicit comparison between the reference measure corresponding to the linear part of the Hamiltonian and the massive Gaussian Free Field. We then show that the tilt corresponding to the nonlinearity is integrable with respect to the massive free field. A combination of these two results verifies Theorem~\ref{thm:typical}. As an immediate corollary, we have that that a typical function in the dispersive phase is bounded above in probability by $\sqrt{3C_{d} \log N}$. We conclude the article with some interesting questions for the future. 
%%%%%%%%%%%%%%%%%%%%%%%%%%%%%%%%%%%%%%%%%%%%%%%%%%%%%%%%%%%%%%%%%%%%%%

%%%%%%%%%%%%%%%%%%%%%%%%%%%%%%%%%%%%%%%%%%%%%%%%%%%%%%%%%%%%%%%%%%%%%%
\section{Gaussian Free Field}\label{sec:GFF}
%%%%%%%%%%%%%%%%%%%%%%%%%%%%%%%%%%%%%%%%%%%%%%%%%%%%%%%%%%%%%%%%%%%%%%

The dispersive contribution to the free energy is given by the integral
\begin{align}
    M_{N}(b,\eps):=\int_{\norm{\psi}^2_2\in N(b-\eps, b+\eps)}\exp \bigl(-\norm{\nabla \psi}_2^2 \bigr)\,\vd \psi
\end{align}
where $b\ge 0,\eps>0$. Understanding the asymptotics as $N \to \infty$ is thus of fundamental importance to this article. This section is first and foremost dedicated to establishing the following theorem.
\begin{thm}\label{thm:DISP}
    Let $b>0$, and $\eps>0$ be a positive number such that $N\eps^d\gg 1$. We have
    \begin{align*}
        \left|\frac1N \log M_{N}(b,\eps)-\pi +W(b)\right|
        \leq  2\eps L(b)+ 2/(N^2\eps^2\cdot m_N(2))+\frac1N \log ((b+\eps)N).
    \end{align*}
    where $m_{N}(\cdot)$ is as per Lemma \ref{lem:epower}. 
\end{thm}
What prevents the immediate representation of the integral as an expectation with respect to a Gaussian random variable is the fact that the quadratic form in the exponential is degenerate; it corresponds to $-\gD$ which has a non-trivial kernel. However, we can still relate this integral to the large deviations of the mass of an appropriate Gaussian field called the zero--average Gaussian Free Field (GFF). We will define the GFF and evaluate the asymptotic via a combination of two probabilistic techniques, exponential tilting, and concentration. We will then conclude the section with some maximum estimates will be of importance later. 

%%%%%%%%%%%%%%%%%%%%%%%%%%%%%%%%%%%%%%%%%%%%%%%%%%%%%%%%%%%%%%%%%%%%%%
\subsection{Definitions}\label{sec:def}
%%%%%%%%%%%%%%%%%%%%%%%%%%%%%%%%%%%%%%%%%%%%%%%%%%%%%%%%%%%%%%%%%%%%%%
The Laplacian $\gD$ is a translation-invariant operator on $\dC^{V}$ with respect to the standard basis and is thus diagonalized by the Fourier basis. It is well-known that the eigenvalues of the Laplacian on the discrete torus with vertex set $[n]^{d}$ are given by
\begin{align}\label{eq:LaplaceValues}
    \gl_{\mvk}=4\sum_{i=1}^{d}\sin^2(\pi k_{i}/n)=f(\mvk/n) \text{ for } \mvk\in [n]^{d}.
\end{align}
The corresponding eigenfunctions are
\begin{align}\label{eq:efns}
    \phi^{\mvk}_{\mvx}=\frac{1}{\sqrt{N}}\exp\left (2\pi\i\cdot {(\mvk \cdot {\mvx})}/n\right),\qquad {\mvx}, \mvk\in [n]^{d}.
\end{align}

\begin{defn}[Massive zero--average GFF]\label{def:0gff}
    Let $y\geq 0$. We define the massive zero--average free field $\Psi^{\mv0,y}\in \dC^{V}$ as the random vector
    \[
        \Psi^{\mv0,y} :=\sum_{\mvk\neq \mv0}\frac{\zeta_{\mvk}}{\sqrt{\gl_{\mvk}+y}}\phi^{\mvk},
    \]
    where $\gl_{\mvk}$ are the eigenvalues of $-\gD$ as in~\eqref{eq:LaplaceValues}, $\phi^{\mvk}$ are the corresponding eigenfunctions as in~\eqref{eq:efns} and $\zeta_{\mvk}$ are i.i.d.~standard complex Gaussian random variables.
\end{defn}
We may explicitly write down the density of the zero--average free field, which is expressible as a Gibbs measure in its own right. We will be working with the subspace $\dC^{N}/\{\phi^{\mv0}\}$, that is the subspace of all $\psi^{\perp}$ that are orthogonal to $\phi^{\mv0}$.
Note, the restriction is negative definite and is therefore invertible. We have 
\[
\dP(\Psi^{\mv0,y}\in \cA):=\frac{1}{Z^{\mv0,y}}\int_{\cA}\exp \left(-\norm{\nabla \psi^{\perp}}_{2}^{2} -y\norm{\psi^{\perp}}_{2}^{2}\right)\vd \psi^{\perp}. 
\]
where $\vd \psi^{\perp}$ denotes the volume element on $\dC/\{\phi^{\mv0}\}$.  We will denote the restriction of the Laplacian on this space as $\gD^{\perp}$. The partition function is given by
\[
Z^{\mv0,y} := \pi^{N-1}/\det(y-\gD^{\perp}).
\]
We refer the interested reader to~\cite{S07} and~\cite{A19} for more details on GFF and zero--average GFF on the discrete torus, respectively.
\begin{rem}
The operator $y-\gD$ for $y>0$ is positive definite, therefore we may define a Gaussian process with covariance $(y-\gD)^{-1}$ without the restriction to $\dC^{N}/\mathrm{span}\{\phi^{\mv0}\}$. This is exactly the massive Gaussian Free Field, seen in~\eqref{def:MGFF}. 
\end{rem}

%%%%%%%%%%%%%%%%%%%%%%%%%%%%%%%%%%%%%%%%%%%%%%%%%%%%%%%%%%%%%%%%%%%%%%
\subsection{Analysis of the Limiting Free Energy}\label{sec:freeenergy}
%%%%%%%%%%%%%%%%%%%%%%%%%%%%%%%%%%%%%%%%%%%%%%%%%%%%%%%%%%%%%%%%%%%%%%
From here onwards, the graph under consideration will be the entirety of the discrete torus $\dT^{d}$, and we will describe the properties of the measure associated with the field $\Psi^{\mv0,y}$. The associated mean free energy is given by
\begin{align*}
    \frac1N\log Z^{\mvzero,y}=\frac{N-1}{N}\log \pi - \frac1N\sum_{\mvk \neq \mv0}\log(\gl_{\mvk}+y).
\end{align*}
In order to more compactly express our results, for ${\mvx}\in [0,1]^{d}$, we define
\begin{align}\label{def:f}
    f({\mvx}):= 4\sum_{i=1}^{d}\sin^2(\pi x_{i}),\quad {\mvx}=(x_1,x_2,\ldots,x_d)\in[0,1]^d
\end{align}
and 
\begin{align*}
    g_y({\mvx}):=\log\left(y+f({\mvx})\right).
\end{align*}
Let ${\mvx}$ denote a uniform random variable on $[0,1]^{d}$. We define ${\mvx}_n:=\lfloor n{\mvx}\rfloor/n\sim \text{Uniform}(\{0,1/n,\ldots,1-1/n\}^d)$ and
\begin{align*}
    K_N(y):=\frac1N\sum_{\mvk \neq \mv0}\log (y+\gl_{\mvk}) = \E g_y({\mvx}_n)\ind_{{\mvx}_n\neq\mvzero}.
\end{align*}
Observe that the expected mass of $\Psi^{\mv0,y}$ is given by
\[
    \E \norm{\Psi^{\mv0,y}}_{2}^{2}=-\frac{d}{dy}\log Z^{\mv0,y}=N\cdot K_{N}'(y).
\]
Recall the function $K$ introduced in~\eqref{def:Kfn}. By definition, $K=\lim_{n\to \infty}K_{N}$. Clearly we should have
\begin{align}\label{eq:Kfnalt}
    K(y)=\int_{[0,1]^{d}}g_y({\mvx})\,d{\mvx}.
\end{align}
We will prove this convergence now, and as an abuse of notation, we use~\eqref{eq:Kfnalt} as the definition of $K$. We refer to the following lemma from~\cite{DK21}, which is important for establishing rates of convergence and follows quite easily.
\begin{lem}[\cite{DK21}*{Lemma~$2.1$}]\label{lem:epower}
    Let $\{\gl_{\mvk}\}_{\mvk \neq 0}$ be the eigenvalues of $-\gD^{\perp}$. Then we have for any $p>0$,
    \[
        m_{N}(p):=\frac1N \sum_{\mvk \neq \mv0} \gl_{\mvk}^{-p}\simeq
        \begin{cases}
            1           & \text{ if } d>2p    \\
            \log N      & \text{ if } d=2p    \\
            N^{-1+2p/d} & \text{ otherwise}.
        \end{cases}
    \]
\end{lem}
This lemma is first of use in order to calculate the rate of convergence of the scaled expected mass, i.e. $\E \norm{\Psi^{0,y}}_{2}^{2}$.
\begin{lem}\label{lem:massconvergerate}
    For $y\geq 0$, we have a constant $c$, depending only on $d\ge 3$, such that
    \begin{align*}
        |K_N(y)-K(y)|\leq c\cdot N^{-1/d} \text{ and }
        |K'_N(y)-K'(y)|\leq c\cdot N^{-1/d}\cdot (1+\vone_{d=3}\cdot \log N).
    \end{align*}
\end{lem}
\begin{proof}
    Let
    \[
        g_y({\mvx}):= \log(y+f({\mvx})) \text{ and } g'_y({\mvx}):=\partial_y g_y({\mvx})=(y+f({\mvx}))^{-1},\quad {\mvx}\in[0,1]^d.
    \]
    It is clear that
    \begin{align*}
        K_N(y)&=\E g_y({\mvx}_n)\ind_{{\mvx}_n\neq\mv0},\text{ } K(y)=\E g_y({\mvx}),\\
        K'_N(y) &=\E g'_y({\mvx}_n)\ind_{{\mvx}_n\neq\mv0} \text{ and } K'(y)=\E g'_y({\mvx}).
    \end{align*}
    Thus, we have, for any $y\ge 0$, we have
    \begin{align*}
    \abs{K(y)-K_N(y)}
         & \le \E |g_y({\mvx})|\ind_{{\mvx}_n=\mv0} +\E\abs{g_y({\mvx})- g_y({\mvx}_n)}\ind_{{\mvx}_n\neq\mv0}\\\text{ and }
         \abs{K'(y)-K'_N(y)}
         & \le \E g'_y({\mvx})\ind_{{\mvx}_n=\mv0} +\E\abs{g'_y({\mvx})- g'_y({\mvx}_n)}\ind_{{\mvx}_n\neq\mv0}.
    \end{align*}
    It is easy to chek that, 
    \[
    \E |g_y({\mvx})|\ind_{{\mvx}_n=\mv0} \le \E |\log(y+f({\mvx})|\ind_{{\mvx}_n=\mv0}\le (1+|\log y|)\cdot N^{-1} 
    \] 
    and 
    \[
    \E g'_y({\mvx})\ind_{{\mvx}_n=\mv0} \le \E f({\mvx})^{-1}\ind_{{\mvx}_n=\mv0}\le cN^{2/d-1}\le cN^{-1/d}.
    \]
    Note that with $y>0$ fixed, $g_y(\cdot)$ is smooth, with bounded derivatives of all orders. Moreover, when $\norm{{\mvx}_n}>0$, we can bound 
    \begin{align*}
    \abs{g_y({\mvx})- g_y({\mvx}_n)} &\le \norm{{\mvx}-{\mvx}_n}\cdot \norm{\nabla g_y({\mvx}_n^*)}\le cN^{-1/d}\cdot f({\mvx}_n)^{-1/2}\\\text{ and } 
    \abs{g'_y({\mvx})- g'_y({\mvx}_n)} &\le \norm{{\mvx}-{\mvx}_n}\cdot \norm{\nabla g'_y({\mvx}_n^*)}\le cN^{-1/d}\cdot f({\mvx}_n)^{-3/2}.
    \end{align*}
    Here we used the fact that 
    \[
    \norm{\nabla f({\mvx})}^{2} = \sum_{i=1}^{d}(8\sin(\pi x_{i})\cos(\pi x_{i}))^{2} \le 16 f({\mvx}).
    \]
    Summing, we get that
    \[
    \abs{K(y)-K_n(y)}
    \le cN^{-1/d} m_N(1/2) \text{ and }
    \abs{K'(y)-K'_N(y)}
    \le cN^{-1/d} m_N(3/2),
    \]
    where $m_n$ is as given in Lemma~\ref{lem:epower}. This completes the proof.
\end{proof}
We now take the opportunity to introduce an important dimension dependent constant. We define
\begin{align}\label{def:Cd}
C_{d}:=K'(0),
\end{align}
which is clearly finite for $d\geq 3$. 
\begin{rem}
The constant $C_{d}$ has an important interpretation in probability, it is the expected number of returns to ${\mvx}$ for a simple symmetric random walk started at ${\mvx}$ in $\dZ^{d}$. Clearly, $C_{d}<\infty$ when $d\geq 3$ and is infinite otherwise due to the recurrence of the random walk. 
\end{rem}

It is easy to check that $2dC_d\ge 2d/\E(f({\mvx}))=1$ for all $d\ge 3$ and converges to $1$ as $d\to\infty$. In Table~\ref{table:Cd} we provide the numerical values of $C_d$ for $d=3,4,\cdots,10$.
\begin{table}[hbtp]
    \centering
    \begin{tabular}{ccccccccc}
        \toprule
        $d$  & 3   & 4   & 5   & 6   & 6   & 8   & 9   & 10 \\
        \midrule
        $C_d$ & 0.252 & 0.155 & 0.116 & 0.093 & 0.078 & 0.067 & 0.059 & 0.053\\
        \bottomrule                              \\[-1ex]
    \end{tabular}
    \caption{Numerical values for $C_d$}
    \label{table:Cd}
\end{table}

Since $K'$ is decreasing and convex, we may define an inverse function
\begin{align}\label{def:L}
    L:(0,C_{d}]\to [0,\infty)
\end{align}
which is also decreasing, convex and by hypothesis satisfies $K'(L(b))=b$. We extend $L$ by defining $L(b)=0$ for $b>C_{d}$. Note that, $bL(b)\le 1$ for all $b>0$ and
\begin{align}
    W(b)=\inf_{y\,:\,K'(y)\le b} (K(y)-yb) =
    \begin{cases}
        K(L(b)) - bL(b) & \text{ if } b<C_{d}  \\
        K(0)      & \text{ if } b\ge C_d.
    \end{cases}
\end{align}
The function $W$ is singular at $0$, however the divergence can be well understood.

\begin{lem}\label{lem:propW}
    The function $W$ given in~\eqref{def:Wfn} is a decreasing convex function of $b$ with $W'(b)=-L(b)$ and $\lim_{b\to 0+}(W(b)+\log eb)=0$. Moreover, $\widehat{W}$ defined by
    \begin{align}\label{eq:what}
        \widehat{W}(b) :=W(b)+\log eb =\int_{0}^{b}(s^{-1}-L(s))\,ds,
    \end{align}
    is an increasing concave function for $b\in (0,C_d]$.
\end{lem}

\begin{proof}[Proof of Lemma~\ref{lem:propW}]
    The function $W$ is decreasing and convex follows from the fact that $W'(b)=-L(b)\le 0$ and $W''(b)=-L'(b)=-1/K'(L(b))=1/\E(L(b)+f({\mvx}))^{-2}>0$.

    Note that
    \[
    b=K'(L(b))=\E(L(b)+f({\mvx}))^{-1}\ge (L(b) +\E f({\mvx}))^{-1}=(L(b)+2d)^{-1}
    \]
    implies that $1/b-L(b)\le 2d$. Moreover, with $y=L(b)>0$, we have
    \begin{align*}
        yb=yK'(y)
         & = 1-\E f({\mvx})(y+f({\mvx}))^{-1}\\
         & \le 1- (y\E f({\mvx})^{-1} +1)^{-1}=1-(C_{d}y+1)^{-1}
        = y(y+1/C_{d})^{-1}.
    \end{align*}
    Simplifying we get
    \[
        1/b-L(b)\ge 1/C_{d}.
    \]

    The last conclusion in Lemma~\ref{lem:propW} follows from the fact that $\widehat{W}'(b) = 1/b - L(b)>0$ and
    \begin{align*}
        \widehat{W}''(b) & = -b^{-2}-L'(b)                                                                              \\
                         & = b^{-2}L'(b) \cdot ( -1/L'(b) - b^2) = b^{-2}L'(b) \cdot (\E(L(b)+f({\mvx}))^{-2} - b^2)\le 0
    \end{align*}
    as $\E(L(b)+f({\mvx}))^{-2} > (\E(L(b)+f({\mvx}))^{-1})^2 = b^2$ and $L'(b)\le 0$.
\end{proof}

\begin{figure}[htbp]
    \centering
    \includegraphics[height=1.5in,width=4in]{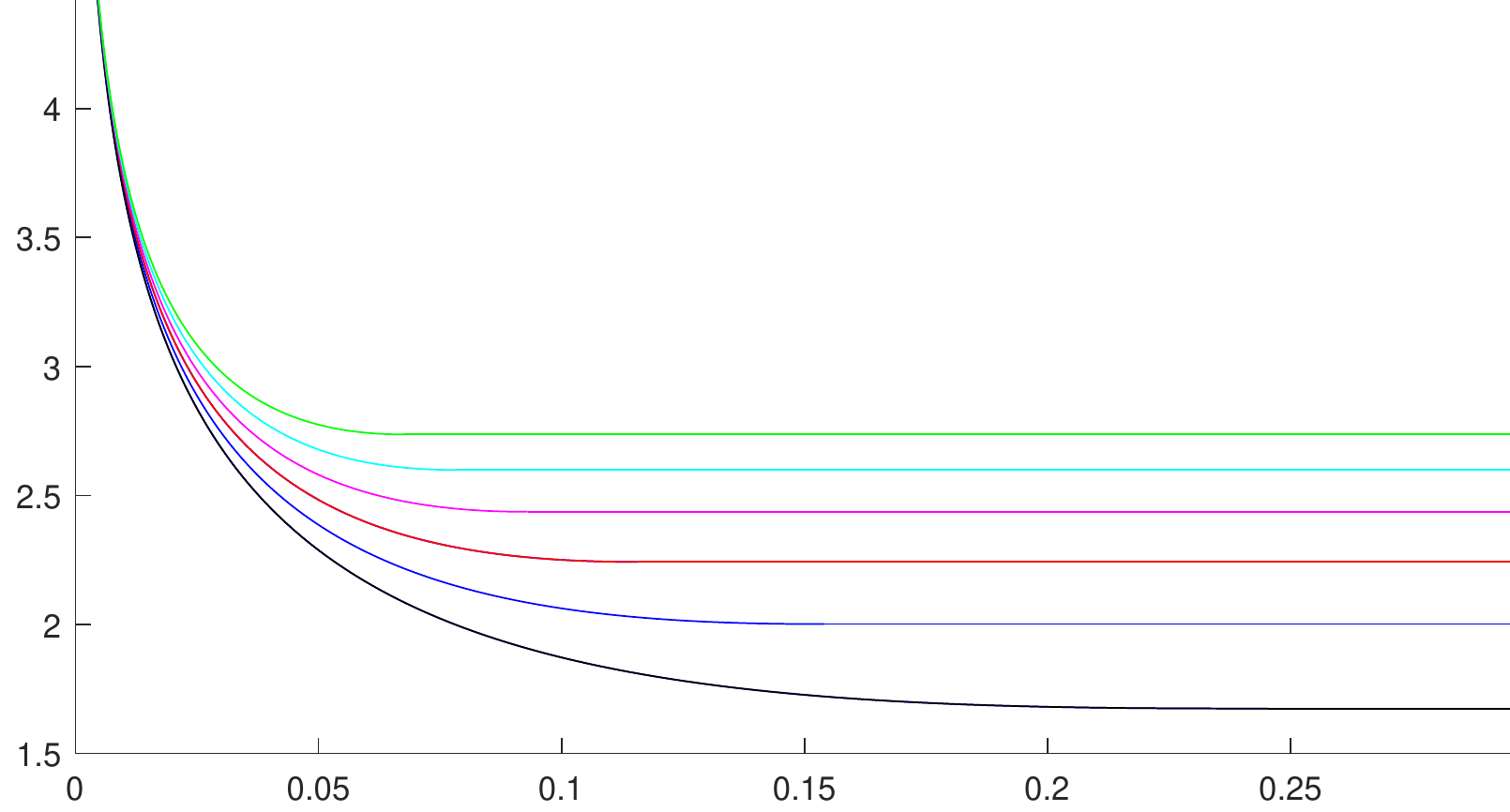}
    \caption{Plot of $W$ functions for $d=3,4,\ldots,8$; Lower function corresponds to higher $d$.}
    \label{fig:W}
\end{figure}

See Figure~\ref{fig:W} for plot of the $W$ function in dimensions $d=3,4,\ldots,10$. 
%%%%%%%%%%%%%%%%%%%%%%%%%%%%%%%%%%%%%%%%%%%%%%%%%%%%%%%%%%%%%%%%%%%%%%
\subsection{Concentration of Mass}\label{sec:cmass}
%%%%%%%%%%%%%%%%%%%%%%%%%%%%%%%%%%%%%%%%%%%%%%%%%%%%%%%%%%%%%%%%%%%%%%
We now return to our integral of interest,
\[
    M_{N}(b,\eps)=\int_{\norm{\psi}_{2}^{2}\in N(b-\eps,b+\eps)}\exp(-\norm{\nabla \psi}_{2}^{2})\vd \psi,
\]
and the proof of Theorem~\ref{thm:DISP}. We relate the integral to the concentration of mass of $\Psi^{\mv0,L(b)}$, and analyze two distinct cases depending on whether $b\le C_{d}$ or $b>C_{d}$. Let $\{X_{\mvk}\}$ be i.i.d.~exponential with rate 1 random variables. A simple change of variables argument tells us that the mass of $\Psi^{\mv0, y}$ may be represented as
\begin{align}\label{eq:expsum}
    \gC_{N,y}:=\norm{\Psi^{\mv0,y}}^2_2\equald \sum_{\mvk \neq \mv0} \frac1{\gl_{\mvk}+y}\cdot X_{\mvk}.
\end{align}

\begin{lem}\label{lem:massGFFlow}
    Let $0<b<C_{d}$ with $C_{d}$ and $L$ from~\eqref{def:Cd} and~\eqref{def:L}, respectively and $n\eps^{d}\gg 1$. Then we have, for some constant $c>0$
    \begin{align*}
        \pr\left( \gC_{N,L(b)}\in N(b-\eps,b+\eps)\right)\geq 1- cm_N(2)/N\eps^{2}.
    \end{align*}
\end{lem}
\begin{proof}
    The proof hinges on the fact that the expectation of $N^{-1}\gC_{N,y}$  is $K'_{N}(y)/N$, which converges $K'(y)$ as $N\to\infty$. Moreover, $K'$ has the well defined inverse $L(\cdot):(0,C_{d}]\to [0,\infty)$. What this means in practice is that we may choose $y$ such that the limiting mean is $b$, so long as $b\in(0,C_{d}]$.
    Using Lemma~\ref{lem:massconvergerate}, we know that
    \begin{align*}
        |K'_{N}(L(b))-b|\leq cN^{-1/d}m_{N}(3/2).
    \end{align*}
    It is therefore adequate to verify concentration of $\gC_{N,y}/N$ about $K_N'(y)$. Applying Chebyshev's inequality to~\eqref{eq:expsum},
    \begin{align*}
        \pr\left(\left|\gC_{N,y} - NK_{N}'(y) \right| \geq  N\eps \right)
        \leq (N\eps)^{-2}\cdot \sum_{\mvk \neq \mv0} (\gl_{\mvk}+y)^{-2}.
    \end{align*}
    Since $y\geq 0$,
    \begin{align*}
        \pr\left(\left|\gC_{N,y} - NK_{N}'(y) \right| \geq  N\eps \right)
        \leq (N\eps)^{-2}\cdot \sum_{\mvk \neq \mv0} \gl_{\mvk}^{-2}
        = m_N(2)/N\eps^{2}.
    \end{align*}
    The last line follows by Lemma~\ref{lem:epower}. This verifies the requisite concentration.
\end{proof}

\begin{cor}\label{cor:PerpInt}
    Let $0<b< C_{d}$, $L(\cdot)$ as in~\eqref{def:L}, and $\phi^{\mv0}$ as in~\eqref{eq:efns}. Let
    \[
        \cA=\left\{\psi^{\perp} \in \dC^{V}/\{\phi^{\mv0}\}: N(b-\eps)\leq \norm{\psi^{\perp}}_{2}^{2}\leq N(b+\eps)\right\}.
    \]
    Then we have
    \[
        \left|\frac{1}{N}\log \int_{\cA}\exp\left(-\norm{\nabla \psi^{\perp}}_{2}^{2}\right)\vd \psi^{\perp} -\bigl(\pi -W(b)\bigr)\right|\leq 2\eps L(b) + \left|\log \left(1- \frac{cm_N(2)}{N\eps^{2}}\right)\right|.
    \]

\end{cor}
\begin{proof}
    This corollary follows almost directly from verifying
    \begin{align}\label{eq:GFFPartfnbd}
        e^{N(b-\eps)L(b)}\left(1- \frac{cm_N(2)}{N\eps^{2}}\right)Z^{\mv0,L(b)} \leq \int_{\cA}\exp \left(-\norm{\nabla \psi^{\perp}}_{2}^{2}\right)\vd\psi^{\perp} \leq Z^{\mv0,L(b)}e^{N L(b)(b+\eps)}.
    \end{align}
    For efficiency of notation, for $y>0$ we will denote the quadratic form $\norm{\nabla \psi^{\perp}}_{2}^{2}+y\norm{\psi^{\perp}}_{2}^{2}$ by $Q^{y}(\psi^{\perp})$. The following bounds may be immediately verified.
    \[
        e^{NL(b)(b-\eps)}\int_{\cA}\exp({-Q^{L(b)}(\psi^{\perp})})\vd \psi \leq \int_{\cA} \exp\left({-\norm{\nabla\psi^{\perp}}_{2}^{2}}\right)\vd \psi
    \]
    and
    \[
        \int_{\cA} \exp\left({-\norm{\nabla \psi^{\perp}}_{2}^{2}}\right)\vd \psi \leq e^{N L(b)(b+\eps)}\int_{\cA} \exp({-Q^{L(b)}(\psi^{\perp})})\vd \psi.
    \]
    We begin with the upper bound to~\eqref{eq:GFFPartfnbd} as it is easier to establish, it follows simply by enlarging the region of integration to all of $\dC^{V}$:
    \[
        e^{N L(b)(b+\eps)}\int_{\norm{\psi^{\perp}}_{2}^{2}\in N(b-\eps,b+\eps)}\exp(-Q^{L(b)}(\psi^{\perp}))\vd \psi\leq \pi^{N-1}\exp\bigl(NL(b)(b+\eps)-NK_{N}(L(b))\bigr).
    \]
    Immediately, we may conclude that
    \[
        \frac1N \log \int_{\norm{\psi^{\perp}}_2^2\in N(b-\eps, b+\eps)}\exp\left(-\norm{\nabla \psi}_2^2\right)\vd \psi \leq \pi+ L(b)(b+\eps)-K_{N}(L(b)).
    \]
    Next, we show that the lower bound converges to the same limit as the upper bound. It is at this point that we introduce the concentration estimates. We have
    \begin{align*}
        \frac{1}{Z^{\mv0,L(b)}}\int_{\norm{\psi}_2^2\in N(b-\eps,b+\eps)}\exp(-Q^{L(b)}(\psi))\vd \psi= \pr \biggl(\norm{\Psi^{\mv0,L(b)}}^{2}\in N(b-\eps,b+\eps) \biggr).     \end{align*}
    Lemma~\ref{lem:massGFFlow} may now be directly applied.
\end{proof}

\begin{rem}
    It is a trivial extension of the above corollary that
    \[
        \frac{1}{N}\log \int_{\norm{\psi^{\perp}}_{2}^{2}\leq N(b+\eps)} \exp\left(\norm{\nabla \psi^{\perp}}_{2}^{2}\right)\vd\psi^{\perp}
    \]
    converges to $\pi -W(b)$ with the same bound for the rate of convergence. This fact is required for the proof of Theorem~\ref{thm:DISP}, but the proof is identical to that above and is therefore omitted.
\end{rem}
We have now assembled all the ingredients required to prove Theorem~\ref{thm:DISP}.
\begin{proof}[Proof of Theorem~\ref{thm:DISP}]
    Let $\psi$ be such that $N(b-\eps)\leq\norm{\psi}_{2}^{2}\leq N(b+\eps)$. We orthogonally decompose $\psi$ with respect to to $\phi^{\mv0}$ obtaining $\psi=c_{\mv0}\phi^{\mv0}+ \psi^{\perp}$. Thus,
    \[
        N(b-\eps)\leq |c_{0}|^{2}+\norm{\psi^{\perp}}_{2}^{2}\leq N(b+\eps).
    \]
    Let $b'=\min\{b,C_{d}\}$. We define
    \begin{align}\label{eq:circLB}
        \cA_{1}:=\left\{c_{\mv0}\phi_{\mv0}: b-b'-\frac{\eps}{2}\leq \frac{1}{N}|c_{0}|^{2}\leq b-b'+\frac{\eps}{2}\right\}\times \left\{\psi^{\perp}: b'-\frac{\eps}{2}\leq\frac{1}{N}\norm{\psi^{\perp}}_{2}^{2}\leq b'+\frac{\eps}{2}\right\}.
    \end{align}
    and
    \begin{align}\label{eq:circUB}
        \cA_{2}:=\left\{c_{0}\phi_{\mv0}: |c_{0}|^{2}\leq N(b+\eps)\right\}\times \left\{\psi^{\perp}: \norm{\psi^{\perp}}_{2}^{2}\leq N(b+\eps)\right\}.
    \end{align}
    It is clear that
    \[
        \cA_{1}\subset \cA \subset \cA_{2}.
    \]
    The proof is reduced to showing that the integrals of $\exp(-\norm{\nabla \psi}_{2}^{2})$ over $\cA_{1}$ and $\cA_{2}$ are logarithmically equivalent. The upper bound is easier, so we establish it first. We have
    \[
        \int_{\cA_{2}} \exp(-\norm{\nabla \psi}_{2}^{2})\vd \psi=\pi N(b+\eps)\cdot \int_{\norm{\psi^{\perp}}_{2}^{2}\leq N(b+\eps)} \exp\left(-\norm{\nabla \psi^{\perp}}_{2}^{2}\right)\vd \psi^{\perp}
    \]
    Applying Corollary~\ref{cor:PerpInt},
    \[
        \frac{1}{N}\log \int_{\cA_{2}}\exp(-\norm{\nabla \psi}_{2}^{2})\vd \psi \leq \pi -W(b)+\frac{1}{N}\log(b+\eps)N
    \]
    Now as for the lower bound,
    \[
        \int_{\cA_{1}} \exp(-\norm{\nabla \psi}_{2}^{2})\vd \psi \geq \pi N\eps\cdot \int_{\norm{\psi^{\perp}}_{2}^{2}\in N(b'-\eps/2,b'+\eps/2)}\exp\left(-\norm{\nabla\psi^{\perp}}_{2}^{2}\right)\vd \psi^{\perp}.
    \]
    The concluding step is to again use Corollary~\ref{cor:PerpInt}
\end{proof}

%%%%%%%%%%%%%%%%%%%%%%%%%%%%%%%%%%%%%%%%%%%%%%%%%%%%%%%%%%%%%%%%%%%%%%
\subsection{Bounds on the Maximum}\label{sec:useful}
%%%%%%%%%%%%%%%%%%%%%%%%%%%%%%%%%%%%%%%%%%%%%%%%%%%%%%%%%%%%%%%%%%%%%%
In order to address how the soliton is distinguishable from the background noise, we need upper bounds on the values that the free field can take at a point. Additionally, maximum bounds are crucial for the stitching procedure required in Section~\ref{sec:lbd}. The most crucial result in this section is the following.
\begin{thm}\label{thm:GFFtruncation}
    Let $C_{d}$ be as defined in~\eqref{def:Cd} and $b\in (0,\infty)$. We have
    \begin{align}\label{eq:intratio}
        \left(\int_{\cA'}\exp(-\norm{\nabla\psi}_{2}^{2})\vd \psi \right) 
        & /\left(\int_{\cA} \exp(-\norm{\nabla \psi}_{2}^{2})\vd \psi \right)\notag\\ 
        & \geq \frac{b-b'+\eps}{b+\eps}\cdot e^{-2N\eps L(b)}\cdot \left(1- 2cm_N(2)/N\eps^{2}\right)
    \end{align}
    where
    \begin{align*}
    \cA &=\{\psi : N(b-\eps)\leq \norm{\psi}_{2}^{2}\leq N(b+\eps) \}\\
    \text{and } 
    \cA' &=\cA \cap \{\psi : \norm{\psi}_{\infty}\leq 2\cdot \sqrt{3C_{d}\log N}\}.
    \end{align*}
\end{thm}
This goes one step beyond mass concentration, as it asserts that on the exponential scale, the dominant contribution to the integral $M_{N}(b,\eps)$ comes from functions for which the mass is relatively evenly spread over the entire torus, we cannot have too many sharp peaks. This phenomenon is closely related to (and proved by) an $\ell^{\infty}$ bound on the associated free field.

\begin{lem} \label{lem:GFFmax}
    Let $0<b\leq C_{d}$. We have for $n$ sufficiently large
    \[
        \pr\left (\norm{\Psi^{\mv0,L(b)}}_{\infty}\geq \sqrt{3C_{d} \log N}\right )\leq N^{-1}.
    \]
    \begin{proof}
        The translation invariance tells us that the random variables $\left\{\Psi^{\mv0,L(b)}_{{\mvx}}\right\}_{{\mvx} \in \dT^{d}}$ are identically distributed. The union bound is applicable and yields
        \begin{align*}
            \pr\left (\norm{\Psi^{\mv0,L(b)}}_{\infty}\geq \sqrt{3C_{d} \log N} \right )
            &\leq \sum _{{\mvx} \in \dT^{d}} \pr\left(\left|\Psi_{{\mvx}}^{\mv0,L(b)}\right|^{2}\geq 3C_{d} \log N\right)\\
            &
            =\exp\left(\left(1-\frac{3C_{d}}{K'_{n}(L(b))}\right)\cdot \log N \right).
        \end{align*}
        The proof follows using the fact that $K'_{N}(L(b)) \to K'(L(b))\leq C_{d}$, and thus for $N$ sufficiently large, $3C_{d}/K_{N}(L(b)) \geq 2$.
    \end{proof}
\end{lem}
\begin{proof}[Proof of Theorem~\ref{thm:GFFtruncation}]
    Recall the definitions of $\cA_{1}$ and $\cA_{2}$ be as in the proof of Theorem~\ref{thm:DISP}. We define
    \[
        \cA_{1}':=\cA_{1}\cap\left \{\psi : \norm{\psi^{\perp}}_{\infty}\leq \sqrt{3C_{d}\log N}  \right \}.
    \]
    It is clear that for $N$ sufficiently large, $\cA_{1}'\subset \cA'$. We will denote
    \[
        \cA_{\perp}:=\left\{\psi^{\perp} \in \dC^{V}/\{\phi^{\mv0}\}: N(b'-\eps)\norm{\psi^{\perp}}_{2}^{2}\leq N(b'+\eps)\right\}
    \]
    and
    \[
        \cA_{\perp}':=\cA_{\perp}\cap \left\{ \norm{\psi^{\perp}}_{\infty}\leq \sqrt{3C_{d} \log N }\right\}.
    \]
    We use the fact that $\cA_{2}$ contains $\cA$ to conclude that the ratio in~\eqref{eq:intratio} may be bounded below by
    \begin{align}\label{eq:intratio2}
        \left( \int_{\cA_{1}'} \exp\left(-\norm{\nabla \psi}\right)\vd \psi \right) /\left (\int_{\cA_{2}} \exp \left(-\norm{\nabla \psi}\vd \psi \right)\right).
    \end{align}
    On separation, we then know that
    \[
        \int_{\cA'_{1}} \exp\left(\norm{-\nabla\psi}_{2}^{2}\right)\vd \psi = N\pi(b-b'+\epsilon_{n})\cdot\int_{\cA_{\perp}'} \exp \left(-\norm{\nabla \psi^{\perp}}_{2}^{2}\right)\vd \psi^{\perp}.
    \]
    On introducing the exponential tilt, we further have
    \[
        \int_{\cA_{1}'}\exp \left( -\norm{\nabla \psi}_{2}^{2}\right)\vd \psi\geq N \pi (b-b'+\eps)\cdot e^{-L(b)(b-\eps)}\cdot \int _{\cA_{\perp}'} \exp\left( -Q^{L(b)}(\psi^{\perp})\right)\vd \psi^{\perp}.
    \]
    Analogously,
    \[
        \int_{\cA_{2}} \exp\left( -\norm{\nabla \psi}_{2}^{2}\right)\vd \psi = N \pi (b+\eps)\cdot e^{L(b)(b+\eps)}\cdot Z^{\mv0, L(b)}.
    \]
    Thus~\eqref{eq:intratio2} is bounded below further as
    \[
        \frac{b-b'+\eps}{b+\eps}\cdot e^{-2\eps L(b)}\cdot \frac{1}{Z^{\mv0,L(b)}}\int_{\cA_{\perp}'} \exp \left(-\norm{\nabla \psi}_{2}^{2} \right)\vd \psi.
    \]
    Now, observe that
    \begin{align*}
        \frac{1}{Z^{\mv0,y}} &\int_{\cA_{\perp}'} \exp \left( -\norm{\nabla \psi^{\perp}}\right)\vd \psi^{\perp}  \\ 
        &=\pr \left( b-\eps\leq \frac{1}{N}\norm{\Psi^{\mv0,L(b)}}_{2}^{2} \leq b+\eps ,\text{ } \norm{\Psi^{\mv0,L(b)}}_{\infty}\leq \sqrt{3C_{d}\log N}\right) \\                                                                                     &\geq 1- \frac{cm_N(2)}{N\eps^{2}} -\frac{1}{N}
    \end{align*}
    which follows via a combination of Lemmas~\ref{lem:massGFFlow} and~\ref{lem:GFFmax}, as well as the union bound.
\end{proof}

%%%%%%%%%%%%%%%%%%%%%%%%%%%%%%%%%%%%%%%%%%%%%%%%%%%%%%%%%%%%%%%%%%%%%%
\section{Soliton Solutions and Minimal Energy}\label{sec:soliton}
%%%%%%%%%%%%%%%%%%%%%%%%%%%%%%%%%%%%%%%%%%%%%%%%%%%%%%%%%%%%%%%%%%%%%%
In this section, we provide a survey on some results describing soliton solutions of the DNLS defined on the lattice $\dZ^{d}$. Recall, this means that $\psi_{{\mvx}}(t) \in \ell^{2}(\dZ^{d})$ solves
\begin{align*}
    \i \ddt\psi_{{\mvx}}=-(\gD\psi)_{{\mvx}}-|\psi_{{\mvx}}|^{p-1}\psi_{{\mvx}}.
\end{align*}
In particular, we will be interested in the time-periodic solutions, which we refer to as discrete breathers or solitons. 

%%%%%%%%%%%%%%%%%%%%%%%%%%%%%%%%%%%%%%%%%%%%%%%%%%%%%%%%%%%%%%%%%%%%%%
\subsection{Definition and Existence}\label{sec:exist}
%%%%%%%%%%%%%%%%%%%%%%%%%%%%%%%%%%%%%%%%%%%%%%%%%%%%%%%%%%%%%%%%%%%%%%
There are two ways of characterizing soliton solutions. The first is finding solutions via the ansatz $\psi_{{\mvx}}(t)=e^{\i \go t}\varphi_{{\mvx}}$ where $\varphi_{{\mvx}}$ is time-invariant. The second is variational, solitons can be realized as the minimizers of the Hamiltonian subject to the constraint $\norm{\varphi}_2^2=a$. The parameter $\go$ is realized as a Lagrange Multiplier and thus depends on $a$. The soliton equation is given by\begin{align}\label{eq:breather}
    \go(a) \cdot \varphi_{{\mvx}}=-(\gD \phi)_{{\mvx}}-|\varphi_{{\mvx}}|^{p-1}\varphi_{{\mvx}}.
\end{align}
The variational characterization also tells us that the values at all sites must have the same complex phase, and thus the discrete solitons may be assumed to be real-valued and nonnegative. Unlike the continuum case, there is no scale invariance for the DNLS defined on a given lattice. Thus, we cannot construct soliton solutions of arbitrary mass. The fundamental requirement for~\eqref{eq:breather} to admit an $\ell^2$ solution is that $\go(a) <0$. As discussed in the introduction, this occurs for all choices of mass $a>0$ when $p<1+4/d$, or when $a>R_{p,d}$ for $p\geq 1+4/d$. 
\begin{lem}[Weinstein, See~\cite{WEI99}]\label{lem:Rp}
    The following holds.
    \begin{enumeratei}
        \item Ground state or the minimizer in~\eqref{def:Ifn} exists when $I(a)\in (-\infty,0)$.
        \item Let $1<p<1+4/d$, then $I(a)<0$ for all $a>0$. Thus $R_{p}=0$.
        \item Let $p \ge 1+4/d$, then there exists a ground state excitation threshold $R_{p}>0$ so that $I(a)<0$ if $a>R_{p}$; and $I(a)=0$ if $a<R_{p}$. Moreover,
        \begin{align*}
            \frac2{p+1} R_{p}^{(p-1)/2} = \inf_{f} \left\{ \frac{ \norm{f}_2^{p-1}\cdot \norm{\nabla f}_2^2 }{ \norm{f}_{p+1}^{p+1} } \right\}.
        \end{align*}
    \end{enumeratei}
\end{lem}
The last statement interprets $R_{p}$ in terms of a functional inequality for the lattice. It is reciprocal to the best possible constant for the discrete Gagliardo-Nirenberg-Sobolev inequality to hold~\cite{WEI99}.

%%%%%%%%%%%%%%%%%%%%%%%%%%%%%%%%%%%%%%%%%%%%%%%%%%%%%%%%%%%%%%%%%%%%%%
\subsection{Dirichlet Solitons and Exponential Decay}\label{sec:expdec}
%%%%%%%%%%%%%%%%%%%%%%%%%%%%%%%%%%%%%%%%%%%%%%%%%%%%%%%%%%%%%%%%%%%%%%
In the continuum, when soliton solutions exist, they are known to be smooth and exponentially decaying. In the discrete setting, exponential decay still holds whenever the solitons exist, that is whenever $r>R_{p}$. A discrete counterpart of the continuum proof can be used to prove this result on $\dZ^{d}$. For our purposes, it suffices to establish a uniform rate of exponential decay for the Dirichlet problem defined on a growing sequence of boxes $\gL$ centered at the origin, and show the existence of an exponentially decaying $\ell^{2}(\dZ^{d})$ minimizer via a tightness argument. The analogous Dirichlet problem may be defined as
\begin{align}\label{eq:DirHam}
    \varphi^{\gL,a}:=\argmin\{\cH(\varphi): \varphi \in \dC^{\gL},\norm{\varphi}_{2}^{2}=a\}.
\end{align}
It is clear that the Dirichlet solitons must satisfy
\begin{align}\label{eq:DirEq}
    \go_{\gL}(a)\cdot \varphi^{\gL,a}_{{\mvx}}=\left(-\gD^{0}_{\gL}\varphi^{\gL,a}\right)_{{\mvx}}-\left|\varphi^{\gL,a}_{{\mvx}}\right|^{p-1}\varphi^{\gL,a}_{{\mvx}}
\end{align}
for a Lagrange multiplier $\go_{\gL}(a)$. Note, minimizers for each $\gL$ may be simultaneously defined on $\dZ^{d}$. Moreover, the resulting sequence itself is minimizing.
\begin{lem}\label{lem:DirMin}
    Let $\varphi^{\gL,a}_{{\mvx}}$ denote the Dirichlet minimizers of~\eqref{eq:DirHam} with mass $a$ on $\gL \subset \dZ^{d}$. Let $\{\gL_{k}\}_{k\geq 1}$ be a sequence of subsets each containing $\mv0$ such that $\gL_{k} \uparrow \dZ^{d}$ as $k \to \infty$. Then the sequence $\{\varphi^{\gL_{k},a}\}_{k\geq 1}$ is minimizing for~\eqref{def:HAM1}.
\end{lem}
\begin{proof}
    Let ${\varphi^{a,j}}$ be a minimizing sequence for $\cH$, each with mass $a$. Note, by the translation invariance of $\cH$, we may recenter the $\varphi^{a,j}$ such that  $\mv0$ is always the site with the largest absolute value. We then define
    \[
        \varphi^{a,j,\gL_{k}}:=\sqrt{\frac{a}{\norm{P_{\gL_{k}}\varphi^{a,j}}_{2}^{2}}}\cdot P_{\gL_{k}}\varphi^{a,j}.
    \]
    Note that as $\gL_{k} \uparrow \dZ^{d}$, $P_{\gL_{k}} \varphi \to \varphi$ in norm for any $\varphi \in \ell^{2}(\dZ^{d})$. We may then choose a diagonal subsequence that is minimizing for $H$ and denote it as $\{\varphi^{a,j_{k},\gL_{l}}\}_{k\geq 1}$. By hypothesis,
    \[
        \cH(\varphi^{\gL_{k},a})\leq \cH(\varphi^{n_{k},a,\gL_{k}}).
    \]
    Thus, the sequence of Dirichlet minimizers is also minimizing for the Hamiltonian.
\end{proof}
On the question of exponential decay, we provide a probabilistic proof directly adapted from Chatterjee~\cite{C14}, who proved the same for soliton solutions in the mass subcritical regime on the discrete torus. 

We bring this in now in order to control the value of the Lagrange Multiplier, which in turn is required for a uniform rate of exponential decay.
\begin{lem}\label{lem:gobound}
    Let $a>R_{p}$, for $|\gL|$ sufficiently large we have $\eps_{1}<\eps_{2}<0$ depending on $a$ such that
    \[
        \eps_{1}< \go_{\gL}(a) <\eps_{2}.
    \]
\end{lem}
\begin{proof}
    Multiplying both sides of~\eqref{eq:DirEq} by $\varphi^{\gL,a}_{{\mvx}}$ (recall that $\varphi^{\gL,a}$ is assumed to be nonnegative), we obtain that
    \begin{align*}
        \go_{\gL}(a)\cdot r
         & =\norm{\nabla \varphi^{\gL,a}}_{2}^{2}-\norm{\varphi^{\gL,a}}_{p+1}^{p+1}                  \\
         & \leq \norm{\nabla \varphi^{\gL,a}}_{2}^{2}-\frac{2}{p+1}\norm{\varphi^{\gL,a}}_{p+1}^{p+1}
        =\cH(\varphi^{\gL,a}).
    \end{align*}
    Now recall that as $\gL\uparrow \dZ^{d}$, $\cH(\varphi^{\gL,a})$ converges to $I(a)<0$, and we choose $\eps_{2}>I(a)$ but still strictly negative. As for the lower bound, we have that
    $
        \go_{\gL}(a)\geq -a^{(p+1)/2}=:\eps_{1}.
    $
\end{proof}

\begin{lem}\label{lem:decay}
    Let $a>R_{p}$. Let $\varphi^{\gL}$ denote a Dirichlet minimizer corresponding to $\gL$. We have $U_0\subset \gL$ and finite constants $\go_{0}>0$ and $C_{0}>0$ depending on $a$ and independent of $\gL$ such that
    \begin{align*}
        \varphi^{\gL}_{{\mvx}}\leq C_{0}e^{-\go_{0}\cdot d({\mvx}, U_0)}.
    \end{align*}
\end{lem}
\begin{proof}
    As we have already seen, $a>R_{p}$ implies that  for $I(a)<\eps<0$, $\go_{\gL}(a)<\eps$ for all $|\gL|$ sufficiently large. In turn,~\eqref{eq:DirEq} may be rewritten using Green's function description of the inverse of the Dirichlet Laplacian as
    \begin{align}\label{eq:DirEq2}
        \varphi^{\gL,a}_{{\mvx}}=\sum_{{\mvx'}}G^{\gL}_{-\go_{\gL}}({\mvx},{\mvx'})\left|\varphi^{\gL,a}_{{\mvx'}} \right|^{p-1}\varphi^{\gL,a}_{{\mvx'}}.
    \end{align}
    Now, let $\gd >0$, and define
    \[
        U_{\gd}:=\{{\mvx} \in \gL : \varphi^{\gL,a}_{{\mvx}}\geq \gd\}.
    \]
    By the usual bound,
    \[
        |U_{\gd}|\leq \frac{a}{\gd^{2}}.
    \]
    We define $z=\go_{\gL}(a)^{2}/(1+\go_{\gL}(a)^{2})$, and let $(X^{z}_{t})_{t\in \dN}$ be the simple symmetric random walk on $\gL$, started at ${\mvx}$ which is killed with probability $z$ at each step, and is killed at the boundary. Recall that the Green's function is given by
    \[
        G^{\gL}_{|\go_{\gL}|}({\mvx},{\mvx'})=\E \sum_{t=1}^{T_{\gL}^{|\go_{\gL}|}} \1_{X_{t}^{z}={\mvx'}}.
    \]
    Using this expression and the fact that the event of death is independent of the step taken, we may rewrite~\eqref{eq:DirEq2} as
    \[
        \varphi^{\gL,a}_{{\mvx}}= \sum_{t=0}^{\infty}(1-z)^{t}\E\left|\varphi^{\gL,a}_{X_{t}^{z}}\right|^{p-1}\varphi^{\gL,a}_{X_{t}^{z}}.
    \]
    Let $p_{\gL}({\mvx},{\mvx'},t)$ denote the probability kernel of $X_{t}$, the random walk on $\gL$ annihilated on the boundary. Observe that for all ${\mvx'}$ such that $d({\mvx'},{\mvx})>t$, the probability that the random walk has reached ${\mvx'}$ is $0$. Thus,
    \begin{align}\label{eq:DirEq3}
        \varphi^{\gL,a}_{{\mvx}}=\sum_{{\mvx'} \in \gL}\sum_{t\geq d({\mvx},{\mvx'})}(1-z)^{t}p_{\gL}({\mvx},{\mvx'},t) \left( \varphi^{\gl ,a }_{{\mvx'}}\right)^{p}.
    \end{align}
    Clearly $p({\mvx},\vy,t)\leq 1$ and on evaluation of the geometric sum, we have
    \[
        \varphi^{\gL,a}_{{\mvx}}\leq \frac{1}{z}\sum_{{\mvx'} \in \gL}{(1-z)^{d({\mvx},{\mvx'})}}\left( \varphi^{\gl,a}_{{\mvx'}}\right)^{p}.
    \]
    Now, if ${\mvx'} \notin U_{\gd}$, then we have that $(\varphi^{\gL,a}_{{\mvx}})^{p}\leq \gd^{p-1} \varphi_{{\mvx'}}^{\gL,a}$, and if ${\mvx'} \in U_{\gd}$, then $d({\mvx}, U_{\gd})\leq d({\mvx},{\mvx'})$. Thus, partitioning the sum in~\eqref{eq:DirEq3},
    \begin{align}
        \varphi^{\gl,a}_{{\mvx}} 
        &\leq \frac{(1-z)^{d(U_{\gd},{\mvx})}}{z}\sum_{{\mvx'} \in U_{\gd}}\left(\varphi^{\gL,a}_{{\mvx'}} \right)^{p}+\frac{\gd^{p-1}}{z}\sum_{{\mvx'} \notin U_{\gd}}\varphi^{\gL,a}_{{\mvx'}}\notag\\
    &\leq \frac{(1-z)^{d(U_{\gd},{\mvx})}}{z}\cdot a^{p/2}\cdot \frac{a}{\gd^{2}} +\frac{\gd^{p-1}}{z}\sum_{{\mvx'} \in \gL}(1-z)^{d({\mvx},{\mvx'})}\varphi^{\gL,a}_{{\mvx'}}.\label{eq:DirInEq1}
    \end{align}
    Let
    \[
        \eta_{{\mvx}}:= \frac{a^{(p+2)/2}}{z \gd^{2}} \cdot \max_{{\mvx'}' \in \gL}(1-z)^{d({\mvx'}',U_{\gd})+d({\mvx},{\mvx'}')/2}.
    \]
    We observe as a consequence of the triangle inequality,
    \[
        \eta_{{\mvx}}\leq \frac{a^{(p+2)/2}}{z\gd^{2}}\max_{{\mvx'}' \in \gL} (1-z)^{d({\mvx'}',U_{\gd})+d({\mvx'},{\mvx'}')/2-d({\mvx},{\mvx'})/2}
        \leq (1-z)^{-d({\mvx},{\mvx'})}\eta_{{\mvx'}}.
    \]
    Let $C$ denote the smallest constant such that
    \begin{align}\label{def:C}
        \varphi^{\gL,a}_{{\mvx}}\leq C \eta_{{\mvx}}.
    \end{align}
    Since $\eta_{{\mvx}}>0$ and $\gL$ is finite, it is clear that $C$ is finite as well. We use this bound for~\eqref{eq:DirInEq1} obtaining
    \begin{align*}
        \varphi^{\gL,a}_{{\mvx}} 
        &\leq \frac{(1-z)^{d(U_{\gd},{\mvx})}}{z}\cdot \frac{a^{(p+2)/2}}{\gd^{2}}+\frac{C\gd^{p-1}}{z}\sum_{{\mvx'} \in \gL}(1-z)^{d({\mvx},{\mvx'})} \eta_{{\mvx'}}\\
     &\leq \frac{(1-z)^{d(U_{\gD},{\mvx})}}{z}\cdot \frac{a^{(p+2)/2}}{\gd^{2}}+\frac{C\gd^{p-1}}{z}\eta_{{\mvx}}\sum_{{\mvx'} \in \gL}(1-z)^{\frac{1}{2}d({\mvx},{\mvx'})}\\
    &\leq \left(1 + C \cdot \gd^{p-1} \sum_{{\mvx'} \in \gL}(1-z)^{d({\mvx},{\mvx'})/2}  \right)\varphi_{{\mvx}}.
        \end{align*}
    We can choose $\gd$ sufficiently small, such that \[
        \frac{\gd^{p-1}}{z}\sum_{{\mvx'} \in \gL}(1-z)^{d({\mvx},{\mvx'} )}\leq \frac{1}{2}.
    \]
    In turn, this tells us that
    \[
        C\leq 1+{C}/{2}
    \]
    since $C$ is the best constant for~\eqref{def:C}. This is the same as $C\leq 2$. With this choice of $\gd$, we have that
    \[
        \varphi^{\gL,a}_{{\mvx}}\leq 2\frac{a^{(p+2)/2}}{z \gd^{2}}\max_{{\mvx'} \in \gL} (1-z)^{d(U_{\gd},{\mvx'})+d({\mvx},{\mvx'})/2}.
    \]
    By the triangle inequality,
    \[
        d(U_{\gd},{\mvx'})+\frac{1}{2}d({\mvx},{\mvx'})\geq \frac{1}{2}d(U_{\gd},{\mvx'})+\frac{1}{2}d({\mvx},{\mvx'})\geq  \frac{d(U_{\gd},{\mvx})-d({\mvx},{\mvx'}')}{2}+\frac{1}{2}d({\mvx},{\mvx'}).
    \]
    Thus,
    \[
        \varphi^{\gL,a}_{{\mvx}}\leq 2\frac{a^{(p+2)/2}}{z \gd^{2}}(1-z)^{d({\mvx},U_{\gd})/2}.
    \]
    To conclude, take
    $
        U_{0}=U_{\gd},\text{ }C_{0}=2\frac{a^{(p+2)/2}}{z \gd^{2}}\text{ and }\go_{0}=-\frac{1}{2}\log(1-z).
    $
\end{proof}
We take this moment to emphasize that we may bound all the constants arising in Lemma~\ref{lem:decay} in terms of the $\eps_{1}$ and $\eps_{2}$ introduced in Lemma~\ref{lem:gobound}. To begin with, as a direct consequence of Lemma~\ref{lem:gobound},
\[
    \frac{\eps^{2}_{2}}{1+\eps^{2}_{2}} \leq z \leq \frac{\eps^{2}_{1}}{1+\eps^{2}_{1}}  .
\]
Next, a valid choice of $\delta$ is
\[
    \gd=\frac{z|\log(1-z)|^{d}}{2(d-1)!}
\]
which in turn yields
\[
    C_{0}=\frac{4\cdot a^{(p+1)/2}\cdot (d-1)!}{z^{2}|\log(1-z)|^{d}}\leq \frac{4\cdot a^{(p+1)/2}\cdot (d-1)!\cdot (1+\eps_{2}^{2})}{\eps_{2}^{2}\cdot |\log(1+\eps^{2}_{2})-2\log \eps_{2}|^{d}}
\]
The uniform rate of exponential decay of the Dirichlet minimizers implies the existence of an exponentially decaying minimizer to~\eqref{def:HAM1}, via a standard compactness argument. For the sake of clarity of exposition, we detail the argument here.
\begin{lem}\label{lem:expdecL}
    Let $a>R_{p}$. There exists an exponentially decaying minimizer $\varphi^{a}$ to~\eqref{def:HAM1}.
\end{lem}
\begin{proof}
    Given a box $\gL=[-n,n]^{d}\cap \dZ^{d}$ centered at the origin, we define $2\cdot \gL$ as $[-2n,2n]^{d}\cap \dZ^{d}$. We embed the minimizer $\varphi^{\gL,a}$ into $\ell^{2}(2\cdot\gL)$ via the standard inclusion map and note that the energy is still the same. We are then free to translate $\varphi^{\gL,a}$ within this larger box and preserve the energy. We define $\tilde{\varphi}^{\gL,a}$ to be the translation of $\varphi^{\gL,a}$ such that the site with the largest mass is situated at the origin. By Lemma~\ref{lem:decay} there exist, independent of $\gL$, a set $U_{0}$ of bounded size and positive constants $C_{0}$ and $\go_{0}$ such that $\tilde{\varphi}^{\gL,a}\leq C_{0} e^{-\go_{0} d({\mvx}, U_{0})}$. By construction, it is clear that $\mv0\in U_{0}$, and thus there is a box $\gL_{0}$ such that $\tilde{\varphi}^{\gL,a} \leq C_{0}e^{-\go_{0}\cdot d(\gL_{0},{\mvx})}$. Thus, $\{\tilde{\varphi}^{\gL,a}\}$ is a pre-compact sequence in $\ell^{2}(\dZ^{d})$. We take $\varphi^{a}$ to be any accumulation point of the sequence. Clearly, $\varphi^{a}_{{\mvx}}\geq 0$, $\norm{\varphi^{a}}_{2}^{2}=a$ and $\varphi^{a}_{{\mvx}}\leq C_{0} e^{-\go_{0}d({\mvx}, \gL_{0})}$.
\end{proof}
\begin{rem}
    We remark that the above process of re-centering should not be necessary as the sequence of Dirichlet minimizers should have a mass that is concentrated towards the center of the box $\gL$. This is due to the decay of the Dirichlet heat kernel at the edges of the box. We, in fact, conjecture that the minimizer for the Dirichlet problem is unique, which would in turn imply the uniqueness of the minimizer to~\eqref{def:HAM1} up to translation and phase rotation.
\end{rem}
The exponentially decaying minimizer and the corresponding truncations are very useful from the perspective of constructing a minimizing sequence with a good rate of convergence. We establish this now.
\begin{lem}\label{lem:minrate1}
    Let $a>R_{p}+\eps$, $\gL=[-n,n]^{d}\cap \dZ^{d}$, and let $\varphi^{\gL,a}$ denote the Dirichlet minimizer. We have
    \[
        \left|I(a)-\cH\left(\varphi^{\gL,a}\right) \right|\leq C_{0}'\cdot a^{p^{2}-1}\cdot e^{-\go_{0}\cdot n} .
    \]

\end{lem}
\begin{proof}

    Let $\varphi^{a}$ be as in Lemma~\ref{lem:expdecL}. Note that for any $p$,
    \begin{align}\label{eq:decay-truncation}
             \norm{\varphi^{a}}^{p}_{p}-\left(\frac{a}{\norm{P_{\gL}\varphi^{r}}_{2}^{2}}\right)^{p/2}\norm{P_{\gL}\varphi^{a}}_{p}^{p}\leq \left|1-\left(\frac{a}{\norm{P_{\gL}\varphi^{r}}_{2}^{2}}\right)^{p}\right|\norm{\varphi^{a}}_{p}^{p}.
    \end{align}
    By construction, we know that for all ${\mvx}$ outside $\gL$
    \[
        \varphi^{a}_{{\mvx}}\leq C_{0}\exp(-\go_{0}\cdot n).
    \]
    Thus, we have that
    \[
        \norm{P_{\gL}\varphi^{a}}_{2}^{2}\geq a-\frac{C_{0}\cdot \Gamma(d)}{\go_{0}^{d}}\exp \left({-\go_{0}\cdot n}\right).
    \]
    For the convenience of notation, we will introduce for the purposes of this proof
    \[
        \epsilon(n):=C_{0}\exp (-\go_{0}\cdot n).
    \]
    Applied to~\eqref{eq:decay-truncation},
    \begin{align*}
        \norm{\varphi^{a}}^{p+1}_{p+1}-\left(\frac{a}{\norm{P_{\gL}\varphi^{a}}_{2}^{2}}\right)^{(p+1)/2}\norm{P_{\gL}\varphi^{a}}_{p+1}^{p+1} &\leq \left|\frac{a^{p+1}-(a-\epsilon(n))^{p+1}}{(a-\epsilon(n))^{p}}\right|\norm{\varphi^{a}}^{p}_{p}  \\ &\leq \frac{(p+1)a^{p(p+1)}\epsilon(n)}{(a-\epsilon(n))^{p+1}}.
    \end{align*}
    Now as for the gradient term, we only have to control the contributions from outside $\gL$ and on the boundary. We have that
    \begin{align*}
        \norm{\nabla \varphi^{a}}^{2}_{2}- \frac{a}{\norm{P_{\gL}\varphi^{a}}_{2}^{2}}\cdot \norm{\nabla P_{\gL}\varphi^{a}}_{2}^{2}\leq \sum_{{\mvx} \sim {\mvx'} \in \gL^{c}}|\varphi^{a}_{{\mvx}}-\varphi^{a}_{{\mvx'}}|^{2} & +\sum_{{\mvx}\sim{\mvx'} \in \gL}\left(1-\frac{a}{\norm{P_{\gL}\varphi^{a}}_{2}^{2}}\right)\cdot |\varphi^{a}_{{\mvx}}-\varphi^{a}_{{\mvx'}}|^{2} \\ &+\sum_{{\mvx}\sim {\mvx'} \in \partial \gL}|\varphi^{a}_{{\mvx}}|^{2}.
    \end{align*}
    Applying the bound obtained from the exponential decay,
    \[
        \norm{\nabla \varphi^{a}}^{2}_{2}- \frac{a}{\norm{P_{\gL}\varphi^{a}}_{2}^{2}}\cdot \norm{\nabla P_{\gL}\varphi^{a}}_{2}^{2}\leq 2\epsilon(n)+\left(\frac{\epsilon(n)}{a-\epsilon(n) }\right) a.
    \]
    Combining, we have
    \[
        \left|I(a)-\cH\left(\sqrt{\frac{a}{\norm{P_{\gL}\varphi^{a}}}}\varphi^{a}\right)\right|\leq \epsilon(n)\cdot \left( \frac{a^{p(p+1)}}{2(a-\epsilon(n))^{p+1}}  +\frac{2a}{a-\epsilon(n)}\right)\leq C_{0}'\cdot a^{p^{2}-1}\cdot e^{-\go_{0}\cdot n}.
    \]
    To conclude, we note that
    \[
        I(a)\leq \cH(\varphi^{\gL,a})\leq \cH\left(\sqrt{\frac{a}{\norm{P_{\gL}\varphi^{a}}_{2}^{2}}} P_{\gL} \varphi^{a} \right).
    \]
    This completes the proof.
\end{proof}

For the scenario where the minimizer does not exist, any weakly convergent sequence to zero is minimizing. We may use this fact to bound the rate of convergence.
\begin{lem}\label{lem:minrate2}
    Let $a\leq R_{p}$, $\gL=[-n,n]^{d}\cap \dZ^{d}$ and $\varphi^{\gL,a}$ be the Dirichlet minimizer with mass $a$. Then we have
    \[
        |\cH(\varphi^{\gL,a})|\leq \frac{2ad}{|\gL|^{1/d}}.
    \]
\end{lem}
\begin{proof}
    Let $\mv1_{\gL,{\mvx}}$ denote the function which takes the value $1$ for every ${\mvx}\in \gL$ and is zero outside $\gL$. It is easy to evaluate the norms of $\mv1_{\gL}$, in particular
    \[
        \norm{\nabla \mv1_{\gL}}_{2}^{2}\leq 2d|\gL|^{(d-1)/d}.
    \]
    Thus,
    \[
         \cH\left(\sqrt{\frac{a}{|\gL|}}\mv1_{\gL}\right) \leq \frac{2ad}{|\gL|^{1/d}}.
    \]
    Since we are working with $a\leq R_{p}$, we know that $I(a)=0$. Combined with the fact that $\cH(\varphi^{\gL,a})$ is monotonically decreasing, we must have that $\cH(\varphi^{\gL,a})\geq 0$. For $|\gL|$ sufficiently large, by the minimizing hypothesis, we have
    \[
        0 \leq \cH(\varphi^{\gL,a})\leq \cH\left( \sqrt{\frac{a}{|\gL|}}\mv1_{\gL} \right)\leq \frac{2ad}{|\gL|^{1/d}}
    \]
    and this completes the proof.
\end{proof}

%%%%%%%%%%%%%%%%%%%%%%%%%%%%%%%%%%%%%%%%%%%%%%%%%%%%%%%%%%%%%%%%%%%%%%
\subsection{Analysis of Minimal Energy}
%%%%%%%%%%%%%%%%%%%%%%%%%%%%%%%%%%%%%%%%%%%%%%%%%%%%%%%%%%%%%%%%%%%%%%
We conclude this section with some basic facts about the function $I$. Clearly, we have
\[
    \norm{\nabla\psi}_2^2 \le 4d \norm{\psi}_{2}^{2} \text{ and } \norm{\psi}_{p+1}^{p+1} \le  \norm{\psi}_{2}^{2}\cdot\max_{x\in\dZ^d}|\psi_x|^{p-1}\le \norm{\psi}_{2}^{p+1}.
\]
Define the functions $J,\widehat{J}:(0,\infty)\to[0,\infty)$ by
    \begin{align}\label{eq:jjhat}
        J(a) :=\frac1aI(a) \text{ and } 
        \widehat{J}(a):=\frac1aI(a)+\frac2{p+1}(a\vee R_p)^{(p-1)/2}.
    \end{align}
    
\begin{lem}\label{lem:Ideriv}
The function $J$ is decreasing and differentiable. 
\end{lem}

\begin{proof}
    Fix $0<a'<a$. Given a function $\psi$ with mass $a'$, we consider the function $\tilde{\psi}(x)=\sqrt{a/a'}\cdot  \psi(x)$ with mass $a$. We have
    \begin{align*}
        \cH(\tilde{\psi}) =  \frac{a}{a'} \norm{\nabla \psi}_{2}^{2} - \frac{2}{p+1}\left(\frac{a}{a'}\right)^{(p+1)/2}\norm{\psi}_{p+1}^{p+1},
    \end{align*}
    or
    \[
        \frac{I(a)}{a}\le \frac{1}{a}\cH(\tilde{\psi}) =  \frac{1}{a'}\cH(\psi) - 2\cdot \frac{a^{(p-1)/2} -  (a')^{(p-1)/2}}{(p+1)\cdot  (a')^{(p+1)/2}} \norm{\psi}_{p+1}^{p+1} \le \frac{\cH(\psi)}{a'}.
    \]
    Thus we have $\frac{I(a)}{a}\le \frac{I(a')}{a'}$.  Similarly, given a function $\tilde{\psi}$ with mass $a$, we consider the function $\psi(x)=\sqrt{a'/a}\cdot  \tilde{\psi}(x)$ with mass $s$ to have
    \begin{align*}
        \frac{I(a')}{a'}\le \frac{1}{a'}\cH(\psi) & = \frac{1}{a}\cH(\tilde{\psi}) + \frac{a^{(p-1)/2} -  (a')^{(p-1)/2}}{(p+1) a^{(p+1)/2}} \norm{\tilde{\psi}}_{p+1}^{p+1} \\
        & \le \frac{1}{a}\cH(\tilde{\psi}) + \frac{2}{p+1}(a^{(p+1)/2} -  (a')^{(p-1)/2}).
    \end{align*}
    Thus we have
    \[
        0\le \frac{I(a')}{a'} - \frac{I(a)}{a} \le \frac{2}{p+1}(a^{(p-1)/2} -  (a')^{(p-1)/2})
    \]
    and the proof is complete.
\end{proof}
A trivial corollary of this result is the differentiability of the $I$ function itself, merely a consequence of the product rule. 
\begin{lem}\label{lem:jhat}
$\widehat{J}$ is increasing in $r$. Moreover, 
\[
\frac2{p+1}R_p^{(p-1)/2}\le \widehat{J}(a) \le 1/C_{d} \text{ for all }a>0.
\]
\end{lem}
\begin{proof}
By evaluating $\cH$ at the function $\psi(x)=\sqrt{a}\cdot \ind_{x=0}$ in~\eqref{def:Ifn}, it is clear that
\begin{align}\label{Irubd}
    I(a) \le 2d\cdot a - \frac2{p+1}a^{(p+1)/2}.
\end{align}
From Lemma~\ref{lem:Ideriv}, we have for $a\ge R_{p}$ and $a'=R_{p}$
\begin{align}\label{def:Irlbd}
    I(a) \ge \frac{2}{p+1}R_{p}^{(p-1)/2}\cdot a - \frac2{p+1}a^{(p+1)/2}.
\end{align}
\end{proof}

%%%%%%%%%%%%%%%%%%%%%%%%%%%%%%%%%%%%%%%%%%%%%%%%%%%%%%%%%%%%%%%%%%%%%%
\section{Convergence of the Free Energy }\label{sec:feconv}
%%%%%%%%%%%%%%%%%%%%%%%%%%%%%%%%%%%%%%%%%%%%%%%%%%%%%%%%%%%%%%%%%%%%%%
In this section, we will establish the convergence of the scaled free energy. Recall that our measure is restricted to  the ball of radius $\sqrt{N}$ in $\dC^{V}$. We will be cutting the mass constraint into pieces corresponding to  ``solitonic'' and dispersive contributions. The vertex set $V$ will be partitioned into a set $U$ which is small, where a typical function is concentrated, and $U^{c}$ where the mass of a typical function is spread out over all the sites. Recall that we denote the restrictions of $\psi$ onto $U$ and $U^{c}$ as $\psi_U$ and $\psi_{U^{c}}$ respectively. We denote the collection of sites ${\mvx}\in U$ adjacent to $U^{c}$ by $\partial U_{int}$ and analogously the collection of sites ${\mvx'}\in U^{c}$ adjacent to $U$ by $\partial U_{ext}$. The Hamiltonian $\cH_n$ may then be written as\begin{align}\label{eq:split}
    \begin{split}
        \cH_n(\psi)
        &= \cH_n(\psi_U)+ \cH_n(\psi_{U^{c}})\\
        &\qquad+\sum_{{\mvx}\sim{\mvx'} \in \partial U} |\psi_{{\mvx}}-\psi_{{\mvx'}}|^2
        - \sum_{{\mvx} \in \partial U_{int}} |\psi_{{\mvx}}|^2 - \sum_{{\mvx'} \in \partial U_{ext}} |\psi_{{\mvx'}}|^2.
    \end{split}
\end{align}
In the definitions of $\cH_n(\psi_U)$ and $\cH_n(\psi_{U^{c}})$, the Laplacian is understood to mean the Dirichlet Laplacian. The decomposition of $\psi$ into $\psi_U$ and $\psi_{U^{c}}$ is to be regarded as the orthogonal decomposition $\dC^{V}=\dC^U\oplus \dC^{U^{c}}$, and as such the expression $\psi=\psi_U\oplus \psi_{U^{c}}$ carries the obvious meaning, as does the decomposition of the volume element $\vd \psi=\vd \psi_U \cdot \vd \psi_{U^{c}}$. The mass constraint, which may be expressed as
\[
    \norm{\psi_{U}}_{2}^{2}+\norm{\psi_{U^{c}}}_{2}^{2}\leq N,
\]
cannot be expressed as a product set, but can be closely approximated by a disjoint union of product sets. For a choice of $0<\ga<1$ to be specified later, let $\kappa=N^{\ga}$. Let $i^{*}$ be the smallest natural number such that $i^{*}\kappa \geq N$ and $i_{*}$ be the largest natural number such that $i_{*}\kappa \leq N$. We have that
\begin{align}
    \bigcup_{i+j\leq i^{*}} \left\{ \psi_{U}: \norm{\psi_{U}}_{2}^{2}\in [i\kappa,i\kappa+\kappa) \right\}\times \left\{ \psi_{U^{c}}: \norm{\psi_{U^{c}}}_{2}^{2}\in [j\kappa,(j+1)\kappa) \right\}
\end{align}
contains the ball of radius $\sqrt{N}$, and
\begin{align}
    \bigcup_{i+j=i_{*}} \left\{ \psi_{U}: \norm{\psi_{U}}_{2}^{2}\in [i\kappa,i\kappa+\kappa) \right\}\times \left\{ \psi_{U^{c}}: \norm{\psi_{U^{c}}}_{2}^{2}\in [j\kappa,(j+1)\kappa) \right\}
\end{align}
is contained by the ball. Replacing the ball with these sets as the region of integration will give us upper and lower bounds for the partition function. For ease of notation in the sections to follow, we define
\begin{align}\label{eq:MassShell}
    \cB_{i,U} &=\left\{\psi_{U}: \norm{\psi_{U}}_{2}^{2}\in [i\kappa, i\kappa+i )\right\}\\
\text{and } \cA_{i,j,U} &=\cB_{i,U}\times \cB_{j,U^{c}}.\notag
\end{align}
To summarize, for a given fixed $U\subset V$, we have
\begin{align}\label{eq:MassPartition}
    \bigcup_{i+j=i_{*}} \cA_{i,j,U} \subset \left\{ \psi: \norm{\psi}_{2}^{2}<N\right\} \subset \bigcup_{i+j\leq i^{*}} \cA_{i,j,U}.
\end{align}
\subsection{Upper Bound}\label{sec:ubd}
With the decomposition of the mass constraint introduced above, we will split the proof of Theorem~\ref{thm:free} into two parts. This section is dedicated to the first portion, the appropriate upper bound for the free energy.
\begin{lem}\label{lem:ZUB}
    We have a constant $C$ and a sequence $\Theta_{S}(N)$ depending on $\nu$, $\theta$, $p$ and $d$ such that
    \begin{align*}
        \frac1N\log Z_N(\theta,\nu)
        \leq \log(\pi/\theta)- \inf_{0<a<1} (W(\theta(1-a)) + \theta\nu^{-1} I(a\nu))+\frac1N\log \Theta_{S}(N)
    \end{align*}
    where
    \[
        \Theta_{S}(N)\leq \exp\left( 10s_{N}N \right)\cdot N^{2s_{N}^{-d-1}}\cdot e^{C \gk_{N}}.
    \]
\end{lem}

We begin with an argument to justify the separation of functions into concentrated and dispersed parts, in other words showing that for any function $\psi \in \dC^{V}$, we may always find a set $U$ such that $\psi$ is not concentrated outside $U$, and the boundary contribution may be controlled.
\begin{lem}\label{lem:separation}
    Given a function $\psi\in\dC^V$ with $\norm{\psi}_2^2< N$ and $\eps>0$, there exists a subset $U\subset V$ with $|U|\le \eps^{-d-1}$ such that $|\psi_x|^2< \eps N$ for all $x\notin U$ and $\sum_{x\in \partial U_{int}\cup \partial U_{ext}}|\psi_{x}|^2 < 3\eps N$.
\end{lem}
\begin{proof}
    Take $U_0=\{\mvx: |\psi_{\mvx}|^2\ge \eps N\}$. Clearly $|U_0|\le 1/\eps$. We define the $U_{i}$ by the successive addition of the 2-step outer boundary, \ie\ $U_i=\{  \mvx \in V: d(\mvx,U_0)=2i \}$ and $B_i=U_i\setminus U_{i-1}$ for all $i\ge 1$. Take $B_0=U_0$. Note that
    \[
        \sum_{i=1}^k\sum_{\substack{ \mvx \in B_{i}}} |\psi_{\mvx}|^2
        \le \sum_{\mvx} |\psi_{\mvx}|^2\le N.
    \]
    In particular, there exists $i\le k$ such that $\sum_{\mvx\in B_{i}} |\psi_{\mvx}|^2
        \le N/k$. Take $k=\lfloor\eps^{-1}/2\rfloor$, so that $N/k\le 10\eps N$. Now, note that $|U_i|\le (2k+1)^d |U_0|\le \eps^{-d-1}$.
\end{proof}

We will refer to the set $U$ as $\eps-$admissible for $\psi$, and what Lemma~\ref{lem:separation} states that given a function $\psi$ with appropriate mass and $\eps>0$, we may find an $\eps-$admissible set. For a  fixed set $U$ and a positive sequence $s_{N}$ decaying to zero at a specified rate (see~\eqref{eq:snrate}), we define
\begin{align*}
    \cU(s_{N}):=\{\psi\in \dC^n : U \text{ is $s_N-$admissible for $\psi$}\}.
\end{align*}
It is easy to see that if $\cU$ is non-empty, then it must be open in $\dC^{V}$ since we may perturb slightly around any $\psi$ that is contained. Therefore, if it is non-empty, it must have a non-zero Lebesgue measure. Further, for every $\psi$ with appropriate mass, we know that there must be a $U$ with size bounded above by $s_{N}^{-d-1}$ such that $U$ is $s_{N}$ admissible for $\psi$. Thus, we have that
\[
    \{\psi : \norm{\psi}_{2}^{2}< N\} \subset \bigcup_{|U|\leq s_{N}^{-d-1}} \cU(s_{N}).
\]
On combination with~\eqref{eq:MassPartition}, we have
\begin{align}
    Z_N\leq\sum_{|U|\leq s_{N}^{-d-1}} \cU(s_{N})\sum_{i+j\leq i^{*}}\int_{\cU \cap \cA_{i,j,U}} \exp(-\cH_N(\psi))\vd \psi.
\end{align}
With a fixed choice of $U$ and admissible $\psi$, the bound on the gradient and~\eqref{eq:split} yield
\begin{align*}
    \exp(-\cH_N(\psi))\leq \exp(-\cH_N(\psi_U))\cdot \exp(-\cH_N(\psi_{U^{c}}))\cdot \exp(3 s_{N} N).
\end{align*}
For any $\psi\in \cU$, we have a bound on the maximum outside of $U$, and the gradient control. Thus, if we take
\begin{align}\label{eq:MassShell1}
    \tilde{\cB}_{j,U^{c}}:=\left\{\psi_{U^{c}} \in \cB_{j,U^{c}} : \norm{\psi_{U^{c}}}_{\infty} \leq \sqrt{s_{N} N} \text{ and } \norm{\psi_{\partial U_{ext}}}_{2}^{2}\leq 3s_{N} N \right\}
\end{align}
then it follows that
\begin{align}
    \cU \cap \cA_{i,j,U} \subset \cB_{i,U} \times \tilde{\cB}_{j,U^{c}}.
\end{align}
This completes the separation of the partition function. We have
\begin{align*}
    Z_n\leq \exp(3\theta s_N N)\cdot \sum_{|U|\leq s_{N}^{-d-1}} \sum_{i+j \leq i^{*}}
     & \int_{\cB_{i,U}}\exp(-\theta \cH_N(\psi_U))\vd\psi_U                                        \\
     & \qquad \cdot \int_{\tilde{\cB}_{j,U^{c}}}\exp(-\theta \cH_N(\psi_{U^{c}}))\vd \psi_{U^{c}}.
\end{align*}

\begin{lem}\label{lem:SUB}
    Let $U$ be a fixed subset with size bounded above by $s_{N}^{-d-1}$, and $\cB_{i,U}$ be as in~\eqref{eq:MassShell}. We have
    \begin{align*}
        \int_{\cB_{i,U}}\exp(-\theta \cH_N(\psi_U))\vd \psi_U \leq \exp\left(-N\theta I_{\nu}\left(\frac{i\gk_{N} +\gk_N}N\right)\right)\cdot\Theta_{1}(N) \end{align*}
    where
    \begin{align*}
        \Theta_{1}(N)\leq {N^{s_{N}^{-d-1}} \pi^{s_{N}^{-d-1}}e^{C\gk_{N}}}.
    \end{align*}

\end{lem}
\begin{proof}
    Since all sets $U$ under consideration have size bounded above $s_{N}^{-d-1}$ and the smallest non trivial cycle has size of order $N^{\frac1{d}}$, it follows that for $N$ sufficiently large $U$ can be embedded in $\dZ^{d}$. We assume that this is the case, fix an embedding, and as an abuse of notation, refer to the corresponding images as $U$ and $\psi_U$. We remind the reader  that
    \begin{align*}
        \cH_N(\psi_U)=\frac{N}{\nu}\cdot \cH(N^{-1/2}\nu^{-1/2}\cdot \psi_U).
    \end{align*}
    We recall the function $I$ introduced in~\eqref{def:Ifn}. By definition and using the monotonicity of $I$, it follows that
    \begin{align*}
        \exp(-\theta \cH_N(\psi_U))\leq \exp\left(-\frac{N\theta}{\nu} I\left(\nu\cdot  \frac{i\gk_N+\gk_N}N\right)\right).
    \end{align*}
    Thus,
    \begin{align*}
        \int_{\cB_{i,U}}\exp(-\theta \cH_N(\psi_U))\vd \psi_U \leq \exp\left(-\frac{N\theta}{\nu} I\left(\nu\cdot  \frac{i\gk_N+\gk_N}N\right)\right)\cdot \text{Vol}(\cB_{i,U}).
    \end{align*}
    The volume of $\cB_{i,U}$ is easy to evaluate, as it is a concentric shell with inner radius $\kappa$ and outer radius $\kappa i+\kappa$. Taking the crude bound (the largest possible shell which has outer radius $N$) we have
    \begin{align*}
        \text{Vol}(\cB_{i,U})\leq  \frac{N^{|U|}\pi^{|U|}\kappa_{N}}{\Gamma(|U|+1)} \leq N^{|U|}\pi^{|U|}\gk_{n}.
    \end{align*}
\end{proof}

Lemma~\ref{lem:SUB} addresses the contribution to the free energy from the structured portion of the Hamiltonian. We now establish the contribution from the dispersive portion.
\begin{lem}\label{lem:DUB}
    Let $U\subset V$ be a fixed subset and let $\tilde{\cB}_{j,U^{c}}$ be as defined in~\eqref{eq:MassShell1}. We have
    \begin{align*}
        \int_{\tilde{\cB}_{j,U^{c}}} \exp(-\theta \cH_N(\psi_{U^{c}}))\vd \psi_{U^{c}} \leq \exp\biggl(N \log(\pi/\theta) -N W\left(\theta \cdot  \frac{j\kappa_{N} +\kappa_{N}}{N}\right)\biggr)\cdot \Theta_2(N)\cdot \Theta_{2}'(N,j)
    \end{align*}
    where
    \begin{align}\label{eq:theta2}
    \Theta_{2}(N)\leq   \exp\left (10 \theta s_{N}N \right) \text{ and } \Theta'_{2}(N,j) =\exp\left(2\kappa_{N} L\left(\theta \cdot  \frac{j \kappa_{N} +\kappa_{N}}{N}\right) \right).
    \end{align}
\end{lem}
\begin{proof}
    The $\ell^{\infty}$ bound on the $\psi \in \tilde{\cB}_{j,U^{c}}$ tells us that
    \begin{align*}
        \|\psi_{U^{c}}\|_{p+1}^{p+1}\leq (s_N N)^{(p-1)/2}N.
    \end{align*}
    For the Hamiltonian, this yields
    \begin{align*}
        \cH_N(\psi_U)= \norm{\nabla_0 \psi_{U^{c}}}_2^2-\frac{ N^{(1-p)/2}}{p+1}\norm{\psi_{U^{c}}}_{p+1}^{p+1}\geq \norm{\nabla_0\psi_{U^{c}}}_2^2-\frac1{p+1}\cdot N\cdot (s_N)^{(p-1)/2}
    \end{align*}
    and in turn,
    \begin{align}\label{eq:disppartfnbd1}
        \int_{\tilde{\cB}_{j,U^{c}}}\exp(-\theta \cH_N(\psi_{U^{c}}))\vd \psi_{U^{c}} \leq \exp\bigl(\theta N s_N^{(p-1)/2}\bigr)\int_{\tilde{\cB}_{j,U^{c}}}\exp(-\theta \norm{\nabla_0\psi_{U^{c}}})\vd \psi_{U^{c}}.
    \end{align}
    This enables us to discard the non linearity and consider only the free portion of the Hamiltonian. The technique we will adopt is to, so to speak, ``patch'' the missing portion to instead consider the free field on the torus. It is easy to see, just by maximum bounds and calculating the volume of a box with side length $\sqrt{3C_{d}\log |U|\cdot \theta^{-1}}$ that 
    \begin{align*}
        (6C_{d}\theta^{-1} \log U)^{|U|}\cdot \exp\left(-6C_{d} d |U| \log |U|\right)\leq \int_{\{\norm{\psi_U}_{\infty}<\sqrt{3C_{d} \theta^{-1} \log |U|}\}} \exp(-\theta \norm{\nabla_{0} \psi_{U}}_{2}^{2})\vd \psi_U.
    \end{align*}
    Thus~\eqref{eq:disppartfnbd1} maybe bounded above by
    \begin{align}\label{eq:disppartfnbd2}
        \left(\frac{\theta |U|^{2 d }}{6C_{d} \log |U|}\right)^{|U|} \cdot\int_{\tilde{\cB}_{j,U} \times \{\norm{\psi_{U}}_{\infty}\leq \sqrt{3C_{d} \log |U|}\}}\exp\bigl(-\theta(\norm{\nabla_0\psi_U}_2^2+\norm{\nabla_0\psi_{U^{c}}}_2^2)\bigr)\vd \psi_{U^{c}} \cdot \vd \psi_U
    \end{align}
    Since we know that $\norm{\psi_{\partial U_{ext}}}_{2}^{2}\leq 3s_{N} N$ and $\norm{\psi_U}_{\infty}\leq \sqrt{3C_{d}\theta^{-1}\log |U|}$,
    \begin{align*}
        \sum_{{\mvx} \sim \vy \in \partial U} |\psi_{{\mvx}}-\psi_{\vy}|^2 \leq 6s_{N}N+6C_{d}\theta^{-1}|U|\log|U|.
    \end{align*}
    We have implicitly used the AM-GM inequality in the above bound. We obtain that~\eqref{eq:disppartfnbd2} may be further bounded above by
    \[
            \left(\frac{\theta |U|^{6C_{d}(1+d) }}{6C_{d} \log |U|}\right)^{|U|} \cdot \exp(Ns_{N}^{(p-1)/2})\cdot\exp(6Ns_{N})\int_{\tilde{\cB}_{j,U^{c}}\times\{\norm{\psi_{U^{c}}}_{\infty}\leq \sqrt{3C_{d}\theta^{-1}\log |U|}\}} \exp(-\norm{\nabla \psi}_{2}^{2})\vd \psi
    \]
    Since $|U|\leq s_{N}^{-d-1}$, the largest mass that may be allocated to $\psi_{U}$ is bounded above by $3\theta^{-1}C_{d}|U|\log|U|$. We may enlarge our mass shell slightly so that for $N$ sufficiently large,
    \[
        \tilde{\cB}_{j,U^{c}}\times \left \{\norm{\psi_{U}}_{\infty}\leq \sqrt{3C_{d}\theta^{-1} \log |U|}\right \}\subset \cB_{j,V}\cup \cB_{j+1,V}.
    \]
    Combining all our bounds we have
    \begin{align*}
         & \int_{\tilde{\cB}_{j,U^{c}}}\exp(-\theta \cH_{N}(\psi_{U^{c}}))\vd \psi_{U^{c}} \leq \Theta_{2}(N,U) \int_{\cB_{j,V} \cup \cB_{j+1,V}} \exp(-\theta \norm{\nabla \psi}_{2}^{2})\vd \psi
    \end{align*}
    where
    \begin{align*}
        \Theta_2 (N,U)
         :=\exp\biggl(\theta N s_N^{(p-1)/2} &+ 6\theta N s_{N} + \bigl(6C_{d}(1+d)\bigr)|U|\log |U| \\ &+(\log  \theta-\log 6C_{d}) |U| - |U|\log \log |U|\biggr).
    \end{align*}
    Using the bound $|U|\leq s_{N}^{-d-1}$, we keep only the leading order terms in the error  to obtain
    \[
    \Theta_{2}(N,U) \leq \left(8 \theta N s_{N} \right)
    \]
To conclude, all that needs be done is apply Theorem~\ref{thm:DISP}, which tells us that
    \begin{align*}
        \int_{\cB_{j,V}\cup \cB_{j+1,V}} \exp(-\theta \norm{\nabla \psi}_{2}^{2})\vd \psi\leq \left(\frac{\pi}{\theta}\right)^{N} & \cdot \exp\left(-N \cdot W\left(\theta \cdot  \frac{j\kappa_{N}+\kappa_{N}}{N}\right)\right) \\ &\cdot \exp \left( 2\kappa_{N} L\left(\theta \cdot \frac{j\kappa_{N}+\kappa_{N}}{N}\right) + 2\log N  \right).
    \end{align*}

    The $\log N$  addition to the error is far smaller than the leading term, we may ignore it. With $\Theta'_{2}(j)$ defined as in~\eqref{eq:theta2} , we conclude the proof.
\end{proof}
Having individually bounded the concentrated and dispersed contributions to the partition function, we are now ready to combine them to yield the upper bound of the limiting free energy.
\begin{proof}[Proof of Lemma~\ref{lem:ZUB}]
    We remind the reader that in the prior sections we have removed the explicit $U$, and have all bounds in terms of the maximum allowed $|U|$. To avoid cumbersome expressions, we begin by combining $\Theta_1$, $\Theta_2$ and the errors arising from the separation at the boundary, carry out the sum over $U$ to define
    \[\Theta_{S}:=\exp(3\theta Ns_{N})\cdot \binom{N}{s_{N}^{-d-1}}\cdot \Theta_1(N)\cdot \Theta_2(N).
    \]
    Before proceeding further, we deal with the binomial coefficient. We have
    \[
    \binom{N}{s_{N}^{-d-1}}\leq e^{s_{N}^{-d-1}}\cdot N^{s_{N}^{-d-1}}. 
    \]
    \begin{align*}
        Z_N\leq \Theta_{S}(N)\cdot \sum_{i +j \leq i^{*}}\exp\left(\log \left(\frac{\pi}{\theta} \right)^{N}- \frac{N\theta}{\nu}\cdot  I\left(\nu \cdot \frac{i\gk_N+\gk_N}N\right)-N \cdot W\left(\theta \cdot  \frac{j\kappa_{N} +\kappa_{N}}N\right)\right)\cdot \Theta_{2}'(j).
    \end{align*}
    Both $I$ and $W$ are monotonically decreasing. In particular, if $j_{1}\leq j_{2}$, then
    \[
        \exp\left(-N W\left( \theta\cdot  \frac{j_{1}\kappa_{N}+\gk_{N}}{N}\right)\right)\leq \exp\left(-N \cdot W\left( \theta\cdot  \frac{j_{2}\kappa_{N}+\gk_{N}}{N}\right)\right).
    \]
    For a given $i$, the largest value that $j$ can take is $i^{*}-i$. We use this fact to also render the sum over $j$ irrelevant. We have
    \begin{align*}
        Z_{N}\leq \Theta_{S}\cdot \sum_{i=1}^{i^{*}} \Theta_{2}'(i) \cdot \exp N\left(\log \left(\frac{\pi}{\theta} \right)-\frac{\theta}{\nu}   I\left(\nu \cdot \frac{i\gk_N}N+\nu \cdot \frac{\gk_{N}}{N}\right)-W\left(\theta\left(1-\frac{i\gk_{N}}{N}\right) + 2\theta\cdot  \frac{\gk_{N}}{N} \right) \right) .
    \end{align*}
    In the above expression, we have used the fact that
    \[
        \frac{i^{*}\gk_{N}}{N}-1\leq \frac{\gk_{N}}{N},
    \]
    and have redefined $\Theta_{S}$ to include the additionally accrued $N$ from summing over $j$. We now need to tackle the $i$ dependent error term $\Theta_{2}'$ and control the regularity of $W$. To do both, we bring in Lemma~\ref{lem:abound}. For $M>0$, we  have an $i'$ dependent on $M$ and $\theta$ such that for $i>i'$,
    \[
        W\left(\theta\cdot \frac{(i^{*}-i)\gk_{N}+\gk_{N}}{N}\right)\geq M+I(1).
    \]
    We also recall from Lemma~\ref{lem:propW} that
    \[
        L(b)\leq b^{-1},
    \]
    which tells us that
    \[
        \gk_{n} L\left (\theta \frac{(i^{*}-i)\gk_{N}+\gk_{N}}{N}\right)\leq N/\theta.
    \]
    Thus on choosing $M>1/\theta$,
    \begin{align*}
        \sum_{i=i'}^{i^{*}} \Theta_{2}'(i)\cdot \exp N\biggl( \log \bigl(\frac{\pi}{\theta}\bigr) - & \frac{\theta}{\nu}\cdot  I\bigl( \frac{i\gk_{N}  +\gk_{N}}{N} \bigr) - W\bigl(\theta \cdot  \frac{(i^{*}-i)\gk_{N}+\gk_{N}}{N}\bigr)    \biggr)  \\ &\leq N\cdot e^{-N(M-1/\theta)}.
    \end{align*}
    When $i\leq i'$, the same argument tells us that $\Theta'_{2}$ is uniformly bounded by $\exp(\kappa_{N}\cdot C')$ where $C'$ is a constant depending on $\theta$, $\nu$ and $M$. Further, in this case both $I$ (rescaled by $\nu$) and $W$ (rescaled by $\theta$) are Lipschitz functions. Let the upper bound on both Lipschitz constants be denoted by $C_{L}$ depending on $\theta$, $\nu$ and $M$. Taking $i\gk_{N}/N$ to be $a$, we discard the perturbations to the arguments of the $I$ and $W$ functions obtaining 
    \[
        Z_{N}\leq e^{C \gk_{N}}\cdot \Theta_{S}\cdot \sum_{i=1}^{i'} \exp N\left(\log \left(\frac{\pi}{\theta}\right) -\theta \nu^{-1} I(\nu a)-W(\theta(1-a))\right) +\Theta_{S}\cdot N e^{-C_{M} N}
    \]
    We then obtain the trivial bound by optimizing over $a$ and finally rendering the sum over $i$ irrelevant. We have
    \[
        Z_{N}\leq N\cdot e^{C \gk_{N}}\cdot \Theta_{S}\cdot \exp N\left(\log \left(\frac{\pi}{\theta}\right)-\inf_{a}\bigl(\theta \nu^{-1}I(\nu a)-W(\theta(1-a))\bigr) \right) +\Theta_{S}\cdot e^{-C_{M}\cdot N}.
    \]
    We remind the reader that by Lemma~\ref{lem:abound} the constant $C_{M}$ may be chosen such that
    \[
        C_{M}-\inf_{a}\left(\theta \nu^{-1}I(\nu a)+W(\theta(1-a))\right)\geq M.
    \]
    Thus,
    \[
        Z_{N}\leq e^{C \gk_{N}}\cdot \Theta_{S}\cdot \exp\left(\log \left(\frac{\pi}{\theta}\right)-\inf_{a}\bigl(\theta \nu^{-1}I(\nu a)-W(\theta(1-a))\bigr) \right) \cdot \left(1+e^{-C N}\right).
    \]
    Redefining $\Theta_{S}$ one last time to include the additionally accrued error term $e^{C\gk_{N}}$, we are done. The error accrued from the $1+e^{-CN}$ is insignificant compared to the leading order error term and is therefore ignored.  

\end{proof}

%%%%%%%%%%%%%%%%%%%%%%%%%%%%%%%%%%%%%%%%%%%%%%%%%%%%%%%%%%%%%%%%%%%%%%
\subsection{Lower Bound}\label{sec:lbd}
%%%%%%%%%%%%%%%%%%%%%%%%%%%%%%%%%%%%%%%%%%%%%%%%%%%%%%%%%%%%%%%%%%%%%%
In this section, we will verify the corresponding lower bound of the limiting free energy. The restriction of the region of integration for the lower bound as chosen in~\eqref{eq:MassPartition} should now be motivated by results of  Section~\ref{sec:ubd}, as we saw this was the region of dominant contribution.
\begin{lem}\label{lem:ZLB}
    We have a constant $C$ and sequences $\gamma_{N}$ and $\Theta_{I}(N)$ depending on $\nu$, $\theta$, $\rho$ and $\nu$ such that 
    \[
        \frac1N \log Z_N(\theta, \nu)\geq \log(\pi/\theta)-\inf_{0<a<1}(W(\theta(1-a))+\theta \nu^{-1}I(a\nu)) +\frac1N \log \Theta_{I}(N)
    \]
    where
    \[
        \Theta_{I}(N)\geq \exp(-N \theta \gamma_{N}-C\cdot \gk_{N})
    \]
    and
    \[
    \lim_{N\to \infty} \gamma_{N}=0.
    \]
\end{lem}
As with the upper bound we and partition the mass of a function into the structured and dispersive part. The lower bound will be established by restricting the integral to a specifically chosen subset of functions within $\cB_{i,U}\times \cB{j,U^{c}}$. For the results to come, it will be helpful for the sake of conciseness to define
\begin{align}
    \rho_i:= \frac{2i\gk_{N} +\gk_{N}}{2N}.
\end{align}
We will take $U$ to be a cube of side length $\lfloor \log N \rfloor$. We will restrict the solitonic portion of our integral to be centered about a specific minimizing sequence. Let $U^{o}$ denote the set of interior points of $U$, and recall that $\varphi^{a,U^{o}}$ denotes the Dirichlet minimizer on $U^{o}$ with mass $a$. We define
\begin{align}\label{eq:MassBreak3}
    \cC_{i,U}:=\left\{\psi_U \in \dC^{U}: \norm{ \psi_U-N^{1/2}\cdot \varphi^{\rho_{i},U^{o}}}_2^2\leq N^{-\gd}\right\},
\end{align}
where $\gd>0 $ and fixed. As for the dispersive portion, we will define \begin{align}\label{eq:MassBreak4}
    \cC_{j,U^{c}}:= \left\{\psi_{U^{c}} \in \cB_{j,U^{c}}: \norm{\psi_{U^{c}}}_{\infty}\leq 2\sqrt{3C_{d} \log N} \right\}.
\end{align}
It is clear that $C_{i,U}\subset \cB_{i,U}$ and the analogous inclusion for $C_{j,U^{c}}$ is obvious by definition. Note that any function sampled from $\cB_{j,U}$ places a maximum mass of $N^{-\delta}$ on the boundary, and any function sampled from $\cB_{j,U^{c}}$ may place a maximum mass of $6 C_{d} d  \cdot  (\log N)^{d}$ on the boundary. Combining these facts  yields the boundary estimate required for separating the soliton and dispersive parts,
\begin{align}\label{eq:BdryLB}
    \sum_{{\mvx} \sim \mvy \in \partial U} |\psi_{{\mvx}}-\psi_{\mvy}|^2\leq \sum_{{\mvx} \in \partial{U}_{int}} |\psi_{{\mvx}}|^{2} +\sum_{\mvy \in \partial U_{ext}}|\psi_{\mvy}|^{2} \leq 12dC_{d} (\log N)^{d}.
\end{align}
Applying~\eqref{eq:MassPartition},~\eqref{eq:split} and~\eqref{eq:BdryLB} we have
\begin{align}\label{eq:sumi}
\begin{split}
    Z_N &\geq \exp\bigl(-12 C_{d}d (\log N)^{d}\bigr)\cdot \sum_{i=1}^{i_{*}-1}\int_{\cC_{i,U}}\exp(-\theta \cH_N(\psi_U))\vd \psi_U\\
    &\qquad\qquad\qquad \cdot \int_{\cC_{i_{*}-i,U^{c}}} \exp(-\theta \cH_N(\psi))\vd \psi_{U^{c}}.
\end{split}
\end{align}

\begin{lem}
    Let $\cB_{j,U}$ be as in~\eqref{eq:MassBreak3}. Then we have
    \[
        \int_{\cC_{i,U}} \exp(-\theta \cH_N(\psi_U))\vd \psi_U \geq \exp\left(-N\theta \nu^{-1}\cdot I(\nu \rho_{i})\right)\cdot \Theta_3(N,i)
    \]
    where
    \[
        \Theta_3(i,N)\geq \exp\left(-7dN^{1/2}-N\gamma_{i,N}\right)
    \]
    and
    \[
        \gamma_{i,N}:=\max_{\psi \in B_{i,U}}|\nu^{-1/2}\cH(\nu^{1/2}\cdot N^{-1/2}\cdot \psi)-\nu^{-1}I(\nu \rho_{i})|.
    \]
\end{lem}
\begin{proof}
    Our region of integration $\cC_{i,U}$ is the ball of radius $N^{-\gd/2}$ centered around $\varphi^{\rho_{i},U^{o}}$. We will verify that this region is appropriately close in $\ell^2$ to the actual minimizer $\phi$, and thus the energy $\cH_N(\psi)$ is close to $N \nu ^{-1} I(\nu \rho_{j})$
    \[
        \int_{\cC_{i,U}} \exp(-\theta \cH_N(\psi_U))\vd \psi_U =\exp(-\theta N \nu^{-1} I\left(\nu \rho_{i}\right))\cdot \int_{\cB_{i,U}} \exp(\theta (N\nu^{-1} I\left(\rho_{i}\right)-\cH_N(\psi)))\vd \psi_U
    \]
    Let $\psi \in \dC^U$ have unit mass and $t\in [0,N^{-\gd/2}]$. Using the Cauchy-Schwarz inequality and the fact that the discrete Laplacian is bounded, we have
    \[
        \norm{\nabla_{0}(N^{1/2}\cdot\varphi^U +t \psi)}^2_2-\norm{N^{1/2}\cdot \nabla_0 \varphi^U}_2^2\leq 4dN^{-\gd/2}\cdot\sqrt{N\rho_{i}}+2dN^{-\gd} \leq 5dN^{(1-\gd)/2}.
    \]
    Further,
    \[
        \norm{N^{1/2}\cdot\varphi^U+t\psi}_{p+1}^{p+1}-\norm{N^{1/2}\cdot\varphi^U}_{p+1}^{p+1}\leq N^{(p+1)/2}\norm{\psi}_{p+1}\int_0^{t\cdot N^{-1/2}}(\norm{\varphi^U}_{p+1}+s\norm{\psi}_{p+1})^{p}\text{d}s
    \]
    \[
        \leq N^{(1+p)/2}\norm{\psi}_{p+1}\norm{\varphi^U}_{p+1}^{p}\int_0^{t\cdot N^{-1/2}}\left(1+s\frac{\norm{\psi}_{p+1}}{\norm{\varphi^U}_{p+1}}\right)^{p} \text{d}s.
    \]
    For any function in $\dC^U$ with mass $a$, we know that the lowest possible $\ell^{p+1}$ norm is attained by the uniform function and is given by $a^{1/2}\cdot|U|^{(1-p)/2(p+1)}$, and the largest possible is attained by concentrating all the mass onto a single site and given by $a^{1/2}$. Thus, we have a bounded constant depending only on $a$ such that
    \[
        \frac{1}{N^{(p-1)/2}}\left( \norm{N^{1/2}\cdot\varphi^U+t\psi}_{p+1}^{p+1}-\norm{N^{1/2}\cdot \varphi^U}_{p+1}^{p+1}\right)\leq C\cdot N^{(-\gd +1)/2}.
    \]
    We then note that
    \[
        |\nu^{-1}\cH_{N}(\psi)-N\nu^{-1}I(\nu \rho_{i})|\leq \max_{\psi \in \cB_{j,U}}|\cH_N(\psi)-NI(\nu \rho_{i})|=NC_{i}
    \]
    and note that $C_{i}$ is bounded and converges to $0$ as $n\to \infty$ as a simple consequence of Lemma~\ref{lem:DirMin}. As $\cC_{i,U}$ is a ball of radius $N^{-\gd}$ in $\dC^U$, we observe
    \[
        \text{Vol}(\cC_{j,U})=\frac{\pi^{|\log N|^{d}}}{\Gamma(1+|\log N|^{d})}\cdot N^{-2\gd|\log N|^{d}}.
    \]
    Defining
    \begin{align}\label{eq:theta3}
        \Theta_3(N):=\exp\left(-6d N^{(1-\gd)/2}-N\gamma_{i,N}\right)\cdot\frac{\pi^{|\log N|^{d}}}{\Gamma(1+|\log N|^{d})}N^{-2\gd|\log N|^{d}}
    \end{align}
    and then taking $\gd=0$ completes the proof. We clearly have 
    \[
    \Theta_{3}(N,i)\geq \exp \left(-6dN^{1/2}-N\gamma_{i,N}-2d|\log N|^{d}\log \log N \right)\geq \exp \left(-7dN^{1/2}-N\gamma_{i,N} \right)
    \]
\end{proof}

\begin{lem}
    Let $\cC_{j,U_{c}}$ be as in~\eqref{eq:MassBreak4}. We have
    \[
        \int_{\cC_{j,U^{c}}} \exp(-\theta \cH_N(\psi_{U_{c}}))\vd \psi_{U^{c}} \geq \exp\left(N \log \left( \frac{\pi}{\theta} \right)-N\cdot W \left(\theta \rho_{j} \right) \right)\cdot \Theta_{4}(j,N)
    \]
    where
    \[
        \Theta_{4}(j,N)\geq (3C_{d} \log |\log N|)^{-2|\log N|^{d}}\cdot e^{-2\gk_{N}L(\theta \rho_{j})}.
    \]
\end{lem}
\begin{proof}
    Since we seek a lower bound, we may immediately discard the non linear part of the Hamiltonian to obtain
    \[
        \int_{\cC_{j,U^{c}}} \exp(-\theta \cH_N(\psi_{U^{c}}))\vd \psi_{U^{c}}\geq \int_{\cC_{j,U^{c}}} \exp(-\theta \norm{\nabla_0\psi}_2^2)\vd \psi_{U^{c}}.
    \]
    The next step, just as in the case of the upper bound, is to ``patch'' the free field to the entire torus. Merely from the fact that the exponential of a negative number is bounded above by 1, we have
    \[
        \int_{\{\norm{\psi_U}_{\infty}\leq \sqrt{3C_{d}\log 
        N}\}}\exp(-\theta \norm{\nabla_0 \psi}_2^2)\vd \psi_U \leq \left( {3C_{d} \log N}\right)^{|\log N|^{d}}
    \]
    On the product space $\cC_{j,U^{c}}\times \{\norm{\psi_U}_{\infty}\leq \sqrt{3C_{d} \log N}\}$, by definition,
    \[
        \norm{\psi}_{2}^{2}\leq (j+1)\gk_{N}+3C_{d}|\log N|^{d+1} \text{ and }\norm{\psi_{U^{c}}}_{2}^{2} \geq j \gk_{N}.
    \]
    If we define
    \[
        \widehat{\cB}_{j}:=\left\{\psi \in \dC^{V} : \norm{\psi}^2_2 \in \bigl[j\gk_{N}+\frac{\gk_{N}}{4},j\gk_{n}+\frac{3\gk_{N}}{4}\bigr), \norm{\psi}_{\infty} \leq \sqrt{3C_{d} \log N} \right\},
    \]
    it is clear that for $N$ sufficiently large,
    \[
        \widehat{\cB}_{j}\subset \cC_{j,U^{c}}\times \{\psi_U: \norm{\psi_U}_{\infty}\leq \sqrt{3C_{d} \log N}\}.
    \]
    Thus,
    \begin{align}\label{eq:solitonlb1}
        \int_{\cC_{j,U^{c}}}\exp(-\theta \norm{\nabla_0 \psi}_2^2)\vd\psi_{U^{c}}\geq \frac{1}{ (3C_{d}\log N)^{|\log N|^{d}}}\int_{\widehat{\cB}_{j}}\exp(-\theta( \norm{\nabla_0 \psi_U }_2^2+\norm{\nabla_0 \psi_{U^{c}}}))\vd \psi.
    \end{align}
    The functions under consideration are bounded above, and we also know the size of the boundary of $U$ is bounded above by $4d|\log N|^{d-1}$. This gives us the gradient control required for patching. We have
    \[
        \sum_{{\mvx} \sim \mvy \in \partial U} |\psi_{{\mvx}}-\psi_{\mvy}|^2 \leq 36dC_{d}(\log N)^{d}.
    \]
    We use this to further lower bound~\eqref{eq:solitonlb1} as
    \[
    \frac{1}{ (3C_{d}\log N)^{|\log N|^{d}}}\cdot \exp(-36dC_{d} (\log N)^{d}) \int_{\widehat{\cB}_{j}} \exp(-\theta \norm{\nabla \psi}_2^2)\vd \psi.
    \]
    Using a combination of Theorems~\ref{thm:DISP} and~\ref{thm:GFFtruncation}, we have that
    \begin{align*}
        &\int_{\cB_{j}}\exp(-\theta \norm{\nabla \psi}_{2}^{2})\vd \psi \\
        &\qquad \geq \left(\frac{\pi}{\theta}\right)^{N}\cdot \exp\bigl(-N W(\theta\rho_{j})\bigr)\cdot \frac{N(\theta\rho_{i}-\theta C_{d})_{+}+\gk_{N}}{N\theta \rho_{i}+\gk_{N}}\cdot e^{-2\gk_{N}L(\theta \rho_{j})}\cdot \left(1-\frac{2cm_{N}(2)N}{\gk_{N}^{2}}\right)^{2}.
    \end{align*}
    Define
    \[
        \Theta_{4}:=\frac{\exp(-36dC_{d}(\log N)^{d})}{(3dC_{d}\log N)^{|\log N|^{d}}}\cdot \frac{\gk_{N}}{N+\gk_{N}}\cdot \left(1-\frac{2cm_{N}(2)N}{\gk_{N}^{2}}\right)^{2} e^{-2\gk_{N}L(\theta \rho_{j})}
    \]
    and observe that 
    \[
    \Theta_{4}(N,j)\geq \exp\left(-4\gk_{N}L(\theta \rho_{j})\right)\cdot \left(1-\frac{2cm_{N}(2)N}{\gk_{N}^{2}}\right)^{2}
    \]
    for $N$ sufficiently large. 
\end{proof}

Again, having bounded the solitonic and free contributions from a given shell below, we are now ready to establish the free energy lower bound. Compared to the upper bound, the process is much easier, as it only involves restricting to the appropriate mass shell.
\begin{proof}[Proof of Lemma \ref{lem:ZLB}]
    We remind the reader that we are taking $j=i_{*}-1$. All we need to do is restrict the sum~\eqref{eq:sumi} to the $i$ that yields $\rho_{i}$ closest to the value of $a_{\star}$. Let $i_{n}$ be a sequence such that $\rho_{i_{n}}\to a_{\star}$ as $n\to \infty$. Note, by the definition of $i_{*}$, we have
    \[
        \rho_{j}=\frac{2j\gk_{N}+\gk_{N}}{2N}=\frac{2i_{*}\gk_{N}-2i\gk_{N}+\gk_{N}}{N}= \frac{i_{*}\gk_{N}+\gk_{N}}{N}-\rho_{i}\leq 1-\rho_{i}.
    \]
    Combining,
    \[
        Z_{N}\geq \exp(-12dC_{d} (\log N)^{d})\cdot \Theta_{3}(N,i)\cdot \Theta_{4}(N,i_{*}-i)\cdot \left(\frac{\pi}{\theta}\right)^{N}\cdot  \exp\biggl(-N\nu^{-1}\theta I(\nu\rho_{i})-NW(\theta(1-\rho_{i}))\biggr).
    \]

    Now note that $|i_{n}-a_{*}|\leq \gk/N$, and we know that for a fixed $M>0$, $a_{*}\leq a_{M}<1$ by Lemma~\ref{lem:abound}. This tells us two things, firstly that  $\Theta_{4}$ may be removed of its $j$ dependence, and  secondly on the interval $[0,a_{M}]$ the function
    \[
        W(\theta(1-a))+\theta\nu^{-1} I(\nu a)
    \]
    is Lipschitz, with Lipschitz constant depending only on $p$, $\theta$ and $\nu$. Let the Lipschitz constant be $C_{2}$.
    Defining
    \[
        \Theta_{I}(n)=\Theta_{3}(n)\cdot \Theta_{4}(n)\cdot \exp(-4d\tau(\log n)^{d})\cdot \exp(-(C+C_{2})\gk_{N})
    \]
    completes the proof. Clearly, we must have appropriately chosen constant
    \[
    \Theta_{I}(N)\geq \exp\left(-C \gk_{N} -N\gamma_{N} \right) \cdot \left(1-\frac{2cm_{N}(2)N}{\gk_{N}^{2}}\right)^{2}
    \]
    where now $\gamma_{N}$ is defined for the optimizing $a_{\star}$. 
\end{proof}
Thus, we have assembled all the requisite ingredients for Theorem~\ref{thm:free}.

%%%%%%%%%%%%%%%%%%%%%%%%%%%%%%%%%%%%%%%%%%%%%%%%%%%%%%%%%%%%%%%%%%%%%%
\section{Analysis of Phases}\label{sec:phase}
%%%%%%%%%%%%%%%%%%%%%%%%%%%%%%%%%%%%%%%%%%%%%%%%%%%%%%%%%%%%%%%%%%%%%%
Recall that the optimizing mass fractions $a_{\star}$ characterize the phase behavior. The optimizing values $a_{\star}$ tell us the proportions of the mass that are energetically favourable to allocate to the soliton contribution. We begin this section with the following effective bound on $a_{\star}$, which was used crucially in proving convergence of free energy. Further, this establishes that it is never energetically favourable to have all mass allocated to the soliton. Recall the expression for the limiting free energy
\[
    F(\theta,\nu)=\log\frac{\pi}{\theta}-  \inf_{0<a<1} \left(  W(\theta(1-a)) + \frac\theta{\nu} I(a\nu)\right).
\]
\begin{lem}\label{lem:abound}
    We have $\mathscr{M}(\theta, \nu)\subseteq[0,a_M]$ where
\[
 a_M:=1-\exp(-\theta(2d-I(\nu)/\nu)). 
\]
\end{lem}
\begin{proof}
    Using the fact that $J(r)=I(r)/r$ is decreasing and properties of $\widehat{W}$ from Lemma~\ref{lem:propW}, we get that
    \begin{align*}
    W(\theta(1-a))-W(\theta)+\theta I(a\nu)/\nu 
    &= \widehat{W}(\theta(1-a))-\widehat{W}(\theta)-\ln(1-a)+a\theta J(a\nu)\\
    &\ge -2da\theta -\ln(1-a)+a\theta J(\nu)\\
    &= -\ln(1-a)-a\theta (2d-J(\nu))>0
    \end{align*}
    if $a>1-\exp(-\theta(2d-J(\nu)))$. In particular, we have
    \[
    \inf_{a>1-\exp(-\theta(2d-I(\nu)/\nu))}\left( W(\theta (1-a)) + \theta I(a\nu)/\nu\right)> W(\theta).
    \]
    This completes the proof.
\end{proof}

It is clear that $\mathscr{M}$ is compact and closed in $[0,1)$ by the continuity of 
$$
a\mapsto G_{\theta,\nu}(a):=W(\theta(1-a)) + \theta I(a\nu)/\nu.
$$ 
We begin with some simple results to show that there are regions each of the phases are achieved. The easiest technique is to note that there are intervals where $W$ and $I$ are constant. We define
\[
    a_{W}(\theta):=1-\frac{C_{d}}{\theta}
\text{ and }
    a_{I}(\nu):=\frac{R_{p}}{\nu}.
\]
Clearly, $W(\theta(1-a))=C_{d}$ for all $a\leq a_{W}$ and $I_{\nu}(a)=0$ for all $a\leq a_{I}$.
\begin{lem}[Phase 1]\label{lem:Phase1exist}
    Let $\nu\leq R_{p}$. Then  $0\in \cF$.
\end{lem}
\begin{proof}
    Since $v\leq R_{p}$, then $a_{I}\geq 1$, and thus $I_{\nu}(a)=0$ for all values of $a$. It follows that $G_{\theta,\nu}$ is non decreasing in $a$, and thus $0 \in \cF$.
\end{proof}
This lemma shows that the dispersive phase is indeed achieved. We show that the soliton phase is achieved as well. 
\begin{lem}[Phase 2]\label{lem:Phase2exist}
    Let $\nu$ and $\theta$ be such that
    $
        \frac{C_{d}}{\theta} +\frac{R_{p}}{\nu} < 1.
    $
    Then $0\notin \mathscr{M}(\theta,\nu)$.
\end{lem}
\begin{proof}
    With $\nu$ and $\theta$ as above, we know that $a_{I} < a_{W}$. On the interval $[0,a_{W}]$, the function $G_{\theta,\nu} $ is non-increasing and is in fact strictly decreasing on the interval $[a_{I},a_{W}]$. Thus, there is an  $a\in [a_{I},a_{W}]$ such that $G_{\theta,\nu}(a)< W(\theta)=G_{\theta,\nu}(0)$.
\end{proof}

%%%%%%%%%%%%%%%%%%%%%%%%%%%%%%%%%%%%%%%%%%%%%%%%%%%%%%%%%%%%%%%%%%%%%
\subsection{Existence of Transition Curve}\label{sec:transition}
%%%%%%%%%%%%%%%%%%%%%%%%%%%%%%%%%%%%%%%%%%%%%%%%%%%%%%%%%%%%%%%%%%%%%%
We have shown that each of the defined phases may be realized, it remains to be shown that the regions are separated by a curve. We also need to establish the continuity of the curve.  
\begin{lem}\label{lem:curve1}
Let $(\theta,\nu)$ be such that $0\in \mathscr{M}(\theta,\nu)$. We say $(\theta_{1},\nu_{1})\peq (\theta,\nu)$ if $\theta_{1}\leq \theta$ and $\nu_{1} \leq \nu$. Now let 
\[
(\theta_{1},\nu_{1})\peq (\theta,\nu) \peq (\theta_{2},\nu_{2}).
\]
\end{lem}
If $0 \in \mathscr{M}(\theta,\nu)$, then $0 \in \mathscr{M}(\theta_{1},\nu_{1})$. Conversely, if $0 \notin \mathscr{M}(\theta,\nu)$, then $0\notin \mathscr{M}(\theta_{2},\nu_{2})$. 
\begin{proof}
We know that $0 \in \mathscr{M}(\theta,\nu)$ if and only if for all $a\in [0,1)$, 
\begin{align}\label{eq:trivialmin}
W(\theta)\leq W(\theta(1-a))+\theta \nu^{-1}I(\nu a). 
\end{align}
Using the fact that $W'(b)=L(b)$, we may rewrite~\eqref{eq:trivialmin}
\begin{align}\label{eq:trivialmin2}
    I_{\nu}(a)+\int_{(1-a)}^{1}L(\theta s)ds  \geq 0
\end{align}
By Lemma~\ref{lem:Ideriv}, decreasing $\nu$ increases $I_{\nu}(a)$, preserving the inequality. In addition, decreasing $\theta$ increases $L(\theta s)$, still preserving the inequality. As for the converse, observe that $0 \notin \mathscr{M}$ if and only if we have an $a_{1}$ such that 
\begin{align}\label{eq:nontrivialmin}
I_{\nu}(a_{1})+\int_{(1-a_{1})}^{1}L(\theta s)ds  <0.
\end{align}
Again, by Lemma~\ref{lem:Ideriv}, increasing $\nu$ decreases $I_{\nu}(a)$, and increasing $\theta$ decreases $L(\theta s)$.
\end{proof}
Lemma~\ref{lem:curve1} is adequate to establish not only the existence of the transition curve but the monotonicity as well. We remark now the following useful consequence of the relation~\eqref{eq:nontrivialmin}, the set of $(\theta,\nu)$ corresponding to the soliton phase is open in the plane $[R_{p},\infty)\times [0,\infty)$. 
\begin{cor}
Let $\theta_{c}:(R_{p},\infty) \to (0,\infty)$ be the measurable function defined as
\[
\theta_{c}(\nu):=\sup\{\theta> 0: 0\in \cF(\theta,\nu)\}. 
\]
By definition, $\theta_{c}$ is our transition curve. We have that $\theta_{c}$ is non-increasing and continuous. 
\end{cor}
\begin{proof}
By Lemma~\ref{lem:Phase2exist}, we know that 
\[
\theta_{c}(\nu)\leq \frac{C_{d}\nu}{\nu-R_{p}}<\infty 
\]
for all $\nu$ under consideration. Not only does this tell us that $\theta_{c}$ is finite, it also tells us that it may be evaluated at every $\nu$ under consideration. Next, suppose we have that $\nu_{1}<\nu_{2}$ such that $\theta_{c}(\nu_{1})<\theta_{c}(\nu_{2})$. Observe that this contradicts Lemma~\ref{lem:curve1}. Since we have proved that $\theta_{c}$ is non-increasing, it further implies that on any compact interval $[\nu_{1},\nu_{2}]\subset (R_{p},\infty)$ there are at most countably many discontinuities. Let $\nu_{3}$ be a point of discontinuity. We have that 
\[
\theta_{1}:=\lim_{\nu \downarrow \nu_{3}}\theta_{c}(\nu) \leq \theta_{2}:= \theta_{c}(\nu_{c})\leq \theta_{3}:=\lim_{\nu \uparrow \nu_{3}} \theta_{c}(\nu), 
\]
Suppose
\[
\theta_{2}<\theta_{3}.
\]
We may find an $\eps>0$ such that $\theta_{2}+\eps<\theta_{3}$. By hypothesis,
\[
0\notin \cF(\nu_{3},\theta_{2}+\eps). 
\]
However, note that every open ball centered at $(\nu_{3},\theta_{2}+\eps)$ intersects $[R_{p},\nu_{3})\times [0,\theta_{3})$ on which $0\in\cF(\theta,\nu)$, which contradicts the fact that the solitonic region is open in the plane, and we must have  $\theta_{2}=\theta_{3}$. To actually establish continuity, we note that the boundary between the solitonic region and the plane may be characterized as all $(\theta,\nu)$ such that~\eqref{eq:trivialmin} holds, and we have an $a_{\star}\geq a_{I}$ such that 
\[
I_{\nu}(a_{\star})+\int_{(1-a\star)}^{1}L(\theta s)ds=0.
\]
Let $\eps>0$ be such that $\theta_{1}+2\eps<\theta_{2}$. We have that $(\nu,\theta_{1}+\eps)$ is on the boundary, as for any $\nu'>\nu$ we enter the soliton phase. Further, we know that
\[
\theta_{1}+\eps<\theta_{2}\leq \frac{C_{d}\nu}{\nu-R_{p}}. 
\]
Thus, $a_{\star}>a_{W}$ and
\[
\int_{(1-a_{\star})}^{1} L((\theta_{1}+\eps)s)ds>\int_{(1-a_{\star})}^{1}L((\theta_{1}+1.5\eps)s)ds. 
\]
Combined with the fact that $I_{\nu}(a_\star)<0$, this tells us that $(\nu, \theta_{1}+1.5\eps)$ is in the soliton phase, further implying that $\theta_{2}\leq \theta_{1}+\eps$ which is a contradiction. Thus, we must have that $\theta_{1}=\theta_{2}$, which verifies continuity.  
\end{proof}
\begin{rem}
The phase curve could be equivalently defined with $\nu_{c}$ as a function of $\theta$, which would be the exact inverse of the function $\theta_{c}$ obtained above. The same properties of monotonicity and continuity can be verified for $\nu_{c}$, which then tells us that $\theta_{c}$ must be strictly decreasing. 
\end{rem}

%%%%%%%%%%%%%%%%%%%%%%%%%%%%%%%%%%%%%%%%%%%%%%%%%%%%%%%%%%%%%%%%%%%%%%
\section{Discussion on the Typical Function in the Subcritical Phase}\label{sec:typical}
%%%%%%%%%%%%%%%%%%%%%%%%%%%%%%%%%%%%%%%%%%%%%%%%%%%%%%%%%%%%%%%%%%%%%%
In this section, we will work with the condition that $1< p < 5+4\sqrt{2}$, and establish some properties of a typical function sampled from the measure $\mu_{N}$ in the dispersive phase. Recall that this corresponds to $(\theta,\nu)\in \cD$. We will also further make the assumption that $\theta < C_{d} $, for reasons that will be made clear. In this section, it will be convenient to work with the following rescaled version of the measure $\mu_{N}$, which is equivalent due to Lemma~\ref{lem:scaling}. 
\[
\tilde{\mu}_{N}(\cA):= \frac{1}{\tilde{Z}(\theta, \nu)}\int_{\cA}\exp\left(-\norm{\nabla \psi}^{2}_{2} +(\theta N)^{-\frac{p-1}{2}}\frac{2 \nu }{p+1}\norm{\psi}_{p+1}^{p+1}\right)\1_{\norm{\psi}^{2}_{2}\leq \theta  N} \cdot \vd \psi
\]
We will define the reference measure $\mu_{\text{ref}}$, and will then regard $\tilde{\mu}_{N}$ as an exponential tilt with respect to the nonlinearity.  We define
\[
\mu^{\theta}_{\text{ref}}(\cA):=\frac{1}{Z_{\text{ref}}(\theta)}\int_{\cA}\exp(-\norm{\nabla \psi}_{2}^{2})\vd \psi. 
\]
We begin by explicitly demonstrating the sense in which this measure is close to the massive Gaussian free field. 
\begin{thm}\label{thm:AC0}
Let $y_{N}$ be defined by the equation $K'_{N}(y_{N})=\theta$, and let $Z^{y_{N}}(\theta)$ denote the partition function of the massive free field on the torus, given by 
\[
Z^{y_{N}}=\frac{\pi^{N}}{\det(y_{N}-\gD)}. 
\]
Then we have a constant $C$ depending on $\theta$ such that 
\[
\lim_{N\to \infty} \frac{e^{-N\theta}\cdot Z^{y_{N}}}{\sqrt{N}Z_{\text{ref}}(\theta)} =C.
\]
\end{thm}
\begin{proof}
We have that 
\begin{align*}
1
&=\frac{1}{Z_{\text{ref}}(\theta)} \int_{\norm{\psi}_{2}^{2}\leq \theta N} \exp(-\norm{\nabla \psi}_{2}^{2})\vd \psi\\ 
&= \frac{e^{N\theta}\cdot Z^{y_{N}}}{Z_{\text{ref}}(\theta)}\E \exp \left(\norm{\Psi^{y_{N}}}_{2}^{2}-N\theta \right)\cdot \1_{\norm{\Psi^{y_{N}}}_{2}^{2}\leq N\theta}
\end{align*}
It, therefore, suffices to prove that 
\[
\sqrt{N}\cdot \E \exp \left(\norm{\Psi^{y_{N}}}_{2}^{2}-N\theta \right)\cdot \1_{\norm{\Psi^{y_{N}}}_{2}^{2}\leq N\theta}
\]
converges to a constant. We have that 
\[
\exp \left(\norm{\Psi^{y_{N}}}_{2}^{2}-N\theta \right) = \int^{\infty}_{0} e^{-s}\cdot  \1_{s\geq N\theta-\norm{\psi^{y_{N}}}_{2}^{2}} \cdot ds,
\]
and thus 
\[
 \E \exp \left(\norm{\Psi^{y_{N}}}_{2}^{2}-N\theta \right)\cdot \1_{\norm{\Psi^{y_{N}}}_{2}^{2}\leq N\theta} = \int_{0}^{\infty} e^{-s} \cdot \dP\left(s\leq N\theta-\norm{\Psi}_{2}^{2}\leq 0\right) ds. 
\]
The problem now comes down to analyzing
\[
\sqrt{N}\cdot \dP\left(s\leq N\theta-\norm{\Psi^{y_{N}}}_{2}^{2}\leq 0\right). 
\]
Note that we may equivalently consider 
\begin{align}\label{eq:LLT1}
\sqrt{N}\cdot \dP \left(\frac{s}{\sqrt{N}}\leq \frac{N\theta-\norm{\Psi^{y_{N}}}_{2}^{2}}{\sqrt{N}} \leq 0\right).
\end{align}
Recall that $\norm{\Psi^{y_{N}}}$ is expressible as a sum of independent random variables, we have that
\[
\norm{\Psi^{y_{N}}}_{2}^{2}\equald \sum_{\mvk \in [n]^{d}} \frac{1}{y_{N}+\gl_{\mvk}} X_{\mvk}. 
\]
Further, $N\theta -\norm{\Psi^{y_{N}}}_{2}^{2}$ is a centered sum by definition, and each random variable in the sum has a uniformly bounded third moment. The Berry-Esseen Theorem then tells us that~\eqref{eq:LLT1} is bounded above by a constant. The local limit theorem yields convergence. 
\end{proof}
Theorem~\ref{thm:AC0} gives us a means of evaluating probabilities of events with respect to the reference measure, in terms of the MGFF, where of course we pay a penalty due to the mass restriction. In the next sequence, we will examine the effect of exponentially tilting the massive free field with respect to the nonlinearity. 
\begin{thm}\label{thm:AC1}
Let $\Psi^{y}$ be distributed according to the massive Gaussian Free Field on $\dT^{d}$ with mass parameter $y$. Then we have a constant $C$ depending on $\nu$, $y$ and $d$ such that
\[
\E\exp\left(\frac{2\nu N^{\frac{1-p}{2}}}{p+1}\norm{\Psi^{y}}^{p+1}_{p+1}\right)\cdot \1_{\norm{\Psi}_{2}^{2}\leq N} \leq e^{C}\]

\end{thm}
The mass constraint is a significant obstacle, as it is not expressible as a product, nor is it smooth for us to apply some standard techniques to deal with Gaussian integration. We resolve both these issues with the following upper bound.  We introduce the following auxilary function in order to deal with this issue:
\begin{align}
h(x):= \frac{ 2x^{p+1}}{1+x^{p-1}}.
\end{align}
Observe that when $x<1$, we have that $x^{p+1}<h(x)$, and when $x>1$, $h(x)< 2x^{2}$. Crucially, $h$ is smoothly differentiable. 
\begin{lem}\label{lem:hconvex}
For $p<5+\sqrt{32}$ we have that 
\[
h''(x)+x^{-1}\cdot h'(x)>0
\]

\end{lem}
\begin{proof}
It is equivalent to verify that $(xh'(x))'>0$. We merely carry out the calculation. We have
\[
(xh'(x))=\frac{2(p+1)x^{p}}{1+x^{p-1}}-\frac{2(p-1)x^{2p-1}}{(1+x^{p-1})^{2}} +\frac{(p-1)(p+1)x^{p}}{(1+x^{p-1})^{2}} -\frac{2(p-1)^{2}x^{2p-1}}{(1+x^{p-1})^{3}}.
\]
Adding and simplifying, we find that the numerator is given by 
\[
x^{p}\left(4x^{2p-2}+(3+6p-p^{2})x^{p-1}+(p+1)^{2}\right),
\]
This equation has no real roots when 
$
1<p<5+\sqrt{32}.
$
\end{proof}
Merely by the definition of $h$, it follows that
\begin{align}\label{eq:hbound}
\E \exp\left(\frac{2\nu N^{\frac{1-p}{2}}}{p+1}\norm{\Psi}^{p+1}_{p+1}\right)\cdot \1_{\norm{\Psi}_{2}^{2}\leq  N} \leq \E \prod_{{\mvx} \in \dT_{n}^{d}} \exp\left(\frac{2\nu N }{p+1} h\left(\frac{|\Psi_{{\mvx}}|}{\sqrt{N}}\right)\right)
\end{align}
The intuition now is to work with the right side of~\eqref{eq:hbound} and attempt to use the decay of correlations to separate as a product of expectations, essentially comparing to the i.i.d. situation. We will first establish a counterpart of Theorem~\ref{thm:AC1} for an i.i.d. standard complex Gaussian vector $\{\Phi_{{\mvx}}\}_{{\mvx}\in \dT^{d}}$, and then use Gaussian interpolation to compare the expectations under the MGFF and i.i.d. cases. 
\begin{lem}\label{lem:predecoupling}
Let $\Phi$ be an i.i.d standard complex Gaussian vector. Then for $p>3$ and \[\nu < (p+1)/8,\] 
we have a constant $C$ depending on $\nu$ and $p$ such that  
\[
\E \exp \left( \frac{2 \nu N }{p+1} h\left( \frac{|\Phi_{{\mvx}}|}{\sqrt{N}}\right) \right) \leq e^{C}.
\]

\end{lem}

\begin{proof}
We may immediately express the expectation as a product using the independence, and noting that $|\Phi_{x}|^{2}$ is distributed according to exponential $1/2$, we have 
\begin{align}\label{eq:ACeq1}
\E \exp \left( \frac{2 \nu N }{p+1} h\left( \frac{|\Phi_{{\mvx}}|}{\sqrt{N}}\right) \right) = \left(\frac{1}{2} \int_{0}^{\infty} \exp \left(\frac{2\nu N}{p+1} h\left(\sqrt{ N^{-1}t}\right)\right) \cdot e^{-t/2} dt \right)^{N}.
\end{align}
We split the integral into separate regions and bound their contributions individually. Let $\ga = 1-4/(p+1)$. We will consider the intervals $[0, N^{\ga})$ and $[N^{\ga}, \infty)$. This is where $p>3$ becomes important, as we want $N^{\ga} \to \infty$ as $N\to \infty$. For the first interval,  
\begin{align*}
\frac{1}{2}\int_{0}^{N^{\ga}}\exp\left( \frac{2\nu N}{p+1} h\left(\sqrt{\frac{t}{N}} \right)\right)e^{-\frac{t}{2}}dt &\leq \exp\frac{4\nu}{(p+1)N}\cdot \frac{1}{2}\int_{0}^{N^{\ga}}e^{-t/2}dt \\ &\leq \exp\frac{2\nu(a+1)}{a(p+1)N}
\end{align*}
In this bound, we have used the fact that on the interval $[0, N^{\ga-1})$, $h(x)\leq 2x^{p+1}$. Now take a fixed $\eps>0$, we know that for $N$ sufficiently large, 
\[
\exp\frac{2\nu(a+1)}{a(p+1)N}\leq 1 + (1+\eps)\frac{2\nu(a+1)}{a(p+1)N}.
\]
The case to consider is $t \in [N^{\ga}, \infty)$. On this interval, we may use the bound $h(x) \leq (a+1)x^{2}$ to obtain 
\[
\int_{N^{\ga}}^{N} \exp \left(\frac{2\nu N}{p+1} h\left(\sqrt{\frac{t}{N}} \right) \right) e^{-\frac{t}{2}}dt  \leq  N\exp\biggl(\bigl(-\frac{1}{2} +\frac{4 \nu}{p+1}\bigr)N^{\ga-1}\biggr).
\]
This decays to zero rapidly as $N\to \infty$, so long as 
$
\nu < \frac{1}{8}(p+1).
$
Thus, combining we obtain that the following is an upper bound  for~\eqref{eq:ACeq1}
\[
\left(1 + \frac{6\nu (1+\eps)}{(p+1)N}\right)^{N}\leq \exp\left(\frac{6\nu (1+\eps)}{ p+1}\right).
\] 
\end{proof}

To use Lemma~\ref{lem:predecoupling} to prove Theorem~\ref{thm:AC1}, we need one more ingredient to compare the expectation w.r.t.~the massive Gaussian Free Field and the i.i.d.~case. We may use Gaussian interpolation and correlation decay to accomplish this. 
\begin{lem}\label{lem:decoupling}
Let $\Psi^{y}$ denote the MGFF on $\dT^{d}_{n}$, let $\Phi$ denote a standard complex Gaussian vector, and let $C_{y}$ be a constant such that 
\begin{align}\label{eq:varub}
C_{y}\geq 2\sum_{{\mvx}_{2}} (y-\gD)^{-1}_{{\mvx}_{1} {\mvx}_{2}}. 
\end{align}
Then we have that 
\[
\E \prod_{x \in \dT_{d}}\exp \left( \frac{2 \nu N }{p+1} h\left( \frac{|\Psi^{y}_{{\mvx}}|}{\sqrt{N}}\right) \right) \leq \left(\E \exp \left( \frac{2 \nu N }{p+1} h\left( \frac{\sqrt{C_{y}}|\Phi_{{\mvx}}|}{\sqrt{N}}\right) \right)\right)^{N} . 
\]
\end{lem}
Lemma~\ref{lem:decoupling} hinges on the application of Gaussian interpolation. We state the version used below, easily adapted from the real-valued case. 
\begin{lem}\label{lem:GIBP}
Let $\Psi^{0}$ and $\Psi^{1}$ be independent, centered, complex-valued Gaussian processes on $\dT^{d}_{n}$, with covariance matrices $G_{0}$ and $G_{1}$ respectively. Let $r:\dC^{N}\to \dR$ be integrable w.r.t.~the laws of both $\Psi^{0}$ and $\Psi^{1}$. Let $\Psi^{t}:=\sqrt{1-t}\cdot \Psi^{0} + \sqrt{t}\cdot \Psi^{1}$, and define  
$R(t):=\E r(\Psi^{t})$. Then we have that 
\[
R'(t)=\sum_{{\mvx}_{1},{\mvx}_{2} \in \dT^{d}_{n}}\bigl( G_{1}({\mvx}_{1},{\mvx}_{2})-G_{0}({\mvx}_{1},{\mvx}_{2})\bigr)\cdot \E \frac{\partial^{2}}{\partial \psi_{{\mvx}_{1}} \partial \psi_{{\mvx}_{2}}} r(\Psi^{t}). 
\]
\end{lem}
It is clear that $R(0)$ denotes the expectation of $r$ under the law of $\Psi^{0}$, and $R(1)$ the expectation under the law of $\Psi^{1}$. Thus, if $R'$ can be shown to be nonnegative, it is clear that $R(1)\geq R(0)$. We are now ready to prove Lemma~\ref{lem:decoupling}, and then obtain Theorem~\ref{thm:AC1} for free. 
\begin{proof}[Proof of Lemma~\ref{lem:decoupling}] We apply Lemma~\ref{lem:GIBP} with $\Psi^{0}=\Psi^{y}$, $\Psi^{1}=\sqrt{C_{y}}\ \Phi$ as defined in~\eqref{eq:varub}, and
\[
r(\psi)=\exp\left(\frac{2\nu N}{p+1} \sum_{{\mvx} \in \dT^{d}} h\left(\frac{|\psi_{{\mvx}}|}{\sqrt{N}}\right)\right).
\]
Observe that for any smooth function $f:\dR \to \dR$, we have 
\[
(\exp(f))'=\exp(f)\cdot f' \text{ and } (\exp(f))''=\exp(f)\cdot (f''+(f')^{2}).
\]
In our cases, $\Re{\Psi}$ and $\Im{\Psi}$ are independent copies of one another, and therefore any calculations done for the real part are exactly replicated for the imaginary part. For convenience, we introduce the notation $\fJ$ which denotes either $\Re$ or $\Im$. With our choice of $r$, we have 
\[
\frac{\partial^{2}}{\partial \fJ_{1} \psi_{{\mvx}_{1}} \partial \fJ_{2} \psi_{{\mvx}_{2}}} r(\psi)= N \cdot\left(\frac{2\nu }{p+1}\right)^{2}\cdot h'\left(\frac{|\psi_{{\mvx}_{1}}|}{\sqrt{N}}\right)\cdot\frac{\partial |\psi_{\mvx_{1}}|}{\partial \fJ_{1} \psi_{\mvx_{1}}}\cdot  h'\left(\frac{|\psi_{{\mvx}_{2}}|}{\sqrt{N}}\right)\cdot \frac{\partial |\psi_{\mvx_{2}}|}{\partial \fJ_{2} \psi_{\mvx_{2}}} \cdot r(\psi) .
\]
For the non mixed second derivative, we have
\begin{align*}
\frac{\partial^{2}}{\partial \fJ \psi_{{\mvx}}^{2}}r(\psi)
&=N \cdot r(\psi) \cdot \left(\frac{2\nu}{p+1} \cdot h'\left(\frac{|\psi_{\mvx}|}{N}\right)\cdot \frac{\partial |\psi_{\mvx}|}{\partial \fJ \psi_{\mvx}}\right)^{2}
+ r(\psi)\cdot \frac{2\nu}{p+1}\cdot \left(  \frac{\partial |\psi_{\mvx}|}{\partial \fJ \psi_{\mvx}}\right)^{2}h''\left(\frac{|\psi_{{\mvx}}|}{N}\right)
\\  
&\qquad\qquad+ N\cdot r(\psi) \cdot\left(\frac{2\nu}{p+1} \right)h'\left(\frac{|\psi_{{\mvx}}|}{N}\right)\cdot \frac{\partial^{2}|\psi_{\mvx}|}{\partial \fJ \psi_{\mvx}^{2}}.  
\end{align*}
We focus our attention on the latter two terms, that is 
\begin{align}\label{eq:rint}
r(\psi)\cdot \frac{2\nu}{p+1}\cdot \left(  \frac{\partial |\psi_{\mvx}|}{\partial \fJ \psi_{\mvx}}\right)^{2}h''\left(\frac{|\psi_{{\mvx}}|}{N}\right)
+ \sqrt{N}\cdot r(\psi) \cdot\left(\frac{2\nu}{p+1} \right)h'\left(\frac{|\psi_{{\mvx}}|}{N}\right)\cdot \frac{\partial^{2}|\psi_{\mvx}|}{\partial \fJ \psi_{\mvx}^{2}}
\end{align}
Evaluating the sum of~\eqref{eq:rint} over both cases of $\fJ$, that is $\Re$ or $\Im$, we obtain
\[
r(\psi)\cdot \left(\frac{\partial^{2}}{\partial|\psi_{\mvx}|^{2}}+\frac{1}{|\psi_{\mvx}|}\frac{\partial}{\partial |\psi_{\mvx}|}\right)h\left(\frac{|\psi_{\mvx}|}{\sqrt{N}}\right).
\]
Using Lemma \ref{lem:hconvex}, we may conclude that 
\[
\frac{\partial^{2}}{\partial \Re\psi_{\mvx}^{2} }r(\psi)+\frac{\partial^{2}}{\partial \Im \psi_{\mvx}^{2}}r(\psi) - \left(\frac{\partial}{\partial \Re \psi_{\mvx}} r(\psi)\right)^{2} -\left(\frac{\partial}{\partial \Im \psi_{\mvx}} r(\psi)\right)^{2}\geq 0
\]
Combining the inequality $a^{2}+b^{2}>2ab$ with~\eqref{eq:varub}, we have that 
\[
R'(t)\geq \sum_{x_{1}}\left( C_{y} -\sum_{{\mvx}_{2}} (y-\gD)^{-1}_{{\mvx}_{1} {\mvx}_{2}}\right)\E\left(\frac{\partial^{2}}{\partial \Re \psi_{{\mvx}}^{2}}+\frac{\partial^{2}}{\partial \Im \psi_{{\mvx}}^{2}}\right)r(\psi) \geq 0. 
\]
This completes the proof.
\end{proof}

All parts are established to prove Theorem~\ref{thm:AC1}.
\begin{proof}
A detailed proof is omitted; all that is required is to combine Lemma~\ref{lem:decoupling} with a rescaled version of Lemma~\ref{lem:predecoupling}.
\end{proof}

%%%%%%%%%%%%%%%%%%%%%%%%%%%%%%%%%%%%%%%%%%%%%%%%%%%%%%%%%%%%%%%%%%%%%%
\section{Proofs of Main  Theorems}\label{sec:pfs}
%%%%%%%%%%%%%%%%%%%%%%%%%%%%%%%%%%%%%%%%%%%%%%%%%%%%%%%%%%%%%%%%%%%%%%
In the sections prior, we have all the pieces required to prove our main results. In this section, we tie them all together. We begin with the proof of free energy convergence. All that needs to be done is to combine the upper and lower bounds found in Section \ref{sec:feconv}

\subsection{Proof of Theorem~\ref{thm:free}}
Essentially, the proof is an immediate  corollary of Lemmas~\ref{lem:ZUB} and~\ref{lem:ZLB}. It suffices to specify $\gk_{N}$ and $s_{N}$, and provide a rate of convergence for $\gc_{N}$. The sequence $\gk_{N}$ is chosen such that the mass shells are of adequate thickness such that the required concentration of mass of the Gaussian free field holds. A larger $\gk_{N}$ yields a better concentration bound but a worse rate of convergence. It is necessary that 
\[
\lim_{N\to \infty} \frac{\gk_{N}}{\sqrt{m_{N}(2)N}}=\infty \text{ and } \lim_{N\to \infty} \frac{\gk_{N}}{N} = 0.
\]

 As for $s_{N}$, which governs the size of the concentrated region in the upper bound, observe that $s_{N}$ needs to be chosen such that the corresponding $U$ will never contain nontrivial cycles, $s_{N} \to 0$ as $N\to \infty$, and the corresponding $U$ are sufficiently small so as to allow concentration. It suffices to consider
\begin{align}\label{eq:snrate}
s_{N}= \frac{1}{\log N}.
\end{align}
Finally, as for $\gc_{N}$, this depends on the optimizing values  $a$. We note that if $0$ is optimizing, then by Lemma \ref{lem:minrate2}, we have that 
\[
\gc_{N}\leq \frac{2d\gk_{N}}{N \log N}.
\]
On the contrary, if we have a non trivial minimezer $a_{\star}$, then we also know that $\nu a_{\star}>R_{p}$, and thus by Lemma \ref{lem:minrate1} we have that 
\[
\gc_{N} \leq C_{1}(\nu) (\nu a_{\star})^{p^{2}-1}\exp(-C_{2}(\nu)\cdot N^{1/d}). 
\]
Thus, the product of all the  error terms arising in Lemmas \ref{lem:ZUB} and \ref{lem:ZLB} can be bounded above by 
\[
e^{C_{1}{\gk_{N}}}\cdot \left(1-C\cdot \frac{Nm_{N}(2)}{\gk_{N}^{2}}\right)^{2}.
\]
Using the fact that $m_{N}(2)=N^{-1+4/d}$, we find that the best choice of $\ga$ for $\gk_{N}=N^{\ga}$ is given by 
\[
\ga=\frac{1}{3} +\frac{4}{3d}
\] \qed

Next, the characterization of the phase transition curve. These are essentially the finishing touches required to extract Theorem \ref{thm:ptcurve} from Section \ref{sec:phase}.
%%%%%%%%%%%%%%%%%%%%%%%%%%%%%%%%%%%%%%%%%%%%%%%%%%%%%%%%%%%%%%%%%%%%%%
\subsection{Proof of Theorem~\ref{thm:ptcurve}}
We address the four parts Theorem~\ref{thm:ptcurve} individually.
\begin{enumeratea}
\item Immediate corollary of Lemma \ref{lem:Phase1exist}.
\item This is exactly established in Lemma \ref{lem:curve1}.

\item We begin with the bound and asymptotic for $\nu\downarrow R_{p}$. Indeed, Lemma \ref{lem:Phase2exist} may be rephrased as 
\[
\theta_{c}(\nu)\leq C_{d}\cdot \frac{\nu}{\nu- R_{p}}.
\]
Further, note that we have a non trivial minimizer $a_{\star}$ iff we have an $a$ such that $G_{}(a)<0$. In particular, thus, if $(\theta,\nu)$ corresponds to the dispersive phase, we must have that $G_{}(a)\geq 0$ for $a \in [0,1)$. Now consider the curve given by 
\[
\frac{R_{p}}{\nu}+\frac{C_{d}}{\theta} =(1+\eps)
\]
We know that the function $J$ has a derivative, moreover $J'=0$ for all $a<R_{p}$. As $\theta \uparrow \infty$, $\nu \downarrow (1+\eps)^{-1}\cdot R_{p}$. Thus we must have that along this curve
\[
\lim_{\theta \to \infty} \theta \cdot J(a\nu)=0
\]
for all values of $a\in [0,1)$. Thus, along this curve as $\theta$ increases, $(\nu_{\eps},\theta_{\eps})$ will eventually correspond to the dispersive phase. Since choice of $\eps$ is arbitrary, this verifies that 
\[
\lim_{\nu \downarrow R_{p}}(\nu-R_{p})\theta_{c}(\nu) = C_{d}R_{p}.
\]
Finally, we address the bound and the asymptotic as $\nu \uparrow \infty$. To do this, recall the functions $\widehat{W}$ and $\widehat{J}$ introduced in~\eqref{eq:jjhat} and~\eqref{eq:what} respectively.  We may then write
\begin{align*}
        \log{\theta}
         & +W(\theta(1-a)) + \frac\theta{\nu} I(a\nu)                                 \\
         & = \widehat{W}(\theta(1-a)) + \theta a \widehat{J}(a\nu) - \frac{2\theta}{p+1}a(a\nu\vee R_p)^{(p-1)/2} - \log(1-a).
\end{align*}
We know that the functions $\widehat{W}$ and $\widehat{J}$ are bounded, moreover, we may discard the $\log \theta$ as it is irrelevant to the phase behavior (independent of $a$). Thus, it suffices to characterize the behavior of  
\[
G_{\theta,\nu}(a)= \widehat{W}(\theta(1-a)) + \theta a \widehat{J}(a\nu) - \frac{2\theta}{p+1}a^{(p+1/2)}\cdot \nu^{(p-1)/2} - \log(1-a).
\]
If we take $C\cdot \theta= \nu^{-(p-1)/2}$ for a constant $C$  then observe that then the limit as $\nu \to \infty$ is non negative iff 
\[
C\leq \xi_{p}(0)
\]
for $\xi_{p}$ defined in~\eqref{def:zeta}. This  completes the proof. 
\item It suffices to consider $\nu>R_{p}$ and $\theta>\theta_{c}$. This implies that there is a $a_{\star}>0$ which is the smallest optimizer. Now, suppose that as $\theta \downarrow \theta_{c}$, and we have that $a_{\star}(\theta_{c})\downarrow 0$. Since $a_{\star}(\theta)$ continue to be non trivial minimizers we must still have $I(a_{\star}\nu)<0$, a contradiction as eventually $a_{\star}\nu<R_{p}$. \qed
\end{enumeratea}

We conclude this section by finally proving Theorem~\ref{thm:typical}. Essentially, we need to combine Theorems \ref{thm:AC0} and \ref{thm:AC1}. 
%%%%%%%%%%%%%%%%%%%%%%%%%%%%%%%%%%%%%%%%%%%%%%%%%%%%%%%%%%%%%%%%%%%%%%
\subsection{Proof of Theorem~\ref{thm:typical}}

Let $\cA \subset \ell^{2}(\dT^{d}_{n})$. Using the trivial fact that $\tilde{Z}_{N}\geq Z_{\text{ref}}$,  we may write 
\begin{align}\label{eq:preholder}
\tilde{\mu}_{N}(A)\leq \frac{Z^{y_{N}}e^{N\theta y_{N}}}{\sqrt{N} \cdot Z_{\text{ref}}(\theta)} \E\left(\1_{\cA} \cdot \frac{\exp(\norm{\Psi^{y}}_{2}^{2})}{e^{N\theta y_{N}}} \cdot \exp\left(\frac{\nu (\theta N)^{(1-p)/2} }{p+1}\norm{\Psi^{y_{N}}}_{p+1}^{p+1}\right) \cdot \1_{\norm{\Psi^{y}}_{2}^{2}\leq N\theta} \right). 
\end{align}
Two applications of H\"older's inequality are all that are required now. Indeed, let $\epsilon_{1}$ and $\epsilon_{2}$ be positive real numbers, such that $(1+\epsilon_{1})\nu \theta^{(p+1)/2}$ satisfies the hypothesis of Theorem \ref{thm:AC1}. By the first application, we have that~\eqref{eq:preholder} can be further bounded above by the product of the following
\begin{align}\label{eq:holder1}
    \E\left( \exp((1+\epsilon_{1})\cdot \frac{\nu N^{-(p-1)/2}}{p+1}\cdot \norm{\Psi^{y_{N}}}_{p+1}^{p+1})\cdot \1_{\norm{\Psi^{y}}_{2}^{2}\leq N\theta}\right)^{1/(1+\epsilon_{1})}
\end{align}
and 

\begin{align}\label{eq:holder2}
 \E\left( \exp\left(\frac{1+\epsilon_{1}}{\epsilon_{1}}y_{N}\norm{\psi^{y_{N}}}_{2}^{2} -\frac{1+\epsilon_{1}}{\epsilon_{1}}y_{N}N\theta \right) \cdot \1_{\norm{\Psi^{y}}_{2}^{2}\leq N\theta}\right)^{\epsilon_{1}/(1+\epsilon_{1})}
\end{align}
We, for convenience, denote
\[
\epsilon_{1,2}=\epsilon_{1}+\epsilon_{2}+\epsilon_{1}\epsilon_{2}.
\]
Applying H\"older's inequality again to~\eqref{eq:holder2}, we obtain the upper bound
\begin{align}\label{eq:holder3}
    \E\left( \exp\left(\frac{1+\epsilon_{1,2}}{\epsilon_{1}\epsilon_{2}}y_{N}\norm{\psi^{y_{N}}}_{2}^{2} -\frac{1+\epsilon_{1,2}}{\epsilon_{1}\epsilon_{2}}y_{N}N\theta \right) \cdot \1_{\norm{\Psi^{y}}_{2}^{2}\leq N\theta}\right)^{\epsilon_{1}\epsilon_{2}/(1+\epsilon_{1,2})} \pr\left(\Psi^{y} \in \cA \right)^{\epsilon_{1}/(1+\epsilon_{1,2})}.
\end{align}
Thus, by Theorem \ref{thm:AC1}, we have that~\eqref{eq:holder1} is bounded above by a constant. By Theorem \ref{thm:AC0}, we know that~\eqref{eq:holder3} is bounded above by $N^{-\epsilon_{1}\epsilon_{2}/(2+2\epsilon_{1,2})}$. Combining, we obtain that 
\[
\tilde{\mu}_{N}(A)\leq C\cdot N^{\frac{1+\epsilon_{1}+\epsilon_{2}}{2+2\epsilon_{1,2}}}\cdot \dP\left(\Psi^{y}\in \cA \right)^{\epsilon_{2}/(1+\epsilon_{1,2})}
\]
Indeed, if $\dP(\Psi^{y_{N}}\in \cA) \leq N^{-\ga}$ for some $\ga>0$, we have that 
\[
\tilde{\mu}_{N}(\cA)\leq C \cdot N^{\frac{1+\epsilon_{1}+(1-2\ga)\epsilon_{2}}{2+2\epsilon_{1,2}}}
\]
We have flexibility in choosing $\epsilon_{2}$, and so long as $\ga>1/2$, we may select $\epsilon_{2}$ such that the numerator of the exponent is negative.  \qed

%%%%%%%%%%%%%%%%%%%%%%%%%%%%%%%%%%%%%%%%%%%%%%%%%%%%%%%%%%%%%%%%%%%%%%
\section{Discussion and Further Questions}\label{sec:disc}
%%%%%%%%%%%%%%%%%%%%%%%%%%%%%%%%%%%%%%%%%%%%%%%%%%%%%%%%%%%%%%%%%%%%%%

%%%%%%%%%%%%%%%%%%%%%%%%%%%%%%%%%%%%%%%%%%%%%%%%%%%%%%%%%%%%%%%%%%%%%%
\subsection{On the Question of Multi-Soliton Solutions}\label{sec:multisol}
%%%%%%%%%%%%%%%%%%%%%%%%%%%%%%%%%%%%%%%%%%%%%%%%%%%%%%%%%%%%%%%%%%%%%%

The question of multi-soliton phases is closely related to having multiple minimizers for the variational formula in the soliton phase; \ie\ $0$ is not a minimizer. We conjecture that this is impossible. Indeed, we know that $I(a)$ is concave for large values of $a$. However, we need detailed behavior of the $I$ function near $R_{p}$ to prove this.

%%%%%%%%%%%%%%%%%%%%%%%%%%%%%%%%%%%%%%%%%%%%%%%%%%%%%%%%%%%%%%%%%%%%%%
\subsection{Ergodicity}\label{sec:ergodic}
%%%%%%%%%%%%%%%%%%%%%%%%%%%%%%%%%%%%%%%%%%%%%%%%%%%%%%%%%%%%%%%%%%%%%%
The behavior of the invariant measure under the dynamics is yet to be explored. In~\cite{CK12}, the corresponding analysis was possible due to the mass of typical functions being concentrated at a single lattice site, significantly simplifying computations. This is no longer possible here; for us, the corresponding concentration occurs on a region of size $O(1)$. It should be possible to consider the dynamics introduced in~\cite{OL} with the regime of scaling considered in this article. In~\cite{WEI3}, a hierarchy of local minima for the Hamiltonian~\eqref{def:HAM1} was provided, which would be the collection of metastable states. In addition, the Witten Laplacian approach in~\cite{LEB} can be considered with restriction to the sphere instead of the ball, so as to remove technicalities arising from carrying out Morse theoretic calculations on a manifold with boundary.

%%%%%%%%%%%%%%%%%%%%%%%%%%%%%%%%%%%%%%%%%%%%%%%%%%%%%%%%%%%%%%%%%%%%%%
\subsection{Dimension two analysis} \label{sec:2d}
%%%%%%%%%%%%%%%%%%%%%%%%%%%%%%%%%%%%%%%%%%%%%%%%%%%%%%%%%%%%%%%%%%%%%%
We have explicitly used the finiteness of $C_{d}$ for our maximum bounds. One of the obstacles in extending our results to the two-dimensional case is that $C_{2}=\infty$. We work with a massive field for most parts, so this issue does not arise often. We anticipate that one can work around it. The more serious obstacle is that the techniques used here do not yield adequate mass concentration in two dimensions. This being said, the asymptotics is more interesting in two dimensions; by choosing the correct scaling, it might be possible to use the 2-D Gaussian Free Field characterization in~\cite{RAY} to evaluate our scaling limit explicitly in the dispersive phase.

%%%%%%%%%%%%%%%%%%%%%%%%%%%%%%%%%%%%%%%%%%%%%%%%%%%%%%%%%%%%%%%%%%%%%%
\mbox{ }\vskip.01in
\noindent{\bf Acknowledgments.} We would like to thank Gayana Jayasinghe, Gourab Ray, and Arnab Sen for many useful discussions. 
\bibliography{nls}

\begin{bibdiv}
\begin{biblist}

\bib{A19}{article}{
      author={Ab\"{a}cherli, Angelo},
       title={Local picture and level-set percolation of the {G}aussian free
  field on a large discrete torus},
        date={2019},
        ISSN={0304-4149},
     journal={Stochastic Process. Appl.},
      volume={129},
      number={9},
       pages={3527\ndash 3546},
         url={https://doi.org/10.1016/j.spa.2018.09.017},
      review={\MR{3985572}},
}

\bib{RAY}{article}{
      author={Berestycki, Nathana\"{e}l},
      author={Powell, Ellen},
      author={Ray, Gourab},
       title={A characterisation of the {G}aussian free field},
        date={2020},
        ISSN={0178-8051},
     journal={Probab. Theory Related Fields},
      volume={176},
      number={3-4},
       pages={1259\ndash 1301},
         url={https://doi.org/10.1007/s00440-019-00939-9},
      review={\MR{4087493}},
}

\bib{B94}{article}{
      author={Bourgain, J.},
       title={Periodic nonlinear {S}chr\"{o}dinger equation and invariant
  measures},
        date={1994},
        ISSN={0010-3616},
     journal={Comm. Math. Phys.},
      volume={166},
      number={1},
       pages={1\ndash 26},
         url={http://projecteuclid.org/euclid.cmp/1104271501},
      review={\MR{1309539}},
}

\bib{BK04}{article}{
      author={Brazhnyi, VA},
      author={Konotop, VV},
       title={Theory of nonlinear matter waves in optical lattices},
        date={2004},
     journal={Modern Physics Letters B},
      volume={18},
      number={14},
       pages={627\ndash 651},
}

\bib{BS96}{article}{
      author={Brydges, David~C.},
      author={Slade, Gordon},
       title={Statistical mechanics of the {$2$}-dimensional focusing nonlinear
  {S}chr\"{o}dinger equation},
        date={1996},
        ISSN={0010-3616},
     journal={Comm. Math. Phys.},
      volume={182},
      number={2},
       pages={485\ndash 504},
         url={http://projecteuclid.org/euclid.cmp/1104288157},
      review={\MR{1447302}},
}

\bib{C14}{article}{
      author={Chatterjee, Sourav},
       title={Invariant measures and the soliton resolution conjecture},
        date={2014},
        ISSN={0010-3640},
     journal={Comm. Pure Appl. Math.},
      volume={67},
      number={11},
       pages={1737\ndash 1842},
         url={https://doi.org/10.1002/cpa.21501},
      review={\MR{3263670}},
}

\bib{CK12}{article}{
      author={Chatterjee, Sourav},
      author={Kirkpatrick, Kay},
       title={Probabilistic methods for discrete nonlinear {S}chr\"{o}dinger
  equations},
        date={2012},
        ISSN={0010-3640},
     journal={Comm. Pure Appl. Math.},
      volume={65},
      number={5},
       pages={727\ndash 757},
         url={https://doi.org/10.1002/cpa.21388},
      review={\MR{2898889}},
}

\bib{CLS03}{article}{
      author={Christodoulides, Demetrios~N},
      author={Lederer, Falk},
      author={Silberberg, Yaron},
       title={Discretizing light behaviour in linear and nonlinear waveguide
  lattices},
        date={2003},
     journal={Nature},
      volume={424},
      number={6950},
       pages={817\ndash 823},
}

\bib{DK21}{article}{
      author={Dey, Partha},
      author={Kim, Daesung},
       title={Fluctuation results for size of the vacant set for random walks
  on discrete torus},
        date={2021-08},
     journal={arXiv e-prints},
       pages={arXiv:2108.06450},
      eprint={2108.06450},
}

\bib{OL}{misc}{
      author={Hannani, Amirali},
      author={Olla, Stefano},
       title={A stochastic thermalization of the discrete nonlinear
  schrödinger equation},
   publisher={arXiv},
        date={2021},
         url={https://arxiv.org/abs/2109.01389},
}

\bib{WEI3}{article}{
      author={Jenkinson, M.},
      author={Weinstein, M.~I.},
       title={Onsite and offsite bound states of the discrete nonlinear
  {S}chr\"{o}dinger equation and the {P}eierls-{N}abarro barrier},
        date={2016},
        ISSN={0951-7715},
     journal={Nonlinearity},
      volume={29},
      number={1},
       pages={27\ndash 86},
         url={https://doi.org/10.1088/0951-7715/29/1/27},
      review={\MR{3460749}},
}

\bib{KMW}{article}{
      author={Kirkpatrick, Kay},
      author={Mirasola, Anthony~E.},
      author={Prescod-Weinstein, Chanda},
       title={Analysis of bose-einstein condensation times for self-interacting
  scalar dark matter},
        date={2022Aug},
     journal={Phys. Rev. D},
      volume={106},
       pages={043512},
         url={https://link.aps.org/doi/10.1103/PhysRevD.106.043512},
}

\bib{LP68}{article}{
      author={Lax, Peter~D.},
       title={Integrals of nonlinear equations of evolution and solitary
  waves},
        date={1968},
        ISSN={0010-3640},
     journal={Comm. Pure Appl. Math.},
      volume={21},
       pages={467\ndash 490},
         url={https://doi.org/10.1002/cpa.3160210503},
      review={\MR{235310}},
}

\bib{LEB}{article}{
      author={Lebowitz, J.~L.},
      author={Mounaix, Ph.},
      author={Wang, W.-M.},
       title={Approach to equilibrium for the stochastic {NLS}},
        date={2013},
        ISSN={0010-3616},
     journal={Comm. Math. Phys.},
      volume={321},
      number={1},
       pages={69\ndash 84},
         url={https://doi.org/10.1007/s00220-012-1632-7},
      review={\MR{3089664}},
}

\bib{LRS88}{article}{
      author={Lebowitz, Joel~L.},
      author={Rose, Harvey~A.},
      author={Speer, Eugene~R.},
       title={Statistical mechanics of the nonlinear {S}chr\"{o}dinger
  equation},
        date={1988},
        ISSN={0022-4715},
     journal={J. Statist. Phys.},
      volume={50},
      number={3-4},
       pages={657\ndash 687},
         url={https://doi.org/10.1007/BF01026495},
      review={\MR{939505}},
}

\bib{MV94}{article}{
      author={McKean, H.~P.},
      author={Vaninsky, K.~L.},
       title={Brownian motion with restoring drift: the petit and
  micro-canonical ensembles},
        date={1994},
        ISSN={0010-3616},
     journal={Comm. Math. Phys.},
      volume={160},
      number={3},
       pages={615\ndash 630},
         url={http://projecteuclid.org/euclid.cmp/1104269714},
      review={\MR{1266067}},
}

\bib{MV97b}{article}{
      author={McKean, H.~P.},
      author={Vaninsky, K.~L.},
       title={Action-angle variables for the cubic {S}chr\"{o}dinger equation},
        date={1997},
        ISSN={0010-3640},
     journal={Comm. Pure Appl. Math.},
      volume={50},
      number={6},
       pages={489\ndash 562},
  url={https://doi.org/10.1002/(SICI)1097-0312(199706)50:6<489::AID-CPA1>3.0.CO;2-4},
      review={\MR{1441912}},
}

\bib{MV97a}{article}{
      author={McKean, H.~P.},
      author={Vaninsky, K.~L.},
       title={Cubic {S}chr\"{o}dinger: the petit canonical ensemble in
  action-angle variables},
        date={1997},
        ISSN={0010-3640},
     journal={Comm. Pure Appl. Math.},
      volume={50},
      number={7},
       pages={593\ndash 622},
  url={https://doi.org/10.1002/(SICI)1097-0312(199707)50:7<593::AID-CPA1>3.3.CO;2-A},
      review={\MR{1447055}},
}

\bib{P04}{article}{
      author={Peyrard, Michel},
       title={Nonlinear dynamics and statistical physics of {DNA}},
        date={2004},
     journal={Nonlinearity},
      volume={17},
      number={2},
       pages={R1\ndash R40},
         url={https://doi.org/10.1088/0951-7715/17/2/r01},
}

\bib{S07}{article}{
      author={Sheffield, Scott},
       title={Gaussian free fields for mathematicians},
        date={2007},
        ISSN={0178-8051},
     journal={Probab. Theory Related Fields},
      volume={139},
      number={3-4},
       pages={521\ndash 541},
         url={https://doi.org/10.1007/s00440-006-0050-1},
      review={\MR{2322706}},
}

\bib{WEI99}{article}{
      author={Weinstein, M.~I.},
       title={Excitation thresholds for nonlinear localized modes on lattices},
        date={1999},
        ISSN={0951-7715},
     journal={Nonlinearity},
      volume={12},
      number={3},
       pages={673\ndash 691},
         url={https://doi.org/10.1088/0951-7715/12/3/314},
      review={\MR{1690199}},
}

\bib{WEI82}{article}{
      author={Weinstein, Michael~I.},
       title={Nonlinear {S}chr\"{o}dinger equations and sharp interpolation
  estimates},
        date={1982/83},
        ISSN={0010-3616},
     journal={Comm. Math. Phys.},
      volume={87},
      number={4},
       pages={567\ndash 576},
         url={http://projecteuclid.org/euclid.cmp/1103922134},
      review={\MR{691044}},
}

\end{biblist}
\end{bibdiv}

\end{document}